\documentclass[a4paper,11pt,reqno]{amsart}
\usepackage{amsmath}
\usepackage{amstext}
\usepackage{amsbsy}
\usepackage{amsopn}
\usepackage{upref}
\usepackage{amsthm}
\usepackage{amsfonts}
\usepackage{amssymb}
\usepackage{mathrsfs}
\usepackage{mathtools}
\usepackage{times}
\usepackage{cite}
\usepackage{bm}
\usepackage{graphicx}
\usepackage{multicol}
\usepackage{color}
\allowdisplaybreaks
\numberwithin{equation}{section}
 \newtheorem{theorem}{Theorem}[section]
 \newtheorem{corollary}[theorem]{Corollary}
 
 \newtheorem{lemma}[theorem]{Lemma}
\theoremstyle{definition}
 \newtheorem{definition}[theorem]{Definition}
\theoremstyle{remark}
 
\begin{document}
%%%%%%
%%%%%% Make Title
%%%%%%
\title[Geodesic X-ray transform]{Geodesic X-ray transform and streaking artifacts \\ on simple surfaces or on spaces of constant curvature}
\author[H.~Chihara]{Hiroyuki Chihara}
\address{College of Education, University of the Ryukyus, Nishihara, Okinawa 903-0213, Japan}
\email{hc@trevally.net}
\thanks{Supported by the JSPS Grant-in-Aid for Scientific Research \#23K03186.}
\begin{abstract}
The X-ray transform on the plane or on the three-dimensional Euclidean space can be considered as the measurements of CT scanners for normal human tissue. If the human body contains metal regions such as dental implants, stents in blood vessels, metal bones, etc., the beam hardening effect for the energy level of the X-ray causes streaking artifacts in its CT image. More precisely, if there are two strictly convex metal regions contained in the cross-section of normal human tissue, then streaking artifacts occur along the common tangent lines of the two regions.

In this paper we study the geodesic X-ray transform and streaking artifacts on simple Riemannian manifolds. We show that the streaking artifacts result from the propagation of conormal singularities on the boundary of metal regions along the common tangent geodesics under the strong and seemingly strange assumption that the manifolds are two dimensional or spaces of constant curvature. This condition ensures that every Jacobi field takes the form of the product of a scalar function and parallel transport along the geodesic. Our results clarify the geometric meaning of the theory, which was imperceptible in the known results on the Euclidean space. 
\end{abstract}
\maketitle
%%%%%%
%%%%%% section 1
%%%%%%
\section{Introduction}
\label{section:introduction}
The present paper studies the geodesic X-ray transform and a nonlinear phenomenon called the streaking artifact on some compact Riemannian manifolds with boundaries. We begin with the X-ray transform on the Euclidean plane $\mathbb{R}^2$ to state the background and the motivation of the present paper. Every planar line $\ell$ is parameterized by $(\theta,t)\in[0,\pi)\times\mathbb{R}$ as %
$$
\ell
=
\{(-s\sin\theta+t\cos\theta,s\cos\theta+t\sin\theta) : s\in\mathbb{R}\}, 
$$   
and the X-ray transform of a function $f(x,y)$ on $\mathbb{R}^2$ is defined by 
$$
\mathcal{R}f(\theta,t)
:=
\int_\ell f
=
\int_{-\infty}^\infty
f(-s\sin\theta+t\cos\theta,s\cos\theta+t\sin\theta)
ds.
$$
This is also defined for $\theta\in[\pi,2\pi)$, and 
$\mathcal{R}f(\theta,t)=\mathcal{R}f(\theta+\pi,-t)$,  
$(\theta,t)\in[0,\pi)\times\mathbb{R}$ holds. 
The adjoint $\mathcal{R}^T$ of $\mathcal{R}$ is explicitly given by 
$$
\mathcal{R}^Tg(x,y)
=
\frac{1}{2\pi}
\int_0^\pi
g(\theta,x\cos\theta+y\sin\theta)
d\theta,
\quad
(x,y)\in\mathbb{R}^2
$$
for a function $g(\theta,t)$ of $(\theta,t)\in[0,\pi)\times\mathbb{R}$. If $f(x,y)$ describes the density distribution of the attenuation coefficient on a cross-section of a human body of normal tissue, then the set $\{\mathcal{R}f(\theta,t) : (\theta,t)\in[0,\pi)\times\mathbb{R}\}$ is the complete data of the measurements of a CT scanner of old generation. Furthermore it is well-known that $f(x,y)$ can be reconstructed theoretically from the measurements by the inversion formula 
$f=(-\partial_x^2-\partial_y^2)^{1/2}\circ\mathcal{R}^T\circ\mathcal{R}f$. 
See \cite{Helgason} for instance. 
Here we assumed from a point of view of physics that the X-ray beams have no width, are monochromatic, and traverse the objects along straight lines.  
\par
Actually, however, there are many factors affecting CT images: 
beam width, 
partial volume effect, 
beam hardening effect, 
noise in measurements, 
numerical errors, 
and etc. 
They cause the artifacts in CT image. 
See Epstein's textbook \cite{Epstein} on the mathematical theory of medical imaging for the detail. 
The X-ray beams consist of photons which actually have a wide range of energies $E\geqq0$. This is described by a spectral function, which is a probability density function $\rho(E)$ of $E\in[0,\infty)$, and pick up the strong part of the measurement. This is called the beam hardening in medical imaging. In this case the measurements of CT scanners are corrected from 
$\mathcal{R}f_E$ to 
$$
P=-\operatorname{log}\left\{\int_0^\infty\rho(E)\exp(-\mathcal{R}f_E)dE\right\}. 
$$
If $f_E$ is independent of $E\geqq0$, say, $f_E=f_{E_0}$ with some fixed $E_0>0$, 
then $P=\mathcal{R}f_{E_0}$. See \cite[Section~3.3]{Epstein} again for the detail of the above. 
It is known that the beam hardening effect causes streaking artifacts in the CT images. 
\par
The mathematical analysis of the streaking artifacts was originated 
by Park, Choi and Seo in their pioneering work \cite{ParkChoiSeo}. 
They studied the case that there is a metal region $D$, 
which is a disjoint union of finite number of strictly convex bounded domains 
$D_1,\dotsc,D_J$ with smooth boundaries. 
They assumed that the spectral function is a characteristic function of an interval. 
Let $E_0>0$ be the central energy level of the beam, 
and let $\varepsilon$ be a small positive constant smaller than $E_0$. 
Their spectral function $\rho(E)$ and the distribution $f_E(x,y)$ of attenuation coefficients 
are supposed to be of the forms   
\begin{align*}
  \rho(E)
& =
  \frac{1}{2\varepsilon}1_{[E_0-\varepsilon,E_0+\varepsilon]}(E),
\\
  f_E(x,y)
& =
  f_{E_0}(x,y)
  +
  \alpha(E-E_0)1_D(x,y),
\end{align*} 
where $1_A$ is the characteristic function of a set $A$, 
$f_{E_0}$ is the distribution of attenuation coefficients for the normal human tissue, 
and $\alpha$ is a positive constant. In this case the measurements become 
$$
P
=
\mathcal{R}f_{E_0}
-
\log\left\{\frac{\sinh\bigl(\alpha\varepsilon\mathcal{R}1_D\bigr)}{\alpha\varepsilon\mathcal{R}1_D}\right\}
=
\mathcal{R}f_{E_0}
+
\sum_{l=1}^\infty
B_l\bigl(\alpha\varepsilon\mathcal{R}1_D\bigr)^{2l}, 
$$
where $\{B_l\}_{l=1}^\infty$ is a sequence of real numbers. 
Then $\bigl(\mathcal{R}1_D\bigr)^2$ and its filtered back-projection 
$(-\partial_x^2-\partial_y^2)^{1/2}\circ\mathcal{R}^T\bigl[\bigl(\mathcal{R}1_D\bigr)^2\bigr]$ 
are considered to be the principal parts of the beam hardening effect and of the streaking artifacts respectively. 
We here show figures of an original grayscale image of a simple model of two dental implants in the cross-section of oral cavity, 
its X-ray transform, the standard filtered back-projection, and the filtered back-projection of $\bigl(\mathcal{R}1_D\bigr)^2$ using the Julia Programming Language. 
\vspace{11pt}
\\
\begin{center}
\includegraphics[width=90mm]{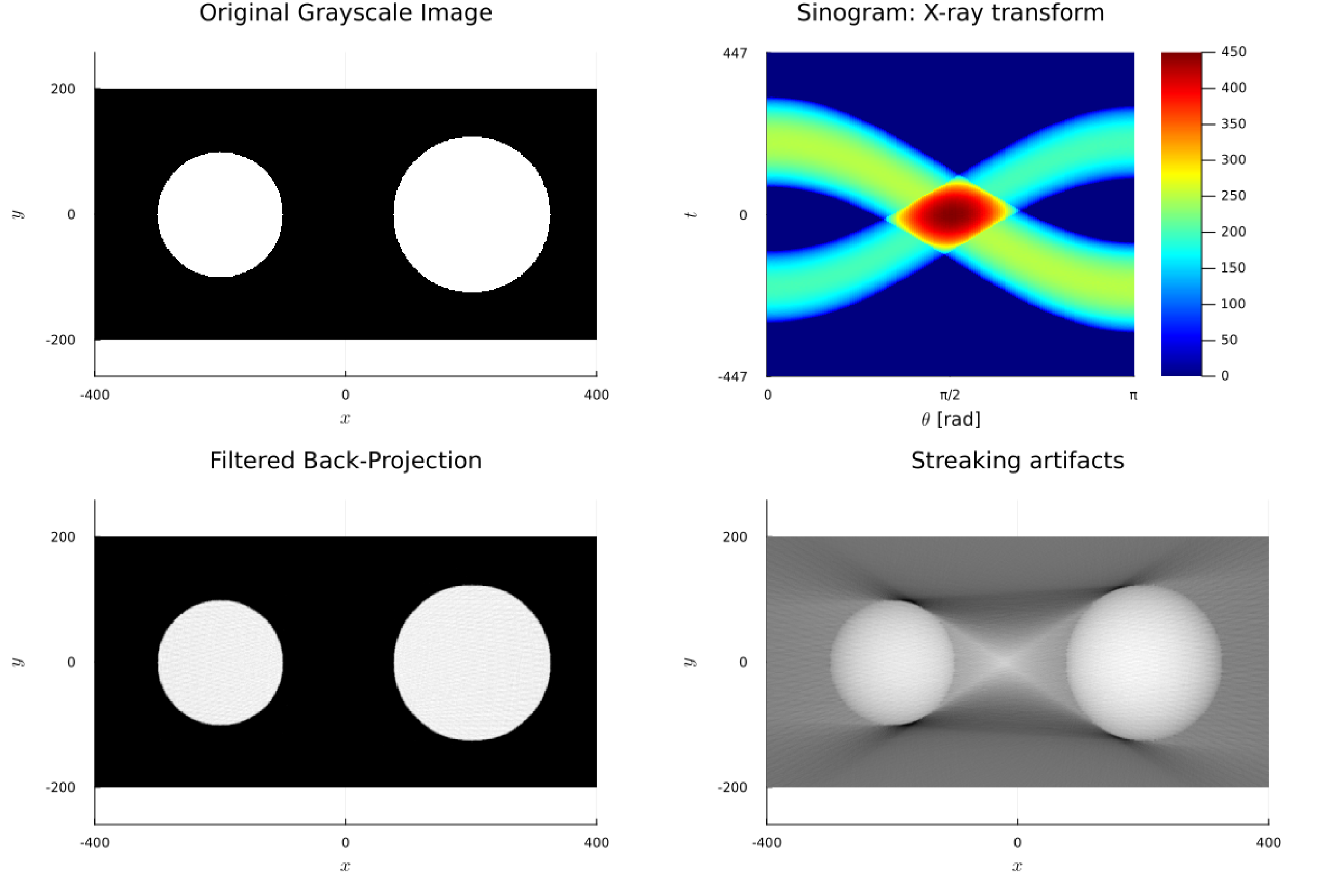}
\\
Figure~1.\ 
The upper left is an original grayscale image of a simple model of two dental implants in the cross-section of oral cavity, 
the upper right is its X-ray transform, 
the lower left is the standard filtered back-projection, 
and the lower right is the filtered back-projection of $\bigl(\mathcal{R}1_D\bigr)^2$.  
\end{center}
\vspace{11pt}
We can see streaking artifacts along four straight lines which are tangent to both implants. Set 
\begin{align*}
  f_\text{CT}
& :=
  (-\partial_x^2-\partial_y^2)^{1/2}\circ\mathcal{R}^TP,
\\
  f_\text{MA}
& :=
  (-\partial_x^2-\partial_y^2)^{1/2}\circ\mathcal{R}^T\bigl[P-\mathcal{R}f_{E_0}\bigr]
  =
  f_\text{CT}-f_{E_0}. 
\end{align*}
In \cite{ParkChoiSeo} Park, Choi and Seo studied the wave front sets of $P-\mathcal{R}f_{E_0}$ 
and $f_\text{MA}$ using the canonical relations of $\mathcal{R}$ and $\mathcal{R}^T$, 
and proved that the streaking artifacts are the singular support of $f_\text{MA}$. 
More precisely, the streaking artifacts are the propagation of conormal singularities 
on the boundary of $D$ along the common tangent lines of $\partial{D_j}$ and $\partial{D_k}$ for $j{\ne}k$.  
\par
In \cite{PalaciosUhlmannWang} Palacios, Uhlmann and Wang developed the arguments of \cite{ParkChoiSeo} 
making use of the modern microlocal analysis, and proved that 
the streaking artifacts are the sum of conormal distributions of the same order, whose singular supports 
are the common tangent lines of $\partial{D_j}$ and $\partial{D_k}$ for $j{\ne}k$. 
\par
In \cite{WangZou} Wang and Zou studied the case that $D$ is a non-convex bounded domain. In this case the relations of lines which are tangent to the non-convex part of $\partial{D}$ are complicated.  
\par
The above mentioned results were studied on the plane $\mathbb{R}^2$. 
More recently in \cite{Chihara} the author developed the arguments in \cite{PalaciosUhlmannWang}, and studied this problem for the $d$-plane transform on the $n$-dimensional Euclidean space $\mathbb{R}^n$, where $n=2,3,4,\dotsc$ and $d=1,\dotsc,n-1$. 
The $d$-plane transform of a function $f(x)$ on $\mathbb{R}^n$ is defined by 
$$
\mathcal{R}_df(\sigma,x^{\prime\prime})
:=
\int_\sigma f(x^\prime+x^{\prime\prime}) dx^\prime,
$$
where $\sigma \in G_{d,n}$, 
$G_{d,n}$ is the Grassmannian which is the set of all $d$-dimensional vector subspaces of $\mathbb{R}^n$, 
$x^\prime\in\sigma$, $x^{\prime\prime}\in\sigma^\perp$, 
and $\sigma^\perp$ is the orthogonal complement of $\sigma$ in $\mathbb{R}^n$.  
We remark that if $n\geqq3$, then there are two types of common tangent lines of $\partial{D}_j$ and $\partial{D}_k$ for $j{\ne}k$.
\begin{itemize}
\item 
Case~B1:  
The common tangent line is contained in a common tangent hyperplane. 
In other words, the directions of the conormal singularities 
at the tangent points are same. 
\item 
Case~B2: 
The tangent hyperplanes at the tangent points are different. 
In other words, the directions of the conormal singularities 
at the tangent points are different.  
\end{itemize}
We show figures illustrating Cases~B1 and B2 below. 
\vspace{2pt}
\begin{center}
\includegraphics[width=150mm]{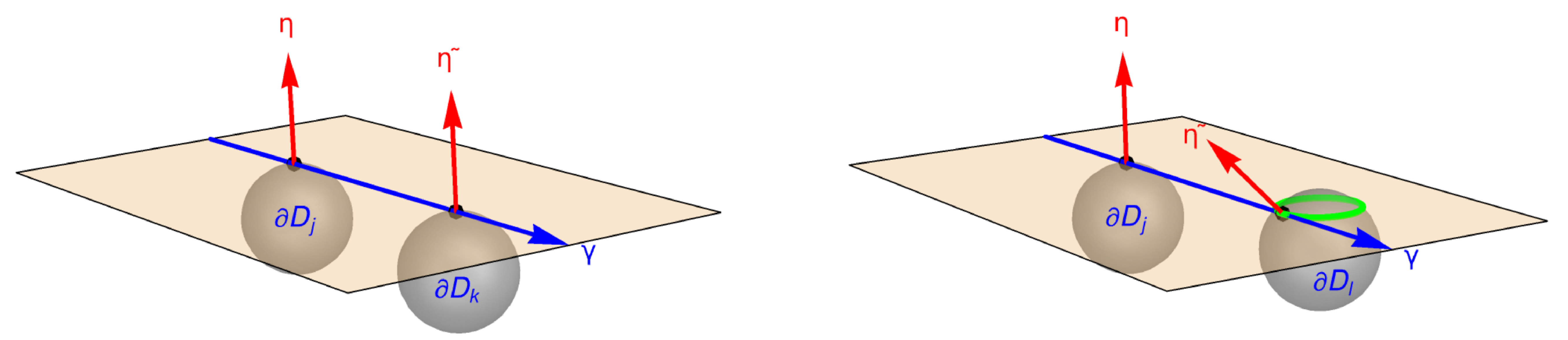}
\\
Figure~2. 
The left is Case~B1 where $\partial{D_j}$ and $\partial{D_k}$ have a common tangent hyperplane, 
and the right is Case~B2 where the tangent hyperplane of $\partial{D_j}$ cuts through the interior of $D_l$. 
\end{center}
\vspace{11pt}
The author generalized the results of \cite{PalaciosUhlmannWang} 
for the $d$-plane transform on $\mathbb{R}^n$. 
More precisely, we proved the following. 
\begin{itemize}
\item 
The conormal singularities at the common tangent points of 
$\partial{D}_j$ and $\partial{D}_k$ of Case~B1 
propagates along the common tangent line. 
The set of all the common tangent lines of Case~B1 for 
$\partial{D}_j$ and $\partial{D}_k$ form two hypersurfaces $\mathcal{L}_{jk}^{(\pm)}$, 
which are cylindrical or conic. 
These hypersurfaces are surrounding $\partial{D}_j$ and $\partial{D}_k$.  
\begin{center}
\includegraphics[width=70mm]{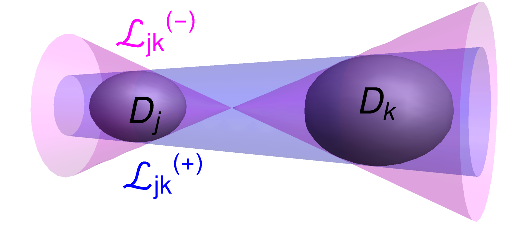}
\\
Figure~3.\ 
Common surrounding hypersurfaces for $D_j$ and $D_k$. 
\end{center}
The filtered back-projection $f_\text{MA}$ is a sum of conormal distributions of the same order determined by $n$ and $d$. 
The singular supports of these conormal distributions are such hypersurfaces. 
\item 
The conormal singularities at the common tangent points of 
$\partial{D}_j$ and $\partial{D}_k$ of Case~B2 do not propagate. 
\end{itemize}
What the author could understand in \cite{Chihara} is that these results are the natural extension of those of \cite{PalaciosUhlmannWang}. 
Unfortunately, however, the author could not understand almost all the geometric meanings of the facts proved in \cite{Chihara}. It was too difficult for the author to see the essential parts in the setting of the Euclidean space. 
\par
The purpose of this paper is to clarify the geometric meaning of the facts proved in \cite{Chihara}. For this purpose we study the geodesic X-ray transform on a compact Riemannian manifold $(M,g)$ with a Riemannian metric $g$ and the smooth boundary $\partial{M}$. 
Set $\langle{u,v}\rangle:=g_p(u,v)$ for $u,v \in T_p(M)$. 
The geodesic X-ray transform for functions and symmetric tensor fields 
has been extensively studied so far. 
See e.g.,\cite{FrigyikStefanovUhlmann,MonardStefanov,PaternainSaloUhlmann1,Quinto1,Quinto2,StefanovUhlmannVasy} and references therein. 
\par
The main results of this paper are stated in Theorem~\ref{theorem:51}, 
which are essentially same as those of \cite{Chihara}. 
In the present paper we assume on $(M,g)$ as follows. 
\begin{description}
\item[(A0)] 
$(M,g)$ is a simple Riemannian manifold, that is, 
\begin{itemize}
\item 
$(M,g)$ is a compact connected Riemannian manifold with smooth boundary. 
\item 
$(M,g)$ is non-trapping, that is, every geodesic exits $M^\text{int}$ in finite time. 
\item 
The boundary $\partial{M}$ is strictly convex.
\item 
There are no conjugate points.  
\end{itemize}
See \cite[Definition~3.8.1]{PaternainSaloUhlmann2}. 
Moreover, there are some equivalent conditions for the simplicity. 
Some of them seemingly look stronger than the above. 
See \cite[Theorem~3.8.2]{PaternainSaloUhlmann2} for the detail. 
\item[(A1)] 
The Riemannian manifold $(M,g)$ is two-dimensional or a space of constant curvature. 
\end{description}
The assumptions on the metal region $D$ and on the beam hardening effect are 
essentially same as those of \cite{ParkChoiSeo}, but we will state them later. 
Let $\mathcal{G}$ be the set of all the normal geodesics of $(M,g)$, that is, 
the set of all the unit speed geodesics of $(M,g)$. 
The X-ray transform of a compactly supported smooth function $f$ on 
$M^\text{int}$ is defined by 
$$
\mathcal{X}f(\gamma)
:=
\int_\gamma f,
\quad
\gamma \in \mathscr{G}. 
$$
Actually, in this paper we deal with the X-ray transform of 
compactly supported half-densities on $M^\text{int}$. 
We describe the detail later. 
The assumption (A1) is very strong and may seem to be strange. 
This is effectively used for composing the canonical relation 
$\mathcal{C}_{\mathcal{X}^T}$ of $\mathcal{X}^T$ 
and the Lagrangian describing the microlocal singularities of the conormal distribution $P-\mathcal{X}f$. 
\par
This paper relies on the foundations of microlocal analysis and Fourier integral operators. 
See celebrated standard textbooks \cite{Duistermaat,GrigisSjoestrand,Hoermander1,Hoermander2,Hoermander3,Hoermander4} for this. One can also refer expository papers \cite{KrishnanQuinto} and \cite[Section~2]{WebberHolmanQuinto} 
for microlocal analysis for computer tomography and related fields. 
Also this paper deeply relies on the methods and the ideas dealing with the geodesic X-ray transform developed by Holman and Uhlmann in \cite{HolmanUhlmann}. 
\par
The plan of this paper is as follows. 
Section~2 consists of basic facts on the X-ray transform $\mathcal{X}$ and its adjoint $\mathcal{X}^T$. 
Section~3 is devoted to studying Jacobi fields and parallel transport along normal geodesics under the assumption ({\bf A1}). 
They play crucial role in clarifying the geometric meanings of the facts in this paper. 
In Section~4 we discuss the geometry of the metal region $D$ for the intersection calculus for $\mathcal{X}$. 
In Section~5 we state the main theorem of the paper, and discuss the intersection calculus for $\mathcal{X}^T$. 
Finally in Section~6 we prove the main theorem in several steps. 
%%%%%%
%%%%%% section 2
%%%%%%
\section{Preliminaries}
\label{section:preliminaries} 
This section deeply relies on the results and the ideas developed by Holman and Uhlmann in \cite{HolmanUhlmann}. 
In this section we assume ({\bf A0}) and not ({\bf A1}). We prepare some basic facts on the geodesic X-ray transform on $(M,g)$. Throughout this paper we use the notation which is the same as or similar to that of \cite{HolmanUhlmann}. We denote by $S(M)$ the unit tangent sphere bundle, that is, the principal bundle on $M$ consisting of all the unit tangent vectors, and $\pi:S(M){\rightarrow}M$ denotes the natural projection. This is a $(2n-1)$-dimensional manifold. The restriction of $\pi$ on $M^\text{int}$ is denoted by the same notation $\pi$. 
Here we introduce the unit outer normal vector $\nu_+(x)$ at $x \in \partial{M}$, and set 
$$
\partial_{\pm}S(M)
:=
\bigl\{w \in S(M) : \pi(w)\in\partial{M}, \pm\bigl\langle{w},\nu_+\bigl(\pi(w)\bigr)\bigr\rangle>0\}.
$$
Then $\partial_-S(M)$ is the set of all incoming unit tangent vectors on $\partial{M}$.  
For $v \in S(M^\text{int}){\cup}\partial_-S(M)$, $\gamma_v$ denotes the normal geodesic such that $\gamma_v(0)=\pi(v)$ and $\dot{\gamma}_v(0)=v$. 
Note that if $v \in S(M^\text{int})$, there exists a constant $\varepsilon>0$ such that 
$\gamma_v(t) \in M$ for $t\in(-\varepsilon,\varepsilon)$, and if $w \in \partial_-S(M)$, there exists a constant $\varepsilon>0$ such that $\gamma_w(t) \in M$ for $t\in[0,\varepsilon)$. 
If we set 
$$
\tau_-(v)
=
\sup\{t\leqq0 : \gamma_v(t) \not\in M^\text{int}\}, 
\quad
\tau_+(v)
=
\inf\{t\geqq0 : \gamma_v(t) \not\in M^\text{int}\}, 
$$
then the non-trapping condition on $(M,g)$ implies that $[\tau_-(v),\tau_+(v)]$ is a closed finite interval. 
We also denote by $\Psi$ the geodesic flow on $S(M^\text{int})$, that is, 
$\pi\bigl(\Psi(v,t)\bigr)=\gamma_v(t)$ and $\Psi(v,t)=\dot{\gamma}_v(t)$ 
for $v \in S(M^\text{int})$ and $t\in[\tau_-(v),\tau_+(v)]$. 
The simplicity condition on $(M,g)$ implies that for any $w\in\partial_-S(M)$ there exists a $w^\prime\in\partial_+S(M)$ uniquely such that $w^\prime=\dot{\gamma}_w\bigl(\tau_+(w)\bigr)$. 
Let $\mathcal{G}$ be the set of all normal geodesics on $M$. 
In the same way, for any $\gamma \in \mathcal{G}$ there exists a $w \in \partial_-S(M)$ uniquely such that $\gamma=\gamma_w(\cdot)$, and $\partial_-S(M)$ is identified with $\mathcal{G}$. More precisely $\mathcal{G}$ becomes a $(2n-2)$-dimensional smooth manifold with the parametrization $\partial_-S(M)$. 
\par
We now introduce an important submersion $F:S(M^\text{int}) \rightarrow \partial_-S(M)$. For $v \in S(M^\text{int})$ set $F(v):=\dot{\gamma}_v\bigl(\tau_-(v)\bigr)$. Then we have $v \mapsto \tau_-(v)$ is smooth and $F$ becomes a smooth mapping since $\partial{M}$ is a strictly convex smooth boundary. 
If $F(v)=F(\tilde{v})$ for $v,\tilde{v} \in S(M^\text{int})$, then there exists a time $t_0\in\bigl(\tau_-(v),\tau_+(v)\bigr)$ such that $\tilde{v}=\dot{\gamma}_v(t_0)$. Hence we have 
$$
\operatorname{rank}(DF\vert_v)
=
\operatorname{dim}\bigl(S(M^\text{int})\bigr)-1
=
2n-2
=
\operatorname{dim}\bigl(\partial_-S(M)\bigr), 
$$ 
and we conclude that $F$ is a submersion. 
\par
We shall define the geodesic X-ray transform $\mathcal{X}$ for half densities. 
To obtain its canonical relation we see that 
$\mathcal{X}=F_\ast{\circ}\pi^\ast$ and $\mathcal{X}^T=\pi_\ast{\circ}F^\ast$, 
where $\mathcal{X}^T$ is the adjoint of $\mathcal{X}$, 
$F_\ast$ and $F^\ast$ are the pushforward and the pullback induced by $F$ respectively, 
and $\pi_\ast$ and $\pi^\ast$ are the pushforward and the pullback induced by $\pi$ respectively. 
We denote by $dv_g$ the Riemannian measure of $(M,g)$, 
by $\lvert{dS(M)}\rvert$ the Riemannian measure of $S(M)$ with respect to the Sasaki metric, 
and by $\lvert{d\partial{S(M)}}\rvert$  the Riemannian measure of 
$$
{\partial}S(M):=S(M)\lvert_{\partial{M}}=\partial_-S(M) \cup S(\partial{M}) \cup \partial_+S(M)
$$
with respect to the Sasaki metric. 
Santal\'o's formula takes the form   
\begin{equation}
\int_{S(M)}h(v)\lvert{dS(M)}\rvert
=
-
\int_{\partial_-S(M)}
\int_0^{\tau_+(w)}
h\bigl(\dot{\gamma}_w(s)\bigr)
\bigl\langle{w,\nu_+\bigl(\pi(w)\bigr)}\bigr\rangle
ds
\lvert{d\partial{S(M)}}\rvert
\label{equation:santalo}
\end{equation}
for any continuous function $h$ on $S(M)$. 
See \cite[Theorem~8.4.1]{Sharafutdinov} for instance. 
So we set 
$$
\lvert{d\mu(w)}\rvert:=-\bigl\langle{w,\nu_+\bigl(\pi(w)\bigr)}\bigr\rangle\lvert{d\partial{S(M)}}\rvert,
$$
and use it as the standard measure on $\partial_-S(M)$. 
\par
Let $\Omega_X^{1/2}$ be the half density bundle on a manifold $X$. See, e.g., \cite[page 92]{Hoermander3} for the definition of the powers of density. We denote by $C^\infty(\Omega_X^{1/2})$ the set of all smooth half densities on $X$, and by $C^\infty_0(\Omega_X^{1/2})$ the set of all compactly supported smooth half densities on $X$. Let $\mathscr{D}(\Omega_X^{1/2})$ be the locally convex space which is $C^\infty_0(\Omega_X^{1/2})$ with the standard inductive limit topology, and let $\mathscr{D}^\prime(\Omega_X^{1/2})$ be the topological dual of $\mathscr{D}(\Omega_X^{1/2})$, that is, the space of distribution sections of half density bundle on $X$. Let $\mathscr{E}(\Omega_X^{1/2})$ be the Fr\'echet space which is $C^\infty(\Omega_X^{1/2})$ with the standard topology, and let $\mathscr{E}^\prime(\Omega_X^{1/2})$ be the topological dual of $\mathscr{E}(\Omega_X^{1/2})$, that is, the space of compactly supported distribution sections of the half density bundle on $X$. We now define the geodesic X-ray transform of a compactly supported smooth half density. 
\begin{definition}[{\bf Geodesic X-ray transform of a compactly supported half density}] 
\label{theorem:geodesicxraytransform}
The geodesic X-ray transform on $(M,g)$ is defined for $f \in C^\infty_0(\Omega_{M^\text{int}}^{1/2})$ as a half density on $\partial_-S(M)$ by 
\begin{equation}
\mathcal{X}[f](w)
:=
\left(
\int_0^{\tau_+(w)}
\frac{f}{\lvert{dv_g}\rvert^{1/2}}
\bigl(\dot{\gamma}_w(s)\bigr)
ds
\right)
\lvert{d\mu(w)}\rvert^{1/2},
\quad
w \in \partial_-S(M). 
\label{equation:geodesicxraytransform} 
\end{equation}
\end{definition} 
It is easy to see that 
$\mathcal{X}: \mathscr{D}(\Omega_{M^\text{int}}^{1/2}) \rightarrow \mathscr{D}(\Omega_{\partial_-S(M)}^{1/2})$ 
is a continuous mapping. Furthermore $\mathcal{X}$ can be extended on 
$\mathscr{E}^\prime(\Omega_{M^\text{int}}^{1/2})$ since $\mathcal{X}$ is a Fourier integral operator. See \cite[Sections~25.1 and 25.2]{Hoermander4} for the foundation of Fourier integral operators. 
\par
We now introduce a pushforward and a pullback of half densities induced by a submersion. 
\begin{definition}[{\bf Pushforward and pullback induced by a submersion}]
\label{theorem:pushforwardpullback1}
Let $X$ be an $n_X$-dimensional smooth manifold equipped with a density $\lvert{d\mu_X}\rvert$, and let $Y$ be an $n_Y$-dimensional smooth manifold equipped with a density $\lvert{d\mu_Y}\rvert$. 
Suppose that $G: X \rightarrow Y$ is a submersion. 
Then the pushforward $G_\ast: C^\infty_0(\Omega_X^{1/2}) \rightarrow C^\infty_0(\Omega_Y^{1/2})$ and 
the pullback $G^\ast: C^\infty_0(\Omega_Y^{1/2}) \rightarrow C^\infty_0(\Omega_X^{1/2})$ 
are defined by the requirement that 
\begin{equation}
\int_Y h(y){\cdot}G_\ast[f](y)
=
\int_X \frac{h}{\lvert{d\mu_Y}\rvert^{1/2}}\bigl(G(x)\bigr)\lvert{d\mu_X}\rvert^{1/2}{\cdot}f(x)
=
\int_X G^\ast[h](x){\cdot}f(x)
\label{equation:adjoint}
\end{equation}
for $f \in C^\infty_0(\Omega_X^{1/2})$ and $h \in C^\infty_0(\Omega_Y^{1/2})$. 
\end{definition}
Apply \eqref{equation:santalo} and \eqref{equation:adjoint} to $F$ and $\pi$.  
Then for 
$x \in M^\text{int}$, 
$v \in S(M^\text{int})$, 
$w \in \partial_-S(M)$, 
$f \in C^\infty_0(\Omega_{M^\text{int}}^{1/2})$, 
$h \in C^\infty_0(\Omega_{S(M^\text{int})}^{1/2})$ 
and 
$k \in C^\infty_0(\Omega_{\partial_-S(M)}^{1/2})$, 
we have 
\begin{align*}
  F_\ast[h](w)
& =
  \left(
  \int_0^{\tau_+(w)}
  \frac{h}{\lvert{dS(M^\text{int})}\rvert^{1/2}}\bigl(\dot{\gamma}_w(s)\bigr)
  ds
  \right)
  \lvert{d\mu(w)}\rvert^{1/2}, 
\\
  F^\ast[k](v)
& =
  \frac{k}{\lvert{d\mu}\rvert^{1/2}}\bigl(F(v)\bigr)
  \lvert{dS(M^\text{int})(v)}\rvert^{1/2},
\\
  \pi_\ast[h](x)
& =
  \left(
  \int_{S_x(M^\text{int})}
  \frac{h}{\lvert{dS(M^\text{int})}\rvert^{1/2}}(v)
  \lvert{dS_x(M^\text{int})(v)}\rvert
  \right)
  \lvert{dv_g}(x)\rvert^{1/2},
\\
  \pi^\ast[f](v)
& =
  \frac{f}{\lvert{dv_g}\rvert^{1/2}}\bigl(\pi(v)\bigr)
  \lvert{dS(M^\text{int})(v)}\rvert^{1/2},
\\
  \mathcal{X}[f](w)
& =
  F_\ast{\circ}\pi^\ast[f](w),
\\
  \mathcal{X}^T[k](x)
& =
  \pi_\ast{\circ}F^\ast[k](x)
\\
& =
  \left(
  \int_{S_x(M^\text{int})}
  \frac{k}{\lvert{d\mu}\rvert^{1/2}}\bigl(F(v)\bigr)
  \lvert{dS_x(M^\text{int})(v)}\rvert
  \right)
  \lvert{dv_g}(x)\rvert^{1/2}.
\end{align*}
\par
To state the basic facts on $F_\ast$, $F^\ast$, $\pi_\ast$, and $\pi^\ast$, we now introduce the notation of the space of Lagrangian distributions, and the space of conormal distributions. See \cite[Sections~25.1 and 25.2]{Hoermander4} and \cite[Section~18.2]{Hoermander3} for the precise definitions and detailed information. 
Let $X$ and $Y$ be smooth manifolds, and let $Z$ be a closed submanifold of $X$. 
We denote by $I^m(X,\Lambda;\Omega_X^{1/2})$ the set of all Lagrangian distribution sections of the half density bundle $\Omega^{1/2}_X$ of order $m\in\mathbb{R}$ associated to $\Lambda$ which is a conic Lagrangian submanifold of $T^\ast(X)\setminus0$. In particular, if $\Lambda=N^\ast(Z)\setminus0$, then elements of $I^m(X,N^\ast(Z)\setminus0;\Omega_X^{1/2})$ are said to be conormal distributions and we abbreviate $I^m(X,N^\ast(Z)\setminus0;\Omega_X^{1/2})$ as $I^m(N^\ast(Z)\setminus0;\Omega_X^{1/2})$. 
Let $C$ be a homogeneous canonical relation from $T^\ast(Y)\setminus0$ to $T^\ast(X)\setminus0$, that is, $C$ is a Lagrangian submanifold of $\bigl(T^\ast(X)\setminus0\bigr) \times \bigl(T^\ast(Y)\setminus0\bigr)$ with respect to the symplectic form $\sigma_X-\sigma_Y$, and conic in the fiber variables, where $\sigma_X$ and $\sigma_Y$ are the canonical symplectic forms of $T^\ast(X)$ and $T^\ast(Y)$ respectively. See \cite[Chapter~21]{Hoermander3} for the detail. 

We set 
\begin{align*} 
  C^\prime
& :=
  \{(\xi,-\eta) \in T^\ast(X){\times}T^\ast(Y) : (\xi,\eta) \in C\}, 
\\
  C^T
& :=
  \{(\eta,\xi) \in T^\ast(Y){\times}T^\ast(X) : (\xi,\eta) \in C\}.
\end{align*}
A linear operator whose distribution kernel belongs to $I^m(X{\times}Y,C^\prime;\Omega_{X{\times}Y}^{1/2})$ is said to be a Fourier integral operator of order $m$ from $Y$ to $X$ associated to the canonical relation $C$. 
If the distribution kernel of an operator $A$ belongs to $I^m(X{\times}Y,C^\prime;\Omega_{X{\times}Y}^{1/2})$, we express $A \in I^m(X{\times}Y,C^\prime;\Omega_{X{\times}Y}^{1/2})$ abusing symbols. We believe that any confusion will not occur. 
\par
Linear operators induced by submersions such as Definition~\ref{theorem:pushforwardpullback1} are Fourier integral operators, and it is easy to obtain their canonical relations. This is basically due to the pioneering paper \cite{GuilleminSchaeffer} of Guillemin and Schaeffer. We here quote the fact needed in the present paper from \cite{HolmanUhlmann}. 
\begin{lemma}[{\bf Holman and Uhlmann \cite[Lemma~1]{HolmanUhlmann}}]
\label{theorem:pushforwardpullback2}
In the same setting as in Definition~\ref{theorem:pushforwardpullback1}, we have
\begin{align*}
  G_\ast
& \in 
  I^{(n_Y-n_X)/4}(Y{\times}X,C_{G_\ast}^\prime;\Omega_{Y{\times}X}^{1/2}),
\\
  C_{G_\ast}
& =
  \bigl\{
  (\xi,DG\vert^T_x\xi)
  : 
  x \in X, 
  \xi \in T^\ast_{G(x)}(Y)\setminus\{0\}
  \bigr\},
\\
  G^\ast
& \in 
  I^{(n_Y-n_X)/4}(X{\times}Y,C_{G^\ast}^\prime;\Omega_{X{\times}Y}^{1/2}),
\\
  C_{G^\ast}
& =
  C_{G_\ast}^T.
\end{align*}
\end{lemma}
Since 
$\operatorname{dim}\bigl(S(M^\text{int})\bigr)=2n-1$, 
$\operatorname{dim}\bigl(\partial_-S(M)\bigr)=2n-2$, 
and 
$\operatorname{dim}(M)=n$, Lemma~\ref{theorem:pushforwardpullback2} implies the following. 
\begin{lemma}
\label{theorem:Fpi}
Assume {\rm ({\bf A0})}. 
We have 
\begin{align*}
  F_\ast
& \in 
  I^{-1/4}\bigl(\partial_-S(M){\times}S(M^\text{int}),C_{F_\ast}^\prime,\Omega_{\partial_-S(M){\times}S(M^\text{int})}^{1/2}\bigr),
\\
  F^\ast
& \in 
  I^{-1/4}\bigl(S(M^\text{int}){\times}\partial_-S(M),C_{F^\ast}^\prime,\Omega_{S(M^\text{int})\times\partial_-S(M))}^{1/2}\bigr),
\\
  C_{F_\ast}
& =
  \bigl\{
  (\xi,DF\vert^T_v\xi)
  : 
  v \in S(M^\text{int}), 
  \xi \in T^\ast_{F(v)}(\partial_-S(M))\setminus\{0\}
  \bigr\},
\\
  C_{F^\ast}
& =
  C_{F_\ast}^T,
\\
  \pi_\ast
& \in 
  I^{-(n-1)/4}\bigl(M^\text{int}{\times}S(M^\text{int}),C_{\pi_\ast}^\prime,\Omega_{M^\text{int}{\times}S(M^\text{int})}^{1/2}\bigr),
\\
  \pi^\ast
& \in 
  I^{-(n-1)/4}\bigl(S(M^\text{int}){\times}M^\text{int},C_{\pi^\ast}^\prime,\Omega_{S(M^\text{int}){\times}M^\text{int}}^{1/2}\bigr),
\\
  C_{\pi_\ast}
& =
  \bigl\{
  (\eta,D\pi\vert^T_v\eta)
  :
  v \in S(M^\text{int}), 
  \eta \in T^\ast_{\pi(v)}(M^\text{int})\setminus\{0\}
  \bigr\},
\\
  C_{\pi^\ast}
& =
  C_{\pi_\ast}^T.
\end{align*}
\end{lemma} 
To state the basic properties of the X-ray transform $\mathcal{X}$, 
we now recall the notion of clean composition for canonical relations. 
One may refer \cite[page~490, Appendix~C.3]{Hoermander3} for the definition of clean intersection for more general manifolds.  
\begin{definition}
\label{theorem:clean} 
Let $X$, $Y$ and $Z$ be smooth manifolds, 
and let $C_1$ and $C_2$ be homogeneous canonical relations such that 
$C_1 \subset T^\ast(X){\times}T^\ast(Y)$ and 
$C_2 \subset T^\ast(Y){\times}T^\ast(Z)$. 
\begin{itemize}
\item 
The composition $C_1{\circ}C_2$ is defined by 
$$
C_1{\circ}C_2
=
\{
(\xi,\zeta) 
: 
\exists \eta \in T^\ast(Y)\ \text{s.t.}\ 
(\xi,\eta) \in C_1, (\eta,\zeta) \in C_2
\}.
$$
\item 
We say that the composition $C_1{\circ}C_2$ is clean if the following conditions are satisfied. 
\begin{itemize}
\item[(i)] 
The intersection 
$$
C
=
(C_1{\times}C_2)
\cap
\Bigl(
T^\ast(X){\times}\Delta\bigl(T^\ast(Y)\bigr){\times}T^\ast(Z)
\Bigr)
$$
is clean, that is, $C$ is an embedded submanifold of 
$T^\ast(X){\times}T^\ast(Y){\times}T^\ast(Y){\times}T^\ast(Z)$, 
and 
$$
T_p(C)
=
T_p\bigl(C_1{\times}C_2\bigr)
\cap
T_p\Bigl(T^\ast(X){\times}\Delta\bigl(T^\ast(Y)\bigr){\times}T^\ast(Z)\Bigr)
$$
for any point $p \in C$, 
where $\Delta\bigl(T^\ast(Y)\bigr)$ is the diagonal part of $T^\ast(Y){\times}T^\ast(Y)$. 
\item[(ii)]
The projection map $\pi_C: C \rightarrow T^\ast(X){\times}T^\ast(Z)$ is proper. 
\item[(iii)] 
For any $\Xi \in T^\ast(X){\times}T^\ast(Z)$, $\pi_C^{-1}(\Xi)$ is connected. 
\end{itemize}
\item 
If the composition $C_1{\circ}C_2$ is clean, then 
$$
e
:=
\operatorname{codim}(C_1{\times}C_2)
+
\operatorname{codim}\Bigl(T^\ast(X){\times}\Delta\bigl(T^\ast(Y)\bigr){\times}T^\ast(Z)\Bigr)
-
\operatorname{codim}(C)
$$
is said to be the excess of the composition $C_1{\circ}C_2$. 
In particular, if $e=0$, that is, 
$$
N_p^\ast(C_1{\times}C_2)
\cap
N_p^\ast\Bigl(T^\ast(X){\times}\Delta\bigl(T^\ast(Y)\bigr){\times}T^\ast(Z)\Bigr)
=
\{0\}
$$
for any $p \in C$, we say that the composition $C_1{\circ}C_2$ is transversal. 
It is worth mentioning that $e=\dim\bigl(\pi_C^{-1}(\Xi)\bigr)$ and this can be employed as the definition of the excess. See \cite[Section~21.2 and Appendix~C.3]{Hoermander3} for the detail. 
\end{itemize}
\end{definition}
Using Lemma~\ref{theorem:Fpi} we have the basic properties of $\mathcal{X}$. In particular, we will often use the concrete expression of the canonical relation of $\mathcal{X}$ later.
\begin{theorem}
\label{theorem:geodesicxraytransform1}
Assume {\rm ({\bf A0})}. 
The composition $C_{F_\ast}{\circ}C_{\pi^\ast}$ is transversal, and 
\begin{align*}
  \mathcal{X}
& =
  F_\ast{\circ}\pi^\ast
  \in
  I^{-n/4}
  \bigl(
  \partial_-S(M){\times}M^\text{int},
  C_\mathcal{X}^\prime;
  \Omega_{\partial_-S(M){\times}M^\text{int}}^{1/2}
  \bigr),
\\
  C_\mathcal{X}
& =
  \bigl\{
  (\xi,\eta) 
  \in 
  T^\ast\bigl(\partial_-S(M)\bigr){\times}T^\ast(M^\text{int}): 
  \exists v \in S(M^\text{int})\ \text{s.t.}\ 
\\
& \qquad 
  \xi \in T^\ast_{F(v)}\bigl(\partial_-S(M)\bigr)\setminus\{0\}, 
  \eta \in T^\ast_{\pi(v)}(M^\text{int})\setminus\{0\}, 
  DF\vert^T_v\xi=D\pi\vert^T_v\eta
  \bigr\}.
\end{align*}
Moreover, the principal symbol of $\mathcal{X}$ does not vanish. 
\end{theorem}
\begin{proof}
It is well-known that $\mathcal{X}$ is an elliptic Fourier integral operator. 
We use the same method of \cite{HolmanUhlmann} to obtain the concrete expression of $C_\mathcal{X}$, and carefully compute the excess of the composition. Set 
\begin{align*}
  \mathcal{Z}
& :=
  T^\ast\bigl(\partial_-S(M)\bigr)
  \times
  T^\ast\bigl(S(M^\text{int})\bigr)
  \times
  T^\ast\bigl(S(M^\text{int})\bigr)
  \times
  T^\ast(M^\text{int})\\
  \mathcal{D}
& :=
  T^\ast\bigl(\partial_-S(M)\bigr)
  \times
  \Delta\Bigl(T^\ast\bigl(S(M^\text{int})\bigr)\Bigr)
  \times
  T^\ast(M^\text{int})
\end{align*}
for short. We have 
$\operatorname{codim}(\mathcal{D})=\operatorname{dim}\Bigl(T^\ast\bigl(S(M^\text{int})\bigr)\Bigr)=4n-2$. 
Lemma~\ref{theorem:Fpi} implies that 
\begin{align*}
  C_{F_\ast}{\times}C_{\pi^\ast}
& =
  \bigl\{
  (\xi,DF\vert^T_v\xi,D\pi\vert^T_u\eta,\eta)
  : 
  v,u \in S(M^\text{int}),
\\
& \qquad  
  \xi \in T^\ast_{F(v)}\bigl(\partial_-S(M)\bigr)\setminus\{0\}, 
  \eta \in T^\ast_{\pi(u)}(M^\text{int})\setminus\{0\}
  \bigr\}. 
\end{align*}
We have 
\begin{align*}
  \operatorname{dim}(C_{F_\ast}{\times}C_{\pi^\ast})
& =
  \operatorname{dim}\bigl(S(M^\text{int})\bigr)
  +
  \operatorname{dim}\Bigl(T^\ast_{F(v)}\bigl(\partial_-S(M)\bigr)\Bigr)
\\
& +
  \operatorname{dim}\bigl(S(M^\text{int})\bigr)
  +
  \operatorname{dim}\bigl(T^\ast_{\pi(u)}(M^\text{int})\bigr)
\\
& =
  (2n-1)+(2n-2)+(2n-1)+n=7n-4.
\end{align*}
Then we have 
\begin{align}
  C
& :=
  (C_{F_\ast}{\times}C_{\pi^\ast})
  \cap
  \mathcal{D}
\nonumber
\\
& =
  \bigl\{
  (\xi,DF\vert^T_v\xi,D\pi\vert^T_v\eta,\eta)
  : 
  v \in S(M^\text{int}),
\nonumber 
\\
& \qquad 
  \xi \in T^\ast_{F(v)}\bigl(\partial_-S(M)\bigr)\setminus\{0\}, 
  \eta \in T^\ast_{\pi(v)}(M^\text{int})\setminus\{0\}, 
  DF\vert^T_v\xi=D\pi\vert^T_v\eta
  \bigr\}.
\label{equation:transversal1} 
\end{align}
To compute the dimension of $C$, we need to know the properties of 
$\xi \in T^\ast_{F(v)}\bigl(\partial_-S(M)\bigr)$ and  
$\eta \in T^\ast_{\pi(v)}(M^\text{int})$ 
satisfying $DF\vert^T_v\xi=D\pi\vert^T_v\eta$. 
In this case we have $\eta(v)=0$. 
Indeed $F(v)=F\bigl(\dot{\gamma}_v(t)\bigr)$ for any $t \in \bigl(\tau_-(v),\tau_+(v)\bigr)$, and this implies that 
$DF\vert_v\ddot{\gamma}_v(0)=0$. Hence we have 
\begin{equation}
\eta(v)
=
\eta\bigl(D\pi\vert_v\ddot{\gamma}_v(0)\bigr)
=
D\pi\vert^T_v\eta\bigl(\ddot{\gamma}_v(0)\bigr)
=
DF\vert^T_v\xi\bigl(\ddot{\gamma}_v(0)\bigr)
=
\xi\bigl(DF\vert_v\ddot{\gamma}_v(0)\bigr)
=
0. 
\label{equation:etav}
\end{equation}
Since $F:S(M^\text{int}) \rightarrow \partial_-S(M)$ and $\pi:S(M^\text{int}) \rightarrow M^\text{int}$ are submersions, 
we deduce that $DF\vert^T_v:T^\ast_{F(v)}\bigl( \partial_-S(M)\bigr) \rightarrow T^\ast_v\bigl(S(M^\text{int})\bigr)$ 
and $D\pi\vert^T_v:T^\ast_{\pi(v)}(M^\text{int}) \rightarrow T^\ast_v\bigl(S(M^\text{int})\bigr)$ are linear injections for any $v \in S(M^\text{int})$. 
We note that 
$T^\ast_v(T_{\pi(v)}(M^\text{int}))$ is identified with $T^\ast_{\pi(v)}(M^\text{int})$  
and $\zeta(v)=0$ holds for any $\zeta \in T^\ast_v(S_{\pi(v)}(M^\text{int}))$ 
since $v \in S_{\pi(v)}(M^\text{int})$. Then we can identify 
$T^\ast_v\bigl(S(M^\text{int})\bigr)$ with 
$$
\{(\eta,\zeta) \in T^\ast_{\pi(v)}(M^\text{int}){\times}T^\ast_{\pi(v)}(M^\text{int}): \zeta(v)=0\},
$$
and we have $\operatorname{dim}\Bigl(\operatorname{Ran}\bigl(DF\vert^T_v\bigr)\Bigr)=2n-2$. 
Hence we obtain two identifications denoted by $\simeq$ and an identity:  
\begin{align*}
  \operatorname{Ran}\bigl(DF\vert^T_v\bigr)
& \simeq
  \{(\eta,\zeta) \in T^\ast_{\pi(v)}(M^\text{int}){\times}T^\ast_{\pi(v)}(M^\text{int}): \eta(v)=0, \zeta(v)=0\},
\\
  \operatorname{Ran}\bigl(D\pi\vert^T_v\bigr)
& =
  \{(\eta,0) : \eta \in T^\ast_{\pi(v)}(M^\text{int})\},
\\
  \operatorname{Ran}\bigl(DF\vert^T_v\bigr)\cap\operatorname{Ran}\bigl(D\pi\vert^T_v\bigr)
& \simeq
  \{(\eta,0): \eta \in T^\ast_{\pi(v)}(M^\text{int}), \eta(v)=0\}. 
\end{align*}
Hence if 
$\xi \in T^\ast_{F(v)}\bigl(\partial_-S(M)\bigr)$ and  
$\eta \in T^\ast_{\pi(v)}(M^\text{int})$ 
satisfying $DF\vert^T_v\xi=D\pi\vert^T_v\eta$, then 
$\eta(v)=0$ and $\xi$ is uniquely determined by 
$\xi=(DF\vert^T_v)^{-1}{\circ}D\pi\vert^T_v\eta$, 
where $(DF\vert^T_v)^{-1}$ is the inverse linear mapping of the linear bijection 
$DF\vert^T_v:T^\ast_{F(v)}\bigl(\partial_-S(M)\bigr) \rightarrow \operatorname{Ran}\bigl(DF\vert^T_v\bigr)$. 
Then we have 
\begin{align*}
  \operatorname{dim}(C)
& =
  \operatorname{dim}\bigl(S(M^\text{int})\bigr)
  +
  \operatorname{dim}\bigl(\{\eta \in T^\ast_{\pi(v)}(M^\text{int}): \eta(v)=0\}\bigr)
\\
& =
  (2n-1)+(n-1)
  =
  3n-2. 
\end{align*}
Thus we have 
\begin{align*}
  e
& =
  \operatorname{codim}(\mathcal{D})
  +
  \operatorname{codim}(C_{F_\ast}{\times}C_{\pi^\ast})
  -
  \operatorname{codim}(C)
\\
& =
  \operatorname{codim}(\mathcal{D})
  +
  \operatorname{dim}(\mathcal{Z})
  -
  \operatorname{dim}(C_{F_\ast}{\times}C_{\pi^\ast})
  -
  \operatorname{dim}(\mathcal{Z})
  +
  \operatorname{dim}(C)   
\\
& =
  \operatorname{codim}(\mathcal{D})
  -
  \operatorname{dim}(C_{F_\ast}{\times}C_{\pi^\ast})
  +
  \operatorname{dim}(C)
\\
& =
  (4n-2)-(7n-4)+(3n-2)=0.
\end{align*}
If we show that $C_{F_\ast}{\circ}C_{\pi^\ast}$ is clean, 
it follows that $C_{F_\ast}{\circ}C_{\pi^\ast}$ is transversal, and therefore  
\cite[Theorem~25.2.3]{Hoermander4} proves the first part of Theorem~\ref{theorem:geodesicxraytransform1}. 
It follows that $\mathcal{X}$ is an elliptic Fourier integral operator 
since the amplitude function is identically equal to $1$.   
\par
So it suffices to show that $C_{F_\ast}{\circ}C_{\pi^\ast}$ is clean. 
Since $DF\vert^T_v\xi$ and $D\pi\vert^T_v\eta$ are continuous in $(v,\xi)$ and $(v,\eta)$ respectively, 
it is easy to see that $\pi_C$ is proper, and that $\pi_C^{-1}\bigl((\xi,\eta)\bigr)$ is connected. 
Hence it suffices to show that $C=(C_{F_\ast}{\times}C_{\pi^\ast})\cap\mathcal{D}$ is a clean intersection. 
Fix arbitrary 
$$
c=(\xi,DF\vert^T_v\xi,D\pi\vert^T_v\eta,\eta) \in C,
\quad
X=(X_1,X_2,X_3,X_4) \in T_c(C_{F_\ast}{\times}C_{\pi^\ast})
$$ 
with $DF\vert^T_v\xi=D\pi\vert^T_v\eta$. 
Consider a curve on $C_{F_\ast}{\times}C_{\pi^\ast}$ of the form
\begin{align*}
  \mathbb{R}\ni\alpha 
& \mapsto
  \bigl(\xi(\alpha),Y(\alpha),Z(\alpha),\eta(\alpha)\bigr)
\\
& =
  c+\alpha{X}+\mathcal{O}(\alpha^2)
\\
& =
  (\xi+\alpha{X}_1,DF\vert^T_v\xi+\alpha{X_2},D\pi\vert^T_v\eta+\alpha{X_3},\eta+\alpha{X_4})
  +
  \mathcal{O}(\alpha^2)
\end{align*}
near $\alpha=0$. If $X \in T_c(\mathcal{D})$, then we have $\dot{Y}(0)=X_2=X_3=\dot{Z}(0)$. Hence 
\begin{align*}
  X
& =
  \frac{d}{d\alpha}\bigg\vert_{\alpha=0}
  \bigl(\xi(\alpha),Y(\alpha),Z(\alpha),\eta(\alpha)\bigr)
\\
& =
  \frac{d}{d\alpha}\bigg\vert_{\alpha=0}
  \bigl(\xi(\alpha),Y(\alpha),Y(\alpha),\eta(\alpha)\bigr)
  \in 
  T_c(C), 
\end{align*}
and we have $T_c(C) \supset T_c(C_{F_\ast}{\times}C_{\pi^\ast}){\cap}T_c(\mathcal{D})$. 
Since $T_c(C) \subset T_c(C_{F_\ast}{\times}C_{\pi^\ast}){\cap}T_c(\mathcal{D})$ is obvious, 
we deduce that $T_c(C) = T_c(C_{F_\ast}{\times}C_{\pi^\ast}){\cap}T_c(\mathcal{D})$, and that 
$(C_{F_\ast}{\times}C_{\pi^\ast}){\cap}\mathcal{D}$ is a clean intersection. This completes the proof.
\end{proof}
Combining Theorem~\ref{theorem:geodesicxraytransform} and \cite[Theorem~25.2.2]{Hoermander4},  
we immediately conclude that the adjoint $\mathcal{X}^T$ of $\mathcal{X}$ is also a similar elliptic Fourier integral operator. 
It is well-known that if there is no conjugate point in $(M,g)$, 
the normal operator $\mathcal{X}^T\circ\mathcal{X}$ becomes an elliptic pseudodifferential operator of order $-1$, 
and that there exists a parametrix for 
$\mathcal{X}^T\circ\mathcal{X}$ locally in $M$. 
See \cite{GuilleminSternberg, HolmanUhlmann} for instance.  
Holman and Uhlmann in \cite{HolmanUhlmann} studied the case that there are conjugate points, 
and proved the decomposition formula of the normal operator according to the order of the conjugate points. This type of results was firstly obtained by Stefanov and Uhlmann in their pioneering work \cite{PlamenGunther}. 
\begin{theorem}
\label{theorem:adointofgeodesicxraytransform} 
Assume {\rm ({\bf A0})}. 
$\mathcal{X}^T$ is also an elliptic Fourier integral operator such that 
\begin{align*}
  \mathcal{X}^T
& =
  F_\ast{\circ}\pi^\ast
  \in
  I^{-n/4}
  \bigl(
  M^\text{int}\times\partial_-S(M),
  C_{\mathcal{X}^T}^\prime;
  \Omega_{M^\text{int}\times\partial_-S(M)}^{1/2}
  \bigr),
\\
  C_{\mathcal{X}^T}
& =
  \bigl\{
  (\eta,\xi) 
  \in 
  T^\ast\bigl(\partial_-S(M)\bigr){\times}T^\ast(M^\text{int}): 
  \exists v \in S(M^\text{int})\ \text{s.t.}\ 
\\
& \qquad 
  \xi \in T^\ast_{F(v)}\bigl(\partial_-S(M)\bigr)\setminus\{0\}, 
  \eta \in T^\ast_{\pi(v)}(M^\text{int})\setminus\{0\}, 
  DF\vert^T_v\xi=D\pi\vert^T_v\eta
  \bigr\}.
\end{align*}
Furthermore the normal operator $\mathcal{X}^T\circ\mathcal{X}$ is an elliptic pseudodifferential operator of order $-1$. 
\end{theorem}
%
%  
%%%%%%
%%%%%% section 3
%%%%%%
\section{Jacobi Fields and Parallel Transport}
\label{section:jacobifields}
In this section we assume both ({\bf A0}) and ({\bf A1}), 
and study Jacobi fields on $M^\text{int}$. 
Let $Y(t)$ be the Jacobi field along $\gamma_v(t)$ with $v \in S(M^\text{int})$ 
such that 
$\bigl(Y(0),\nabla{Y(0)}\bigr)=\Xi \in T_v\bigl(S(M^\text{int})\bigr)$, 
where $\nabla{Y}(t):=\nabla_{\partial/\partial{t}}Y(t)$. 
More precisely, if we express things in the local coordinates such as 
$$
x=(x_1,\dotsc,x_n) \in M^\text{int}, 
\quad
\gamma_v(t)=\bigl(c_1(t),\dotsc,c_n(t)\bigr)
\quad
Y(t)
=
\sum_{i=1}^n
Y_i(t)
\left(\frac{\partial}{\partial x_i}\right)_{\gamma_v(t)},
$$
$$
\nabla{Y}(t)
=
\nabla_{\partial/\partial{t}}Y(t)
:=
\sum_{i=1}^n
\left\{
\frac{dY_i}{dt}(t)
+
\Gamma^i_{jk}\bigl(\gamma_v(t)\bigr)\dot{c}_j(t)Y_k(t)
\right\}
\left(\frac{\partial}{\partial x_i}\right)_{\gamma_v(t)}, 
$$
where $\Gamma^i_{jk}(x)$ are the Christoffel symbols of $(M,g)$ at $x \in M$. 
Then we have 
$$
D\pi\vert_{\Psi(v,s)}{\circ}D_v\Psi\vert_{(v,s)}\Xi=Y(s), 
\quad
D_v\Psi\vert_{(v,s)}\Xi=\nabla{Y(s)}. 
$$
See \cite[Lemma~4.3 in Page~56]{Sakai} for instance. 
Note that 
\begin{align*}
  \operatorname{Ker}(D\pi\vert_v)
& =
  \bigl\{
  \Xi \in T_v\bigl(S(M^\text{int})\bigr): 
  D\pi\vert_v(\Xi)=0
  \bigr\}
\\
& =
  \bigl\{(0,\zeta): \zeta \in T_v\bigl(S_{\pi(v)}(M^\text{int})\bigr)\bigr\}
\\
& \simeq
  \{(0,\zeta): \zeta \in T_{\pi(v)}(M^\text{int}), \langle{v,\zeta}\rangle=0\}. 
\end{align*}
Then we have for any $\Xi=(0,\zeta) \in \operatorname{Ker}(D\pi\vert_v)$ 
$$
Y(t)
=
\frac{d}{ds}\bigg\vert_{s=0}\exp_{\pi(v)}t(v+s\zeta)
=
tD\exp_{\pi(v)}\Big\vert_{tv}\zeta. 
$$
We now recall the Gauss lemma for the property of Jacobi fields.  
\begin{lemma}[{\bf Gauss Lemma}]
\label{theorem:gauss}
Let $Y$ be a Jacobi field along a geodesic $\gamma$. Then 
$$
\langle{Y(t),\dot{\gamma}(t)\rangle}
=
\langle{Y(0),\dot{\gamma}(0)\rangle}
+
t
\langle{{\nabla}Y(0),\dot{\gamma}(0)\rangle}.
$$
\end{lemma}
\begin{proof} 
Lemma~\ref{theorem:gauss} follows from 
$$
\frac{d^2}{dt^2}
\langle{Y(t),\dot{\gamma}(t)\rangle}
=
-
\bigl\langle
R\bigl(Y(t),\dot{\gamma}(t)\bigr)\dot{\gamma}(t),\dot{\gamma}(t)\bigr)
\bigr\rangle
=
0,
$$
where $R$ is the Riemann curvature tensor. 
See \cite[Proposition~2.3 in Page 36]{Sakai} for the detail. 
\end{proof}
The assumption ({\bf A1}) ensures that all the Jacobi fields on $M^\text{int}$ 
are of the form of the product of a scalar function and a parallel transport of a tangent vector. This property plays a crucial role in the computation of the composition of $C_{\mathcal{X}^T}$ with some Lagrangian submanifolds. So we need to introduce the parallel transport of covectors along geodesics in $M^\text{int}$. Let $c$ be a smooth curve in $M^\text{int}$. 
We denote by $P(t_0,t_1;c)v$ the parallel transport of $v \in T_{c(t_0)}$ at $c(t_1)$ along the curve $c$.  
Suppose that $t_0{\ne}t_1$ and 
$\tilde{\eta} \in T^\ast_{c(t_1)}(M^\text{int})$. 
Then we have 
$$
\tilde{\eta} \bigl(P(t_0,t_1;c)v\bigr)
=
P(t_0,t_1;c)^T\tilde{\eta}(v),
\quad
v \in T_{c(t_0)}(M^\text{int}).
$$ 
We say that $P(t_0,t_1;c)^T\tilde{\eta} \in T^\ast_{c(t_0)}(M^\text{int})$ is the parallel transport of 
$\tilde{\eta} \in T_{c(t_1)}(M^\text{int})$ at $c(t_0)$ along the curve $c$. 
\par
We now recall F.~Schur's lemma for the property of sectional curvatures. 
See, e.g., \cite[Proposition~3.6 in page~46]{Sakai} for this. 
\begin{lemma}
\label{theorem:schur} 
Let $(M,g)$ be a Riemannian manifold, 
and let $K_x(\sigma)$ be the sectional curvature at $x \in M$ 
for a two-dimensional vector subspace $\sigma$ of $T_x(M)$. 
\begin{itemize}
\item 
The sectional curvature $K_x(\sigma)$ is independent of the choice of $\sigma$ 
at a point $x \in M$ if and only if 
$$
R(X,Y)Z
=
k(x)
\{
\langle{Y,Z}\rangle{X}
-
\langle{X,Z}\rangle{Y}
\}
$$
for any $X,Y,Z \in T_x(M)$, where $k(x):=K_x(\cdot)$. 
\item
{\rm ({\bf F.~Schur's lemma})}\ 
If $K_x(\sigma)$ is independent of the choice of 
$\sigma$ at every point $x \in M$, and $\operatorname{dim}(M)\geqq3$, 
then $k(x):=K_x(\cdot)$ is independent of $x \in M$, that is, 
$k(x)$ is a constant function on $M$.  
\end{itemize}
\end{lemma}
If we suppose ({\bf A1}), the equation for the Jacobi field is reduced to a simpler one. 
\begin{lemma}
\label{theorem:jacobi1}
Suppose that $(M,g)$ is a Riemannian manifold satisfying {\rm ({\bf A1})}. 
Let $Y(t)$ be a Jacobi field along a normal geodesic $\gamma(t)$ with 
$$
\langle{Y(0),\dot{\gamma}(0)}\rangle=0,
\quad
\langle{{\nabla}Y(0),\dot{\gamma}(0)}\rangle=0. 
$$
If $(M,g)$ is a space of constant curvature $K$, then $Y(t)$ solves 
$$
\nabla\nabla{Y(t)}+KY(t)=0.
$$
If $(M,g)$ is a two-dimensional manifold 
with a sectional curvature $k(x)$ at $x \in M$, 
then $Y(t)$ solves 
$$
\nabla\nabla{Y(t)}+k\bigl(\gamma(t)\bigr)Y(t)=0.
$$
\end{lemma} 
\begin{proof}
In view of Lemma~\ref{theorem:schur}, we have only to show the latter one. 
Applying Lemma~\ref{theorem:schur} to the equation of the Jacobi field, we have 
\begin{align*}
  \nabla\nabla{Y(t)}
& =
  -R\bigl(Y(t),\dot{\gamma}(t)\bigr)\dot{\gamma}(t)
\\
& =
  -
  k\bigl(\gamma(t)\bigr)
  \{
  \langle{\dot{\gamma}(t),\dot{\gamma}(t)}\rangle{Y(t)}
  -
  \langle{Y(t),\dot{\gamma}(y)}\rangle\dot{\gamma}(t)
  \} 
\\
& =
  - 
  k\bigl(\gamma(t)\bigr)
  \{
  Y(t)
  -
  \langle{Y(t),\dot{\gamma}(y)}\rangle\dot{\gamma}(t)
  \} 
\end{align*}
Using Lemma~\ref{theorem:gauss} 
and the orthogonal condition on $Y(0)$ and $\nabla{Y(0)}$, 
we deduce that 
\begin{align*}
  \nabla\nabla{Y(t)}
& =
  -
  k\bigl(\gamma(t)\bigr)
  \{
  Y(t)
  -
  \langle{Y(0),\dot{\gamma}(0)}\rangle\dot{\gamma}(t)
  -
  t\langle{{\nabla}Y(0),\dot{\gamma}(0)}\rangle\dot{\gamma}(t)
  \}
\\
& =
  -
  k\bigl(\gamma(t)\bigr)Y(t).  
\end{align*}
This completes the proof.
\end{proof}
We show that if ({\bf A1}) holds, then all the Jacobi fields are 
of the form of the product of a scalar function and a parallel transport of 
a tangent vector. 
\begin{lemma}
\label{theorem:jacobi2} 
Suppose that $(M,g)$ is a Riemannian manifold satisfying {\rm ({\bf A1})}. 
Let $Y(t)$ be a Jacobi field along a normal geodesic $\gamma(t)$ with 
$$
\langle{Y(0),\dot{\gamma}(0)}\rangle=0,
\quad
\langle{{\nabla}Y(0),\dot{\gamma}(0)}\rangle=0. 
$$
Set 
$$
X(t):=P(0,t;\gamma)Y(0),
\quad
\zeta(t):=P(0,t;\gamma)\nabla{Y}(0). 
$$
If $(M,g)$ is a space of constant curvature $K$,  
then $Y(t)$ is given by 
$$
Y(t)=a(t)X(t)+b(t)\zeta(t),
$$
\begin{alignat*}{3}
  a^{\prime\prime}(t)+Ka(t)
& =
  0,
& \quad 
  a(0)
& =
  1,
& \quad
  a^\prime(0)
& =
  0,
\\
  b^{\prime\prime}(t)+Kb(t)
& =
  0,
& \quad 
  b(0)
& =
  0,
& \quad
  b^\prime(0)
& =
  1. 
\end{alignat*}
If $(M,g)$ is a two-dimensional manifold 
with a sectional curvature $k(x)$ at $x \in M$, 
then $Y(t)$ is given by 
$$
Y(t)=a(t)X(t)+b(t)\zeta(t),
$$
\begin{alignat*}{3}
  a^{\prime\prime}(t)+k\bigl(\gamma(t)\bigr)a(t)
& =
  0,
& \quad 
  a(0)
& =
  1,
& \quad
  a^\prime(0)
& =
  0,
\\
  b^{\prime\prime}(t)+k\bigl(\gamma(t)\bigr)b(t)
& =
  0,
& \quad 
  b(0)
& =
  0,
& \quad
  b^\prime(0)
& =
  1. 
\end{alignat*}
\end{lemma}
\begin{proof}
We show only the latter one. The former one can be proved in the same way. 
Suppose that $\operatorname{dim}(M)=2$. 
Lemma~\ref{theorem:jacobi1} shows that $Y(t)$ solves 
$\nabla\nabla{Y(t)}+k\bigl(\gamma(t)\bigr)Y(t)=0$. 
Set $Z(t)=a(t)X(t)+b(t)\zeta(t)$. 
Note that $\nabla{X}(t)=0$ and $\nabla{\zeta}(t)=0$. 
Then we have 
\begin{align*}
  Z(0)
& =
  a(0)X(0)+b(0)\zeta(0)=X(0)=P(0,0;\gamma)Y(0)=Y(0),
\\
  \nabla{Z}(t)
& =
  a^\prime(t)X(t)+a(t)\nabla{X}(t)+b^\prime(t)\zeta(t)+b(t)\nabla\zeta(t)
\\
& =
  a^\prime(t)X(t)+b^\prime(t)\zeta(t),
\\
  \nabla{Z}(0)
& =
  a^\prime(0)X(0)+b^\prime(0)\zeta(0), 
\\
& =
  \zeta(0)
  =
  P(0,0;\gamma)\nabla{Y(0)}=\nabla{Y(0)}. 
\end{align*}
If we prove that $Z(t)$ solves 
$\nabla\nabla{Z(t)}+k\bigl(\gamma(t)\bigr)Z(t)=0$, 
then $Y=Z$ follows from the uniqueness of the solution to the initial value problem for the system of linear ordinary differential equations. 
We deduce that 
\begin{align*}
& \nabla\nabla{Z(t)}+k\bigl(\gamma(t)\bigr)Z(t)
\\
  =
& \nabla\{a^\prime(t)X(t)+b^\prime(t)\zeta(t)\}
  +
  k\bigl(\gamma(t)\bigr)\{a(t)X(t)+b(t)\zeta(t)\}
\\
  =
& a^{\prime\prime}(t)X(t)+a^\prime(t)\nabla{X(t)}
  +
  b^{\prime\prime}(t)\zeta(t)+b^\prime(t)\nabla\zeta(t)
\\
  +
& k\bigl(\gamma(t)\bigr)\{a(t)X(t)+b(t)\zeta(t)\}
\\
  =
& a^{\prime\prime}(t)X(t)+b^{\prime\prime}(t)\zeta(t)
  +
  k\bigl(\gamma(t)\bigr)\{a(t)X(t)+b(t)\zeta(t)\}
\\
  =
& \{a^{\prime\prime}(t)+k\bigl(\gamma(t)\bigr)a(t)\}X(t)
  +
  \{b^{\prime\prime}(t)+k\bigl(\gamma(t)\bigr)b(t)\}\zeta(t)
\\
  =
& 0.
\end{align*}
This completes the proof. 
\end{proof}
Finally we show the properties of solutions $a(t)$ and $b(t)$.  
These properties are essentially related to the assumption that there is no conjugate point in $M^\text{int}$.  
\begin{lemma}
\label{theorem:noconjugatepoint}
Suppose {\rm ({\bf A0})} and {\rm ({\bf A1})}. 
Fix arbitrary $w \in \partial_-S(M)$ and 
$s,t_0,\tilde{t}_0 \in \bigl(0,\tau_+(w)\bigr)$ 
with $t_0{\ne}\tilde{t}_0$. 
If $(M,g)$ is a space of constant curvature $K$, 
we assume that $a(t;s)$ and $b(t;s)$ are solutions to 
\begin{alignat}{3}
  a_{tt}(t;s)+Ka(t;s)
& =
  0,
& \quad 
  a(s;s)
& =
  1,
& \quad
  a_t(s;s)
& =
  0,
\label{equation:aK}
\\
  b_{tt}(t;s)+Kb(t;s)
& =
  0,
& \quad 
  b(s;s)
& =
  0,
& \quad
  b_t(s;s)
& =
  1. 
\label{equation:bK}
\end{alignat}
If $(M,g)$ is a two-dimensional manifold 
with a sectional curvature $k(x)$ at $x \in M$, 
we assume that $a(t;s)$ and $b(t;s)$ are solutions to 
\begin{alignat}{3}
  a_{tt}(t;s)+k\bigl(\gamma_w(t)\bigr)a(t;s)
& =
  0,
& \quad 
  a(s;s)
& =
  1,
& \quad
  a_t(s;s)
& =
  0,
\label{equation:akappa}
\\
  b_{tt}(t;s)+k\bigl(\gamma_w(t)\bigr)b(t;s)
& =
  0,
& \quad 
  b(s;s)
& =
  0,
& \quad
  b_t(s;s)
& =
  1. 
\label{equation:bkappa}
\end{alignat}
If we set 
$\Delta(t_0,\tilde{t}_0;s):=a(t_0;s)b(\tilde{t}_0;s)-a(\tilde{t}_0;s)b(t_0;s)$, 
then we have $\Delta(t_0,\tilde{t}_0;s)\ne0$.   
\end{lemma}
\begin{proof}
Let $\kappa$ be a positive constant. 
We prove Lemma~\ref{theorem:noconjugatepoint} for four cases: 
$K=\kappa^2>0$, $K=0$, $K=-\kappa^2<0$, and $\operatorname{dim}(M)=2$.
\vspace{6pt}
\\
\underline{Case of $K=\kappa^2>0$}.\ 
In this case $\overline{M}$ must be contained 
in the interior of the hemisphere of the radius $1/\kappa$ 
since there is no conjugate point in $M$. 
Otherwise, there are multiple great circles in $M^\text{int}$, 
and the pair of the crossing points of two great circles are conjugate points to each other. 
One may refer to \cite[Chapter~13, 2.4 Example]{doCarmo}. 
Hence $0 \leqq s, t_0, \tilde{t_0} < \pi/\kappa$ follows, 
and we have $0<\lvert{t_0-\tilde{t}_0}\rvert<\pi/\kappa$. 
In this case the solutions $a(t;s)$ and $b(t;s)$ are given by 
$$
a(t;s)
=
\cos\{\kappa(t-s)\}, 
\quad
b(t;s)
=
\frac{\sin\{\kappa(t-s)\}}{\kappa}. 
$$
Since $0<\lvert{t_0-\tilde{t}_0}\rvert<\pi/\kappa$, we have 
$$
\Delta(t_0,\tilde{t}_0;s)
=
\frac{\sin\{\kappa(\tilde{t}_0-t_0)\}}{\kappa}
\ne
0.
$$
\vspace{6pt}
\\
\underline{Case of $K=0$}.\ 
In this case we have $a(t;s)\equiv1$, $b(t;s)=t-s$, 
and $\Delta(t_0,\tilde{t}_0;s)=\tilde{t}_0-t_0\ne0$.
\vspace{6pt}
\\
\underline{Case of $K=-\kappa^2<0$}.\ 
In this case we have 
$$
a(t;s)=\cosh\{\kappa(t-s)\}, 
\quad 
b(t;s)=\frac{\sinh\{\kappa(t-s)\}}{\kappa}, 
$$
$$
\Delta(t_0,\tilde{t}_0;s)=\frac{\sinh\{\kappa(\tilde{t}_0-t_0)\}}{\kappa}\ne0.
$$
\vspace{6pt}
\\
\underline{Case of $\operatorname{dim}(M)=2$}.\ 
Fix arbitrary $s$. Suppose that $\Delta(t_0,\tilde{t}_0;s)=0$.  
Then there exists a nontrivial solution $[c_1,c_2]^T$ to the linear system of the form 
$$
c_1a(t_0;s) + c_2b(t_0;s)=0, 
\quad
c_1a(\tilde{t}_0;s) + c_2b(\tilde{t}_0;s)=0,  
$$
since $\Delta(t_0,\tilde{t}_0;s)$ is a determinant of a $2\times2$ matrix:
$$
\Delta(t_0,\tilde{t}_0;s)
=
\operatorname{det}
\begin{bmatrix}
a(t_0;s) & b(t_0;s) 
\\
a(\tilde{t}_0;s) & b(\tilde{t}_0;s) 
\end{bmatrix}. 
$$
Pick up a unit tangent vector field $\dot{\gamma}_w^\perp$  which is orthogonal to $\dot{\gamma}_w(s) $ 
in $T_{\gamma_w(s)}(M^\text{int})$, and set $\dot{\gamma}_w^\perp(t):=P(s,t;\gamma_w)\dot{\gamma}_w^\perp$. 
Then $\langle\dot{\gamma}_w^\perp(t),\dot{\gamma}_w(t)\rangle=0$ for any $t$, 
$$
Y(t)
:=
c_1a(t;s)P(s,t;\gamma_w)\dot{\gamma}_w^\perp+c_2b(t;s)P(s,t;\gamma_w)\dot{\gamma}_w^\perp
=
\bigl(c_1a(t;s)+c_2b(t;s)\bigr)\dot{\gamma}_w^\perp(t)
$$
becomes a nontrivial Jacobi field along $\gamma_w$, and $\langle{Y(t),\dot{\gamma}_w(t)}\rangle=0$ for any $t$. 
If we take the inner product of $Y(t)$  and $\dot{\gamma}_w^\perp(t)$ at $t=t_0$ and $t=\tilde{t}_0$, we obtain 
$$
\langle{Y(t_0),\dot{\gamma}_w^\perp(t_0)}\rangle
=
c_1a(t_0;s) + c_2b(t_0;s)=0, 
\quad
\langle{Y(\tilde{t}_0),\dot{\gamma}_w^\perp(\tilde{t}_0)}\rangle
=
c_1a(\tilde{t}_0;s) + c_2b(\tilde{t}_0;s)=0. 
$$
Then we have $Y(t_0)=0$ and $Y(\tilde{t}_0)=0$. This shows that $\gamma_w(t_0)$ is a conjugate point of $\gamma_w(\tilde{t}_0)$ along $\gamma_w$. This is a contradiction. 
\end{proof}
%
%
%%%%%%
%%%%%% section 4
%%%%%%
\section{Geometry of Metal Region}
\label{section:geometry}
In this section we study the geometry of metal region. 
We now assume ({\bf A0}), and will assume ({\bf A1}) later. We now state the assumption on the metal region $D$ such that $\overline{D}$ is contained in $M^\text{int}$. Suppose the following. 
\begin{description}
\item[(A2)] 
$D$ is of the form $D=\displaystyle\bigcup_{j=1}^JD_j$ with $\overline{D_j}\cap\overline{D_k}=\emptyset$ ($j{\ne}k$), that is, $D$ is a disjoint union of finite number of strictly convex bounded domains $D_j$ ($j=1,\dotsc,J$) with smooth boundary.  
\end{description}  
\begin{center}
\includegraphics[width=50mm]{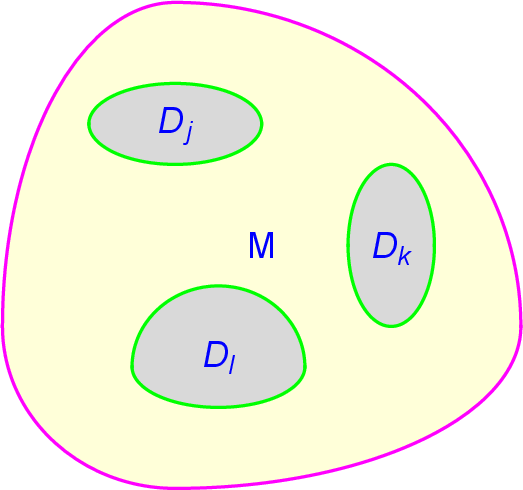}
\vspace{11pt}
\\
Figure~4. $M$ and the metal region $D$.
\end{center}
In this case $M^\text{int}\setminus\overline{D}$ is the region of normal tissue. 
The streaking artifacts are caused by 
$$
\mathcal{X}^T
\left[
\left(
\frac{\mathcal{X}[1_D]}{\lvert{d\mu}\rvert^{1/2}}
\right)^{2l}
\lvert{d\mu}\rvert^{1/2}
\right],
\quad
l=1,2,3,\dotsc.
$$ 
See \eqref{equation:singularity} for this. 
In this section we prepare for studying the microlocal singularities. 
More precisely, we study so-called the intersection calculus needed in the next section. 
It is well-known that $1_D$ is a conormal distribution section whose singular support is the boundary $\partial{D}$. 
So the next lemma will be used for characterizing the conormal singularities of $\mathcal{X}[1_D]$. 
\begin{lemma}
\label{theorem:4-1}
Let $j=1,\dotsc,J$. 
Then the composition $C_\mathcal{X}{\circ}N^\ast(\partial{D_j})$  is transversal. 
Moreover if we set 
$$
\Sigma_j
:=
\pi_{\partial_-S(M)}
\bigl(C_\mathcal{X}{\circ}N^\ast(\partial{D_j})\bigr), 
\quad
j=1,\dotsc,J,
$$
then $\operatorname{codim}(\Sigma_j)=1$ and $N^\ast(\Sigma_j)=C_\mathcal{X}{\circ}N^\ast(\partial{D_j})$. 
\end{lemma} 
\begin{proof}
We first show that $C_\mathcal{X}{\circ}N^\ast(D_j)$  is transversal. 
Set 
\begin{align*}
  \mathcal{Z}
& =
  T^\ast\bigl(\partial_-S(M)\bigr)
  \times
  T^\ast(M^\text{int})
  \times
  T^\ast(M^\text{int}),
\\
  \mathcal{D}
& =
  T^\ast\bigl(\partial_-S(M)\bigr)
  \times
  \Delta\bigl(T^\ast(M^\text{int})\bigr)
\end{align*}
for short. We have 
$\operatorname{codim}(\mathcal{D})=\operatorname{dim}\bigl(T^\ast(M^\text{int})\bigr)=2n$. 
In view of Theorem~\ref{theorem:geodesicxraytransform} and \eqref{equation:etav}, we have  
\begin{align*}
  C_\mathcal{X}{\times}N^\ast({\partial}D_j)
& =
  \bigl\{
  (\xi,\eta,\zeta) 
  \in 
  T^\ast\bigl(\partial_-S(M)\bigr){\times}T^\ast(M^\text{int}){\times}N^\ast(\partial{D_j}): 
\\
& \qquad
  \exists v \in S(M^\text{int})\ \text{s.t.}\  
  \eta \in T^\ast_{\pi(v)}(M^\text{int})\setminus\{0\}, 
  \eta(v)=0, 
\\
& \qquad 
  \xi \in T^\ast_{F(v)}\bigl(\partial_-S(M)\bigr)\setminus\{0\} 
  \ \text{is determined by}\ 
  DF\vert^T_v\xi=D\pi\vert^T_v\eta
  \bigr\}. 
\end{align*}
Since $C_\mathcal{X}$ and $N^\ast(\partial{D_j})$ are Lagrangian submanifolds, we have 
\begin{align*}
  \operatorname{dim}\bigl(C_\mathcal{X}{\times}N^\ast({\partial}D_j)\bigr)
& = 
  \operatorname{dim}\bigl(\partial_-S(M)\bigr)
  +
  \operatorname{dim}(M^\text{int})
  +
  \operatorname{dim}(M^\text{int})
\\
& =
  (2n-2)+n+n
  =4n-2.
\end{align*} 
We set 
\begin{align}
  C
& :=
  \bigl(C_\mathcal{X}{\times}N^\ast({\partial}D_j)\bigr)\cap\mathcal{D}
\nonumber
\\
& =
  \bigl\{
  (\xi,\eta,\zeta) 
  \in 
  T^\ast\bigl(\partial_-S(M)\bigr){\times}T^\ast(M^\text{int}){\times}N^\ast(\partial{D_j}): 
\nonumber
\\
& \qquad 
  \exists v \in S(M^\text{int})\ \text{s.t.}\  
  \xi \in T^\ast_{F(v)}\bigl(\partial_-S(M)\bigr)\setminus\{0\}, 
\nonumber
\\
& \qquad
  \eta \in T^\ast_{\pi(v)}(M^\text{int})\setminus\{0\}, 
  \eta(v)=0, 
  DF\vert^T_v\xi=D\pi\vert^T_v\eta,\eta=\zeta\}
\nonumber
\\
& =
  \bigl\{
  (\xi,\eta,\eta): 
  \exists v \in S({\partial}D_j)\ \text{s.t.}\  
  \xi \in T^\ast_{F(v)}\bigl(\partial_-S(M)\bigr)\setminus\{0\},
\nonumber
\\
& \qquad 
  \eta \in N^\ast_{\pi(v)}(\partial{D_j})\setminus\{0\},  
  DF\vert^T_v\xi=D\pi\vert^T_v\eta\bigr\}.
\label{equation:intersection4-1}
\end{align}
Since $\xi$ is uniquely determined by $\eta$ and $v$ for $(\xi,\eta,\eta) \in C$, we deduce that 
$$
\operatorname{dim}(C)
=
\operatorname{dim}\bigl(S(\partial{D}_j)\bigr)
+
\operatorname{dim}\bigl(N^\ast_{\pi(v)}(\partial{D_j})\bigr)
=
(2n-3)+1
=
2n-2.
$$
The excess of $C$ is computed as  
\begin{align*}
  e
& =
  \operatorname{codim}(\mathcal{D})
  +
  \operatorname{codim}\bigl(C_\mathcal{X}{\times}N^\ast(\partial{D_j})\bigr)
  -
  \operatorname{codim}(C)
\\
& =
  \operatorname{codim}(\mathcal{D})
  -
  \operatorname{dim}\bigl(C_\mathcal{X}{\times}N^\ast(\partial{D_j})\bigr)
  +
  \operatorname{dim}(C)
\\
& =
  2n-(4n-2)+(2n-2)
  =
  0. 
\end{align*} 
If we show that $C$ is a clean intersection, then it follows that $C_\mathcal{X}{\circ}N^\ast(\partial{D_j})$  is transversal. 
Fix arbitrary $c=(\xi,\eta,\eta) \in C$, and pick up arbitrary 
$X=(X_1,X_2,X_3) \in T_c\bigl(C_\mathcal{X}{\times}N^\ast(\partial{D_j})\bigr)$. Consider a smooth curve on 
$C_\mathcal{X}{\times}N^\ast(\partial{D_j})$ of the form 
\begin{align*}
  \mathbb{R}\ni\alpha 
& \mapsto 
  \bigl(\xi(\alpha),\eta(\alpha),\zeta(\alpha)\bigr)
  =
  c+\alpha{X}+\mathcal{O}(\alpha^2) 
\\
& =
  (\xi+\alpha{X}_1,\eta+\alpha{X}_2,\eta+\alpha{X}_3)+\mathcal{O}(\alpha^2)
\end{align*}
near $\alpha=0$. Then we have 
$$
X
=
\frac{d}{d\alpha}\bigg\vert_{\alpha=0}\bigl(\xi(\alpha),\eta(\alpha),\zeta(\alpha)\bigr)
$$
If $X \in T_c(\mathcal{D})$, then we have 
$$
\frac{d}{d\alpha}\bigg\vert_{\alpha=0}\eta(\alpha)
=
\frac{d}{d\alpha}\bigg\vert_{\alpha=0}\zeta(\alpha),
$$
and 
$$
X=
\frac{d}{d\alpha}\bigg\vert_{\alpha=0}\bigl(\xi(\alpha),\eta(\alpha),\zeta(\alpha)\bigr)
=
\frac{d}{d\alpha}\bigg\vert_{\alpha=0}\bigl(\xi(\alpha),\eta(\alpha),\eta(\alpha)\bigr)
\in 
T_c(C). 
$$
Hence we obtain $T_c(\bigl(C_\mathcal{X}{\times}N^\ast(\partial{D_j})\bigr)){\cap}T_c(\mathcal{D}){\subset}T_c(C)$. 
Therefore we conclude that 
$$
T_c(\bigl(C_\mathcal{X}{\times}N^\ast(\partial{D_j})\bigr)){\cap}T_c(\mathcal{D})=T_c(C)
$$
and that $C$ is clean intersection. 
Consider the projection 
$\pi_C:C\ni(\xi,\eta,\eta) \mapsto \xi \in T^\ast\bigl(\partial_-S(M)\bigr)$.  
It follows that $C$ is proper and $\pi_C^{-1}(\xi)$ for $\xi \in \pi_C(C)$ is connected since $\xi$ is uniquely determined by 
$\xi=(DF\vert^T_v)^{-1}{\circ}D\pi\vert^T_v$ for 
$\eta \in N^\ast_{\pi(v)}(\partial{D_j})$.    
\par
Next we show that $N^\ast(\Sigma_j)=C_\mathcal{X}{\circ}N^\ast(\partial{D_j})$. 
Using the expression \eqref{equation:intersection4-1} of $C$, we deduce that 
\begin{align}
  C_\mathcal{X}{\circ}N^\ast(\partial{D_j})
& =
  \bigl\{
  \xi: 
  \exists v \in S({\partial}D_j), \exists \eta \in N^\ast_{\pi(v)}(\partial{D_j})\setminus\{0\}, \text{s.t.}\
\nonumber
\\
& \qquad   
  \xi \in T^\ast_{F(v)}\bigl(\partial_-S(M)\bigr)\setminus\{0\},
  DF\vert^T_v\xi=D\pi\vert^T_v\eta\bigr\},
\label{equation:nstarsigmaj}
\\
  \Sigma_j
& :=
  \pi_{\partial_-S(M)}\bigl(C_\mathcal{X}{\circ}N^\ast(\partial{D_j})\bigr)
  =
  \{F(v) : v \in S(\partial{D_j})\}.
\nonumber 
\end{align}
Then we have $\operatorname{dim}(\Sigma_j)=\operatorname{dim}\bigl(S(\partial{D_j})\bigr)=2n-3$. 
Therefore we deduce that 
$\operatorname{codim}(\Sigma_j)=1$ in $\partial_-S(M)$ 
and that $\Sigma_j$ is a hypersurface in $\partial_-S(M)$. 
Since $C_\mathcal{X}$ is a homogeneous canonical relation and 
$N^\ast(\partial{D_j})$ is a conic Lagrangian submanifold, 
we conclude that $C_\mathcal{X}{\circ}N^\ast(\partial{D_j})$ is also a conic Lagrangian submanifold 
and that $N^\ast(\Sigma_j)=C_\mathcal{X}{\circ}N^\ast(\partial{D_j})$. 
This completes the proof. 
\end{proof}
Next we study the intersection 
$\Sigma_{j}\cap\Sigma_k$ for $j{\ne}k$. 
Suppose that $\Sigma_{jk}:=\Sigma_{j}\cap\Sigma_k\ne\emptyset$. Then we have 
\begin{align*}
  \Sigma_{jk}
& =
  \{
  F(v)
  : 
  v \in S(\partial{D_j}),\ 
  \exists \tilde{v} \in S(\partial{D_k})\ \text{s.t.}\ F(v)=F(\tilde{v})
  \}
\\
& =
  \bigl\{
  w \in \partial_-S(M)
  :
  \exists v \in S(\partial{D_j}), 
  \exists s \in \bigl(\tau_-(v),\tau_+(v)\bigr)\setminus\{0\}
  \ \text{s.t.}\ 
\\
& \qquad
  w=F(v), 
  \dot{\gamma}_v(s) \in S(\partial{D_k})
  \bigr\}.  
\end{align*}
We have the following. 
\begin{lemma}
\label{theorem:4-3}
Suppose that $\Sigma_{jk}:=\Sigma_{j}\cap\Sigma_k\ne\emptyset$ with $j{\ne}k$. 
Then  $\Sigma_{jk}$ is a transversal intersection. 
\end{lemma}
\begin{proof}
Fix arbitrary $w \in \Sigma_{jk}$. Then there exist 
$v \in S(\partial{D_j})$ and $s\in\bigl(\tau_-(v),\tau_+(v)\bigr)$ such that 
$w=F(v)$ and $\Psi(v,s)=\dot{\gamma}_v(s) \in S(\partial{D_k})$.  Then \eqref{equation:etav} implies that 
$\eta(v)=0$ for any $\eta \in N^\ast_{\pi(v)}(\partial{D_j})$ and 
$\tilde{\eta}\bigl(\dot{\gamma}_v(s)\bigr)=0$ for any $\tilde{\eta} \in N^\ast_{\gamma_v(s)}(\partial{D_k})$. 
We shall prove that $N^\ast_w(\Sigma_j){\cap}N^\ast_w(\Sigma_k)=\{0\}$.
\par
Fix arbitrary $\xi \in N^\ast_w(\Sigma_j)$ and arbitrary $\tilde{\xi} \in N^\ast_w(\Sigma_k)$.  
Since $N^\ast(\Sigma_l)=C_\mathcal{X}{\circ}N^\ast(\partial{D_l})$ ($l=1,\dotsc,J$), 
there exist $\eta \in N^\ast_{\pi(v)}(\partial{D_j})$ and $\tilde{\eta} \in N^\ast_{\gamma_v(s)}(\partial{D_k})$ 
uniquely such that 
$$
DF\vert^T_v\xi=D\pi\vert^T_v\eta,
\quad
DF\vert^T_{\Psi(v,s)}\tilde{\xi}=D\pi\vert^T_{\Psi(v,s)}\tilde{\eta}.
$$
Since $w=F(v)=F\bigl(\Psi(v,s)\bigr)$, we have 
$$
DF\vert_v\Xi=DF\vert_{\Psi(v,s)}{\circ}D_v\Psi\vert_{(v,s)}\Xi
$$
for any $\Xi \in T_v\bigl(S(M^\text{int})\bigr)$. This implies that
\begin{equation}
DF\vert^T_v\tilde{\xi}
=
D_v\Psi\vert^T_{(v,s)}{\circ}D\pi\vert^T_{\Psi(v,s)}\tilde{\eta}. 
\label{equation:tildexitildeeta} 
\end{equation}
Then we deduce that for any $\Xi \in T_v(S(M^\text{int}))$  
\begin{align}
  (\xi+\tilde{\xi})(DF\vert_v\Xi)
& =
  DF\vert^T_v\xi(\Xi)+DF\vert^T_v\tilde{\xi}(\Xi)
\nonumber
\\
& =
  D\pi\vert^T_v\eta(\Xi)+D_v\Psi\vert^T_{(v,s)}{\circ}D\pi\vert^T_{\Psi(v,s)}\tilde{\eta}(\Xi)
\nonumber
\\
& =
  \eta(D\pi\vert_v\Xi)+\tilde{\eta}(D\pi\vert_{\Psi(v,s)}{\circ}D_v\Psi\vert_{(v,s)}\Xi).
\label{equation:xitildexi}
\end{align}
Since $\gamma_v(s)$ is not a conjugate point of $\pi(v)=\gamma_v(0)$, we deduce that 
$$
\{
D\pi\vert_{\Psi(v,s)}{\circ}D_v\Psi\vert_{(v,s)}\Xi: 
\Xi \in \operatorname{Ker}(D\pi\vert_v)
\}
=
\{
\tilde{u} \in T_{\gamma_v(s)}(M^\text{int}): 
\langle{\tilde{u},\dot{\gamma}_v(s)}\rangle=0
\}. 
$$
Here we used Lemma~\ref{theorem:gauss}. 
\par
We now suppose that $\xi+\tilde{\xi}=0$. 
We consider \eqref{equation:xitildexi}. 
On one hand  we have $\tilde{\eta}(\tilde{u})=0$ 
for any $\tilde{u} \in T_{\gamma_v(s)}(M^\text{int})$ satisfying $\langle{\tilde{u},\dot{\gamma}_v(s)}\rangle=0$. 
On the other hand $\tilde{\eta}\bigl(\dot{\gamma}_v(s)\bigr)=0$. 
Then we have $\tilde{\eta}\bigl(\tilde{u}+\sigma\dot{\gamma}_v(s)\bigr)=0$ 
for any $\sigma\in\mathbb{R}$ and 
for any $\tilde{u} \in T_{\gamma_v(s)}(M^\text{int})$ satisfying $\langle{\tilde{u},\dot{\gamma}_v(s)}\rangle=0$. 
This means that $\tilde{\eta}=0$ in $T^\ast_{\pi\bigl(\dot{\gamma}_v(s)\bigr)}(M^\text{int})$. 
Hence \eqref{equation:xitildexi} becomes $\eta(D\pi\vert_v\Xi)=0$ for any $\Xi \in T_v\bigl(S(M^\text{int})\bigr)$. 
Therefore we have $\eta=0$ in $T^\ast_{\pi(v)}(M^\text{int})$. 
We deduce that if $\xi+\tilde{\xi}=0$, then $\xi=0$ and $\tilde{\xi}=0$ 
since $\xi$ and $\tilde{\xi}$ are uniquely determined by $\eta$ and $\tilde{\eta}$ respectively. 
Thus $\xi \in N^\ast_w(\Sigma_j)\setminus\{0\}$ and $\tilde{\xi} \in N^\ast_w(\Sigma_k)\setminus\{0\}$ 
are linearly independent, and $N^\ast_w(\Sigma_j){\cap}N^\ast_w(\Sigma_k)=\{0\}$. This completes the proof. 
\end{proof}
Lemma~\ref{theorem:4-3} implies that $\Sigma_{jk}$ is a submanifold of $\partial_-S(M)$ with 
$\operatorname{codim}(\Sigma_{jk})=2$, and 
$$
N^\ast(\Sigma_{jk})
=
\bigcup_{w \in \Sigma_{jk}}
N^\ast_w(\Sigma_j) \oplus N^\ast_w(\Sigma_k). 
$$
In other words, the proof of Lemma~\ref{theorem:4-3} shows that all the Jacobi fields along a common tangent geodesic $\gamma$ of $\partial{D}_j$ and $\partial{D_k}$ collaborate to distinguish $\gamma{\cap}\partial{D_j}$ and $\gamma{\cap}\partial{D_k}$ which are the locations of sources of conormal singularities. The author could not understand this fact for the Euclidean case. See \cite[Lemma~4.2]{Chihara} and its proof. 
\par
We have shown that $N_w^\ast(\Sigma_{jk})$ is a two-dimensional vector space for any $w \in \Sigma_{jk} \subset \partial_-S(M)$, 
and we need to study what $C_{\mathcal{X}^T}{\circ}N^\ast(\Sigma_{jk})$ is. 
We denote by $\nu_j(x)$ the unit outer normal vector for $D_j$ at $x \in \partial{D_j}$, and set 
$\eta_j(x):=\langle\nu_j(x),\cdot\rangle=g_x(\nu_j,\cdot) \in N^\ast_x(\partial{D_j})$ for $x \in \partial{D_j}$. 
\par
If $w \in \Sigma_{jk}$, 
then there exist $v \in S(\partial{D_j})$ and $s\in\bigl(\tau_-(v),\tau_+(v)\bigr)\setminus\{0\}$ 
such that 
$w=F(v)$ and $\tilde{v}:=\Psi(v,s)=\dot{\gamma}_v(s) \in S(\partial{D_k})$. 
There are two cases:
\begin{itemize}
\item 
{\bf Case~C1}.\ 
$P(0,s;\gamma_v)^T\eta_k\bigl(\gamma_v(s)\bigr)=\eta_j\bigl(\pi(v)\bigr)$
or 
$-\eta_j\bigl(\pi(v)\bigr)$. 
\item 
{\bf Case~C2}.\ 
$P(0,s;\gamma_v)^T\eta_k\bigl(\gamma_v(s)\bigr){\ne}\pm\eta_j\bigl(\pi(v)\bigr)$. 
\end{itemize}
Replace $\eta$ and $\tilde{\eta}$ by $\eta_j\bigl(\pi(v)\bigr)$ and $\eta_k\bigl(\gamma_v(s)\bigr)$ respectively, 
and see Figure~2 in Section~1. 
We remark that when $n=2$, Case~C2 never occurs. 
When $n\geqq3$ we split $\Sigma_{jk}$ into two parts:
\begin{align*}
  \Sigma_{jk}^{(1)}
& :=
  \bigl\{
  F(v): 
  \exists v \in S(\partial{D_j}), 
  \exists s \in \bigl(\tau_-(v),\tau_+(v)\bigr)\setminus\{0\}
  \ \text{s.t.}\
\\ 
& \qquad
  \Psi(v,s)=\dot{\gamma}_v(s) \in S(\partial{D_k}), 
\\ 
& \qquad
  P(0,s;\gamma_v)^T\eta_k\bigl(\gamma_v(s)\bigr) = \eta_j\bigl(\pi(v)\bigr)
  \ \text{or}\ 
  -\eta_j\bigl(\pi(v)\bigr)
  \bigr\}, 
\\
  \Sigma_{jk}^{(2)}
& :=
  \bigl\{
  F(v): 
  \exists v \in S(\partial{D_j}), 
  \exists s \in \bigl(\tau_-(v),\tau_+(v)\bigr)\setminus\{0\}
  \ \text{s.t.}\
\\ 
& \qquad
  \Psi(v,s)=\dot{\gamma}_v(s) \in S(\partial{D_k}), 
  P(0,s;\gamma_v)^T\eta_k\bigl(\gamma_v(s)\bigr)\ne\pm\eta_j\bigl(\pi(v)\bigr)
  \bigr\}.
\end{align*}

\par
We now study the strictly convexity of $D_j$ to make full use of this property. 
Pick up arbitrary $x_0 \in \partial{D_j}$. 
Then there exist 
a neighborhood $U$ at $x_0$ in $M^\text{int}$ 
and $f_j \in C^\infty(U)$ such that 
\begin{align*}
  U{\cap}D_j
& =
  \{x{\in}U : f_j(x)<0\},
\\
  U{\cap}\partial{D_j}
& =
  \{x{\in}U : f_j(x)=0\},
\\
  U\setminus\overline{D_j}
& =
  \{x{\in}U : f_j(x)>0\}.
\end{align*}
Then for any $v \in S_{x_0}(\partial{D_j})$, we have 
$$
f_j\bigl(\gamma_v(0)\bigr)=f(x_0)=0,
\qquad
f_j\bigl(\gamma_v(t)\bigr)>0
\quad
(t\ne0)
$$
since $D_j$ is strictly convex. 
The function $f_j\bigl(\gamma_v(t)\bigr)$ is smooth near $t=0$ and $f_j\bigl(\gamma_v(0)\bigr)=0$ 
is the unique local minimum of $f_j\bigl(\gamma_v(t)\bigr)$. Hence we have 
$$
\frac{d}{dt}\bigg\vert_{t=0}f_j\bigl(\gamma_v(t)\bigr)=0,
$$
which is 
$$
0
=
df_j\bigl(\dot{\gamma}_v(0)\bigr)
=
df_j(v)
=
\langle\operatorname{grad}f_j,v\rangle
=
g_{x_0}(\operatorname{grad}f_j,v).
$$
The Hessian $\nabla{df_j}$ is given by 
\begin{align*}
  \nabla{df_j}(X,Y)
& =
  X\bigl(df_j(Y)\bigr)-df_j(\nabla_XY)
\\
& =
  X\bigl(Y(f_j)\bigr)-\nabla_XY(f_j)
\\
& =
  \langle\nabla_X\operatorname{grad}f_j,Y\rangle  
\end{align*}
for any vector fields $X$ and $Y$ in $U$. 
If $\gamma$ is a geodesic, then 
$$
\frac{d^2}{dt^2}
f\bigl(\gamma(t)\bigr)
=
\nabla{df_j}\bigl(\dot{\gamma}(t),\dot{\gamma}(t)\bigr)
+
df_j\bigl(\nabla_{\dot{\gamma}(t)}\dot{\gamma}(t)\bigr)
=
\nabla{df_j}\bigl(\dot{\gamma}(t),\dot{\gamma}(t)\bigr). 
$$
Then we have 
\begin{align*}
  f_j\bigl(\gamma_v(t)\bigr)
& =
  f_j\bigl(\gamma_v(0)\bigr)
  +
  t
  \frac{d}{dt}\bigg\vert_{t=0}
  f_j\bigl(\gamma_v(t)\bigr)
  +
  \frac{t^2}{2}
  \frac{d^2}{dt^2}\bigg\vert_{t=0}
  f_j\bigl(\gamma_v(t)\bigr)
  +
  \mathcal{O}(t^3)
\\
& =
  \frac{t^2}{2}\nabla{df_j}(v,v)+\mathcal{O}(t^3)
  >0 
\end{align*}
for any $v \in S_{x_0}(\partial{D_j})$ 
and $t\in\mathbb{R}$ with $0<\lvert{t}\rvert\ll1$. 
Since $x_0 \in U{\cap}\partial{D_j}$ is arbitrary, we obtain
\begin{equation}
\nabla{df_j}(v,v)>0
\quad\text{for}\quad
v \in S(U{\cap}\partial{D_j}).
\label{equation:hessian} 
\end{equation} 
\par
We may assume that $\operatorname{grad}f_j\ne0$ in $U$ 
provided that $U$ is sufficiently small, 
and we can define $f_j/\lvert\operatorname{grad}{f_j}\rvert$ on $U$, 
where 
$\lvert\operatorname{grad}{f_j}\rvert^2=\langle\operatorname{grad}{f_j},\operatorname{grad}{f_j}\rangle$. 
Since 
$$
\operatorname{grad}
\left(\frac{f_j}{\lvert\operatorname{grad}{f_j}\rvert}\right)
=
\frac{\operatorname{grad}f_j}{\lvert\operatorname{grad}{f_j}\rvert}
+
f_j
\cdot
\operatorname{grad}
\left(\frac{1}{\lvert\operatorname{grad}{f_j}\rvert}\right)
$$
in $U$, we have 
$$
\operatorname{grad}_x
\left(\frac{f_j}{\lvert\operatorname{grad}{f_j}\rvert}\right)
=
\frac{\operatorname{grad}_xf_j}{\lvert\operatorname{grad}_x{f_j}\rvert}
=
\nu_j(x),
\quad
x \in U{\cap}\partial{D_j}. 
$$ 
If we replace $f_j$ by $f_j/\lvert\operatorname{grad}f_j\rvert$ and use the same notation $f_j$, we have 
$$
\operatorname{grad}_xf_j=\nu_j(x),
\quad
x \in U\cap\partial{D_j}.
$$
In what follows we choose $f_j$ ($j=1,\dotsc,J$) satisfying this property. 
Then $\lvert\operatorname{grad}f_j\rvert^2$ is a smooth function on $U$ 
and identically equal to one on $U{\cap}\partial{D_j}$. Moreover we have the following. 
\begin{lemma}
\label{theorem:gradient}
Suppose that there exist a neighborhood $U$ at 
$x_0 \in \partial{D_j}$ in $M^\text{int}$ and a function $f_j \in C^\infty(U)$ such that
$$
U{\cap}\partial{D_j}
=
\{x \in U : f_j(x)=0\},
\quad
\operatorname{grad}_xf_j=\nu_j(x)
\quad
(x \in U{\cap}\partial{D_j}).
$$
If we set $G_j(x):=\lvert\operatorname{grad}_xf_j\rvert^2$, 
then we have 
\begin{equation}
\frac{d}{ds}\bigg\vert_{s=0}
G_j\bigl(\gamma_u(s)\bigr)
=0
\label{equation:gradient}
\end{equation}
for any $u \in S_x(\partial{D_j})$ and $x \in U{\cap}\partial{D_j}$.   
\end{lemma}
\begin{proof} 
We have only to show \eqref{equation:gradient}. 
Fix arbitrary $x \in U{\cap}\partial{D_j}$ and $u \in S_x(\partial{D_j})$. 
Let $y(s)$ be a curve on $\partial{D_j}$ near $x$ such that $y(0)=x$ and $y^\prime(0)=u$.  
Then we have $G_j\bigl(y(s)\bigr)\equiv1$, 
$$
\gamma_u(s)=x+su+\mathcal{O}(s^2), 
\quad
y(s)=x+su+\mathcal{O}(s^2), 
\quad
\gamma_u(s)-y(s)=\mathcal{O}(s^2)
$$
near $s=0$. The mean value theorem for $G_j$ implies that 
$$
G_j\bigl(\gamma_u(s)\bigr)
=
G_j\bigl(y(s)\bigr)+\mathcal{O}(s^2)
=
1+\mathcal{O}(s^2). 
$$
Thus we can obtain \eqref{equation:gradient} immediately. 
\end{proof}
\par 
We now study the relationship between $\Sigma_{jk}^{(l)}$ ($l=1,2$), 
$\partial{D_j}$ and $\partial{D_k}$ 
to compute $C_\mathcal{X}^T{\circ}N^\ast(\Sigma_{jk}^{(l)})$. 
Fix arbitrary $x_0 \in \partial{D_j}$. Consider the hyperplane
$$
H_j(x_0)
:=
\bigl\{
\gamma_u(t)
: 
u \in S_{x_0}(\partial{D_j}), 
t\in\bigl(\tau_-(u),\tau_+(u)\bigr)
\bigr\}. 
$$
Consider the relationship between $H_j(x_0)$ and $\partial{D_k}$. 
Logically the following three cases occur: 
\begin{itemize}
\item
Case~D1: 
There exists $y_0 \in \partial{D_k}$ uniquely such that $H_j(x_0)=H_k(y_0)$. 
See Case~B1 of Figure~2. 
In this case there exists $v \in S_{x_0}(\partial{D_j})$ and 
$s \in \bigl(\tau_-(v),\tau_+(v)\bigr)$ such that 
$\dot{\gamma}_v(s) \in S_{y_0}(\partial{D_k})$ and $\nu_k(y_0)={\pm}P(0,s;\gamma_v)\nu_j(x_0)$. 
See Figure~3. 
This case corresponds to the case of $F(v) \in \Sigma_{jk}^{(1)}$. 
\vspace{6pt}
\item 
Case~D2: 
$H(x_0){\cap}D_k\ne\emptyset$. 
Note that this implies that $H_j(x_0)$ is never a tangent hyperplane 
for $\partial{D_k}$. See Case~B2 of Figure~2. 
In this case there exist two common tangential geodesics for $\partial{D_j}$ and $\partial{D_k}$, 
and the normal directions at the common tangential points are different, 
which means that the two normal directions are not parallel. 
Then we deduce that 

$$
y_1, y_2 \in \partial{D_k},
\quad 
v_1, v_2 \in S_{x_0}(\partial{D_j}),
\quad 
s_1 \in \bigl(\tau_-(v_1),\tau_+(v_1)\bigr)
\quad
s_2 \in \bigl(\tau_-(v_2),\tau_+(v_2)\bigr)
$$
such that 
$$
y_1{\ne}y_2,
\quad 
\dot{\gamma}_{v_1}(s_1) \in S_{y_1}(\partial{D_k}),
\quad
\dot{\gamma}_{v_2}(s_2) \in S_{y_2}(\partial{D_k}),
$$
$$
\nu_k(y_1)\ne{\pm}P(0,s_1;\gamma_{v_1})\nu_j(x_0),
\quad
\nu_k(y_2)\ne{\pm}P(0,s_2;\gamma_{v_2})\nu_j(x_0),
$$
These cases correspond to the cases of 
$F(v_1), F(v_2) \in \Sigma_{jk}^{(2)}$. 
\vspace{6pt}
\item 
Case~D3: 
$H_j(x_0){\cap}\overline{D_k}=\emptyset$. 
\end{itemize}
Case~D2 is the intermediate case between 
$\nu_k(y_0)=+P(0,s;\gamma_v)\nu_j(x_0)$ and 
$\nu_k(y_0)=-P(0,s;\gamma_v)\nu_j(x_0)$ in Case~D1 
from a view point of the location of $x_0$ on $\partial{D_j}$. 
So we shall study Case~D1. 
\par
We now introduce a subset of $S(\partial{D_j})$ of the form   
\begin{align*}
  \mathcal{M}_{jk}^{(\pm)}
& :=
  \bigl\{
  v \in S(\partial{D_j})
  : 
  \exists s \bigl(\tau_-(v),\tau_+(v)\bigr)
  \ \text{s.t.}\ 
\\
& \qquad
  \dot{\gamma}_v(s) \in S(\partial{D_k}), 
  \nu_k\bigl(\gamma_v(s)\bigr)
  =
  \pm
  P(0,s;\gamma_v)\nu_j\bigl(\pi(v)\bigr)
  \bigr\}. 
\end{align*}
Since $D_j$ and $D_k$ are strictly convex, 
we deduce that for any $x_0$ in Case~D1, 
$v \in S_{x_0}(\partial{D_j})$, 
$s \in \bigl(\tau_-(v),\tau_+(v)\bigr)$, 
$y_0=\gamma_v(s)$, 
and the signature $+$ or $-$ are determined uniquely. 
When $n=2$ and $\Sigma_{jk}\ne\emptyset$, 
$\mathcal{M}_{jk}^{(\pm)}$ become sets of two elements respectively. 
Furthermore the properties of $\mathcal{M}_{jk}^{(\pm)}$ 
for $n\geqq3$ are the following. 
\begin{lemma}
\label{theorem:mjk} 
Suppose {\rm ({\bf A0})}, {\rm ({\bf A1})} and {\rm ({\bf A2})}. 
When $n\geqq3$ and $\Sigma_{jk}^{(1)}\ne\emptyset$, 
$\mathcal{M}_{jk}^{(\pm)}$ are connected $(n-2)$-dimensional 
submanifolds of $S(\partial{D_j})$. 
\end{lemma}
\begin{proof}
Suppose $n\geqq3$ and $\Sigma_{jk}^{(1)}\ne\emptyset$. 
We shall show that $\mathcal{M}_{jk}^{(+)}$ is an $(n-2)$-dimensional 
submanifold of $S(\partial{D_j})$. 
Similarly we can show that so is $\mathcal{M}_{jk}^{(-)}$, and we omit the proof of this.
Consider the common tangential hyperplanes for $\partial{D_j}$ and $\partial{D_k}$, 
and move the hypersurface while maintaining tangency. 
The hypersurface comes back to the initial position 
since $\partial{D_j}$ and $\partial{D_k}$ are strictly convex. 
This implies that $\mathcal{M}_{jk}^{(+)}$ is connected.   
Fix arbitrary $v_0 \in \mathcal{M}_{jk}^{(+)}$. 
Then there exists $s_0 \in \bigl(\tau_-(v_0),\tau_+(v_0)\bigr)$ uniquely such that 
$$
u_0:=\dot\gamma_{v_0}(s_0)=\Psi(v_0,s_0) \in S(\partial{D_k}),
\quad
\nu_k\bigl(\pi(u_0)\bigr)=P(0,s_0,\gamma_{v_0})\nu_j\bigl(\pi(v_0)\bigr). 
$$
Set $x_0:=\pi(v_0)$ and $y_0:=\pi(u_0)$ for short. 
Pick up a neighborhood $U_j$ at $x_0$ in $M^\text{int}$, 
a neighborhood $U_k$ at $y_0$ in $M^\text{int}$, 
a smooth function $f_j \in C^\infty(U_j)$, 
and a smooth function $f_k \in C^\infty(U_k)$ such that 
\begin{alignat*}{2}
  \partial{D_j}{\cap}U_j
& =
  \{x \in U_j : f_j(x)=0\},
& \quad
  \nu_j(x)
& =
  \operatorname{grad}_xf_j
  \quad\text{for}\quad
  x \in \partial{D_j}{\cap}U_j,
\\
  \partial{D_k}{\cap}U_k
& =
  \{y \in U_k : f_k(y)=0\},
& \quad
  \nu_k(y)
& =
  \operatorname{grad}_yf_k
  \quad\text{for}\quad
  y \in \partial{D_k}{\cap}U_k.
\end{alignat*}
Then we have 
\begin{align*}
  S(\partial{D_j}{\cap}U_j)
& =
  \bigl\{
  v \in S(U_j)
  :
  f_j\bigl(\pi(v)\bigr)=0, 
  df_j(v)=0 
  \bigr\},
\\
  S(\partial{D_k}{\cap}U_k)
& =
  \bigl\{
  u \in S(U_k)
  :
  f_k\bigl(\pi(u)\bigr)=0, 
  df_k(u)=0 
  \bigr\}. 
\end{align*}
For $v \in S(\partial{D_j}{\cap}U_j)$ near $v_0$ and 
for $s \in \bigl(\tau_-(v),\tau_+(v)\bigr)$ near $s_0$, we set
\begin{align*}
  J(s,v)
& =
  \begin{bmatrix}
  f_k\bigl(\gamma_v(s)\bigr)
  \\
  df_k\bigl(\dot{\gamma}_v(s)\bigr)
  \\
  \operatorname{grad}_{\gamma_v(s)}f_k - P(0,s;\gamma_v)\operatorname{grad}_{\pi(v)}f_j 
  \end{bmatrix} 
\\
& =
  \begin{bmatrix}
  f_k\Bigl(\pi\bigl(\Psi(v,s)\bigr)\Bigr)
  \\
  df_k\bigl(\Psi(v,s)\bigr)
  \\
  \operatorname{grad}_{\pi\bigl(\Psi(v,s)\bigr)}f_k - P(0,s;\gamma_v)\operatorname{grad}_{\pi(v)}f_j 
  \end{bmatrix}.
\end{align*} 
The function $J(s,v)$ takes values in 
$\mathbb{R}^2 \times T_{\pi\bigl(\Psi(v,s)\bigr)}(M^\text{int})$, 
and $\mathcal{M}_{jk}^{(+)}$ is characterized by $J(s,v)=0$ near $(s_0,v_0)$. 
Indeed if $v \in \mathcal{M}_{jk}^{(+)}$, 
then there exists $s\in\bigl(\tau_-(v),\tau_+(v)\bigr)$ uniquely such that 
$J(s,v)=0$. Conversely if $J(s,v)=0$, then $v$ satisfies the condition for $v \in \mathcal{M}_{jk}^{(+)}$ with $s$. 
Hence it suffices to show that $\operatorname{rank}\bigl(DJ\vert_{(s_0,v_0)}\bigr)=n$ 
to prove that $\mathcal{M}_{jk}^{(+)}$ is an $(n-2)$-dimensional submanifold of $S(\partial{D_j}{\cap}U_j)$ 
since $(s,v)$ moves in the $(2n-2)$-dimensional manifold $\mathbb{R}{\times}S(\partial{D_j})$ 
and $s$ is uniquely determined by $v$. 
\par
To compute the differential $DJ\vert_{(s_0,v_0)}$, we now introduce a coordinate system in 
$T_{v_0}\bigl(S(\partial{D_j})\bigr)$. Roughly speaking, we can see this as 
\begin{align*}
  T_{v_0}\bigl(S(\partial{D_j})\bigr)
& =
  T_{x_0}(\partial{D_j})
  \times
  T_{v_0}\bigl(S_{x_0}(\partial{D_j})\bigr)
\\
& \simeq
  T_{x_0}(\partial{D_j})
  \times
  \{
  \zeta \in T_{x_0}(\partial{D_j})
  : 
  \langle{\zeta,v_0}\rangle=0
  \}.
\end{align*}
Pick up an orthonormal basis $\{v_0,v_1,\dotsc,v_{n-2}\}$ 
of $T_{x_0}(\partial{D_j})$, 
and an orthonormal basis  $\{\zeta_1,\dotsc,\zeta_{n-2}\}$ 
of $T_{v_0}\bigl(S_{x_0}(\partial{D_j})\bigr)$. 
Then 
$$
\{(v_0,0), (v_1,0), \dotsc, (v_{n-2},0), (0,\zeta_1), \dotsc, (0,\zeta_{n-2})\}
$$ 
becomes an orthonormal basis of $T_{v_0}\bigl(S(\partial{D_j})\bigr)$. 
Note that $\langle\zeta_j,v_0\rangle=0$ for $j=1,\dotsc,n-2$. 
Fix arbitrary $(X,\zeta) \in T_{v_0}\bigl(S(\partial{D_j})\bigr)$. 
We consider a curve $v(\sigma)$ on $S(\partial{D_j})$ of the form 
$$
\mathbb{R}\ni\sigma
\mapsto 
v(\sigma)
:=
v_0+\sigma(X,\zeta)+\mathcal{O}(\sigma^2)
$$
near $\sigma=0$. Let $Y(s)$ be the Jacobi field along $\gamma_{v_0}$ 
with $\bigl(Y(0),\nabla{Y(0)}\bigr)=(X,\zeta)$. Then we have 
\begin{align*}
  \bigl(Y(s),\nabla{Y(s)}\bigr)
& =
  \frac{d}{d\sigma}\bigg\vert_{\sigma=0}
  \Bigl(
  \pi\bigl(\Psi(v(\sigma),s)\bigr), \Psi(v(\sigma),s)
  \Bigr)
\\
& =
  \Bigl(
  D\pi\vert_{\Psi(v_0,s)}
  \circ
  D_v\Psi\vert_{(v_0,s)}(X,\zeta),
  D_v\Psi\vert_{(v_0,s)}(X,\zeta),
  \Bigr). 
\end{align*}
The assumption ({\bf A1}) ensures that 
$$
Y(s)
=
a(s)P(0,s;\gamma_{v_0})X,
\quad
\nabla{Y}(s)
=
b(s)P(0,s;\gamma_{v_0})\zeta,
$$
where $a(s)$ and $b(s)$ are solutions to 
\begin{alignat*}{3}
  a^{\prime\prime}(s)+k\bigl(\gamma_{v_0}(s)\bigr)a(s)
& =
  0,
& \quad
  a(0)
& =
  1,
& \quad
  a^\prime(0) 
& =
  0,
\\
  b^{\prime\prime}(s)+k\bigl(\gamma_{v_0}(s)\bigr)b(s)
& =
  0,
& \quad
  b(0)
& =
  0,
& \quad
  b^\prime(0) 
& =
  1. 
\end{alignat*}
Hence we have $Y(s), \nabla{Y(s)} \in T_{\gamma_{v_0}(s)}\bigl(H_j(x_0)\bigr)$, 
in particular, $Y(s_0), \nabla{Y(s_0)} \in T_{y_0}(\partial{D_k})$. 
We now introduce Jacobi fields along $\gamma_{v_0}$ by 
\begin{align*}
  Y_p(s)
& =
  D\pi\vert_{\Psi(v_0,s)}{\circ}D_v\Psi\vert_{(v_0,s)}(v_p,0),
  \quad
  p=0,1,\dotsc,n-2,
\\
  Z_q(s)
& =
  D\pi\vert_{\Psi(v_0,s)}{\circ}D_v\Psi\vert_{(v_0,s)}(0,\zeta_q),
  \quad
  q=1,\dotsc,n-2. 
\end{align*}
Obviously we have $Y_0(s)=\dot{\gamma}_{v_0}(s)$ and $Z_q(0)=0$ for $q=1,\dotsc,n-2$. 
Since $y_0=\gamma_{v_0}(s_0)$ is not a conjugate point of $x_0$, 
$$
\{u_0,Y_1(s_0),\dotsc,Y_{n-2}(s_0),Z_1(s_0),\dotsc,Z_{n-2}(s_0)\}
$$
becomes a basis of $T_{u_0}\bigl(S(\partial{D_k})\bigr)$.
\par
We now start to compute $DJ\vert_{(s_0,v_0)}$. 
Let $A_j$ be the shape operator for $\partial{D_j}$ at $x_0$, 
and let $A_k$ be the shape operator for $\partial{D_k}$ at $y_0$. 
It follows that $\operatorname{rank}(A_j)=n-1$ and 
$\operatorname{rank}(A_k)=n-1$ since $D_j$ and $D_k$ are strictly convex respectively. 
\par
First we compute $\partial J/\partial s$. 
Let $\Gamma^p_{qr}(y)$ be the Christoffel symbol of $(M,g)$ in $U_k$. 
Since $P(0,s;\gamma_{v_0})$ is the operator giving the parallel transport along $\gamma_{v_0}$, 
we deduce that 
\begin{align*}
  \frac{\partial J}{\partial s}(s_0,v_0)
& =
  \frac{\partial}{\partial s}\bigg\vert_{s=s_0}
  \begin{bmatrix}
  f_k\Bigl(\pi\bigl(\Psi(v_0,s)\bigr)\Bigr)
  \\
  df_k\bigl(\Psi(v_0,s)\bigr) 
  \\
  \operatorname{grad}_{\pi\bigl(\Psi(v_0,s)\bigr)}f_k
  -
  P(0,s;\gamma_{v_0})\operatorname{grad}_{x_0}f_j
  \end{bmatrix} 
\\
& =
  \frac{\partial}{\partial s}\bigg\vert_{s=s_0}
  \begin{bmatrix}
  f_k\Bigl(\pi\bigl(\Psi(v_0,s)\bigr)\Bigr)
  \\
  df_k\bigl(\Psi(v_0,s)\bigr) 
  \\
  \operatorname{grad}_{\pi\bigl(\Psi(u_0,s-s_0)\bigr)}f_k
  -
  P(0,s-s_0;\gamma_{u_0})\operatorname{grad}_{y_0}f_k
  \end{bmatrix} 
\\
& =
  \begin{bmatrix}
  df_k\bigl(\Psi(v_0,s_0)\bigr)
  \\
  {\nabla}df_k(u_0,u_0) 
  \\
  \nabla_{u_0}\operatorname{grad}f_k
  -
  \Bigl(\Gamma^p_{qr}(y_0)\cdot(u_0)_q\cdot(\operatorname{grad}_{y_0}f_k)_r\Bigr)
\vspace{3pt}
  \\
  \qquad\qquad\quad
  +
  \Bigl(\Gamma^p_{qr}(y_0)\cdot(u_0)_q\cdot(\operatorname{grad}_{y_0}f_k)_r\Bigr)
  \end{bmatrix}
\\
& =
  \begin{bmatrix}
  df_k\bigl(\Psi(v_0,s_0)\bigr)
  \\
  {\nabla}df_k(u_0,u_0) 
  \\
  \nabla_{u_0}\operatorname{grad}f_k
  \end{bmatrix}
\\
& =
  \begin{bmatrix}
  0
  \\
  {\nabla}df_k(u_0,u_0) 
  \\
  A_ku_0+\langle\nabla_{u_0}\operatorname{grad}f_k,\operatorname{grad}f_k\rangle\operatorname{grad}_{y_0}f_k 
  \end{bmatrix}. 
\end{align*}
Lemma~\ref{theorem:gradient} shows that 
$\langle\nabla_{u_0}\operatorname{grad}f_k,\operatorname{grad}f_k\rangle=0$, and we obtain 
\begin{equation}
\frac{\partial J}{\partial s}(s_0,v_0) 
=
\begin{bmatrix}
0
\\
{\nabla}df_k(u_0,u_0) 
\\
A_ku_0
\end{bmatrix}. 
\label{equation:000}
\end{equation}
\par
Secondly we compute the differentiation of $J(s,v)$ with respect to the variables of 
$T_{v_0}\bigl(S(\partial{D_j})\bigr)$. 
We consider curves on $S(\partial{D_j})$ of the form 
\begin{alignat*}{2}
  v(\alpha_p)
& =
  v_0+\alpha_pv_p+\mathcal{O}(\alpha_p^2),
& \quad
  p
& =
  0,1,\dotsc,n-2,
\\
  v(\beta_q)
& =
  v_0+\beta_q\zeta_q+\mathcal{O}(\beta_q^2),
& \quad
  q
& =
  1,\dotsc,n-2. 
\end{alignat*}
In other words we take local coordinates 
$(\alpha_0,\alpha_1,\dotsc,\alpha_{n-2},\beta_1,\dotsc,\beta_{n-2})$ 
of $T_{v_0}\bigl(S(\partial{D_j})\bigr)$ near the zero section at $v_0 \in S(\partial{D_j})$. 
Furthermore the following two mappings  
\begin{align*}
  \mathbb{R}^n
&  \ni 
  (h,\alpha_0,\alpha_1,\dotsc,\alpha_{n-2})
\\
& \mapsto 
  \exp_{x_0}\bigl(h\nu_j(x_0)+\alpha_0v_0+\alpha_1v_1+\dotsb+\alpha_{n-2}v_{n-2}\bigr) \in U_j,
\\
  \mathbb{R}^n
& \ni 
  (h,\alpha_0,\beta_1,\dotsc,\beta_{n-2})
\\
& \mapsto 
  \exp_{x_0}\bigl(h\nu_j(x_0)+\alpha_0v_0+\beta_1\zeta_1+\dotsb+\beta_{n-2}\zeta_{n-2}\bigr) \in U_j,
\end{align*}
give geodesic local coordinate systems in $U_j$ respectively. 
Hence if we use these coordinates, we may assume 
that all the Christoffel symbols of $(M,g)$ vanish at $x_0$, 
and differentiation of vector fields with respect to these variables 
at $x_0$ become covariant derivatives with respect to the corresponding vector fields. 
See \cite[Exercise~4 in Page~37 and its Solution in Page~326]{Sakai} 
for instance. We shall compute 
\begin{align*}
  \frac{\partial}{\partial \alpha_p}\bigg\vert_{\alpha_p=0}
  J\bigl(s_0,v(\alpha_p)\bigr)
& =
  \frac{\partial}{\partial \alpha_p}\bigg\vert_{\alpha_p=0}
  J\bigl(s_0,v_0+\alpha_p(v_p,0)\bigr),
\\
  \frac{\partial}{\partial \beta_q}\bigg\vert_{\beta_q=0}
  J\bigl(s_0,v(\beta_q)\bigr)
& =
  \frac{\partial}{\partial \beta_q}\bigg\vert_{\beta_q=0}
  J\bigl(s_0,v_0+\beta_q(0,\zeta_q)\bigr). 
\end{align*}
Thus we shall actually compute
\begin{align*}
& \frac{\partial}{\partial \alpha_p}\bigg\vert_{\alpha_p=0}
  J\bigl(s_0,v_0+\alpha_p(v_p,0)\bigr),
\\
  =
& \frac{\partial}{\partial \alpha_p}\bigg\vert_{\alpha_p=0}
  \begin{bmatrix}
  f_k\Bigl(\pi\bigl(\Psi(v_0+\alpha_p(v_p,0),s_0)\bigr)\Bigr)
  \\
  df_k\bigl(\Psi(v_0+\alpha_p(v_p,0),s_0)\bigr)
  \\
  \operatorname{grad}_{\pi\bigl(\Psi(v_0+\alpha_p(v_p,0),s_0)\bigr)}f_k
  -
  P(0,s_0;\gamma_{v_0+\alpha_p(v_p,0)})
  \operatorname{grad}_{\pi(v_0+\alpha_p(v_p,0))}f_j
  \end{bmatrix},
\\
& \frac{\partial}{\partial \beta_q}\bigg\vert_{\beta_q=0}
  J\bigl(s_0,v_0+\beta_q(0,\zeta_q)\bigr)
\\
  =
& \frac{\partial}{\partial \beta_q}\bigg\vert_{\beta_q=0}
  \begin{bmatrix}
  f_k\Bigl(\pi\bigl(\Psi(v_0+\beta_q(0,\zeta_q),s_0)\bigr)\Bigr)
  \\
  df_k\bigl(\Psi(v_0+\beta_q(0,\zeta_q),s_0)\bigr)
  \\
  \operatorname{grad}_{\pi\bigl(\Psi(v_0+\beta_q(0,\zeta_q),s_0)\bigr)}f_k
  -
  P(0,s_0;\gamma_{v_0+\beta_q(0,\zeta_q)})
  \operatorname{grad}_{\pi(v_0+\beta_q(0,\zeta_q))}f_j
  \end{bmatrix}
\end{align*} 
for $p=0,1,\dotsc,n-2$ and $q=1,\dotsc,n-2$. 
\par
We compute the first row. 
Since $u_0, Y_p(s_0), Z_q(s_0) \in T_{y_0}(\partial{D_k})$, 
we deduce that for $p,q=1,\dotsc,n-2$  
\begin{align}
  \frac{\partial}{\partial \alpha_0}\bigg\vert_{\alpha_0=0}
&  f_k\Bigl(\pi\bigl(\Psi(v_0+\alpha_0(v_0,0),s_0)\bigr)\Bigr)
\nonumber
\\
  =
& df_k\bigl(D\pi\vert_{\Psi(v_0,s_0)}{\circ}D_v\Psi\vert_{v_0,s_0}(v_0,0)\bigr)
\nonumber
\\
  =
& df_k\bigl(\Psi(v_0,s_0)\bigr)
  =
  df_k(u_0)
  =0,
\label{equation:011}
\\
  \frac{\partial}{\partial \alpha_p}\bigg\vert_{\alpha_p=0}
& f_k\Bigl(\pi\bigl(\Psi(v_0+\alpha_p(v_p,0),s_0)\bigr)\Bigr)
\nonumber
\\
  =
& df_k\bigl(D\pi\vert_{\Psi(v_0,s_0)}{\circ}D_v\Psi\vert_{v_0,s_0}(v_p,0)\bigr)
\nonumber
\\
  =
& df_k\bigl(Y_p(s_0)\bigr)
  =0, 
\label{equation:012}
\\
  \frac{\partial}{\partial \beta_q}\bigg\vert_{\beta_q=0}
& f_k\Bigl(\pi\bigl(\Psi(v_0+\beta_q(0,\zeta_q),s_0)\bigr)\Bigr)
\nonumber
\\
  =
& df_k\bigl(D\pi\vert_{\Psi(v_0,s_0)}{\circ}D_v\Psi\vert_{v_0,s_0}(0,\zeta_q)\bigr)
\nonumber
\\
  =
& df_k\bigl(Z_q(s_0)\bigr)
  =0.
\label{equation:013}
\end{align}
\par
We compute the second row. 
Since $X\bigl(df_k(Y)\bigr)=\nabla{df_k(X,Y)}+df_k(\nabla_XY)$ for vector fields $X$ and $Y$ on $U_k$, 
we deduce for $p,q=1,\dotsc,n-2$ 
\begin{align}
  \frac{\partial}{\partial \alpha_0}\bigg\vert_{\alpha_0=0}
& df_k\bigl(\Psi(v_0+\alpha_0(v_0,0),s_0)\bigr)
\nonumber
\\
  =
& {\nabla}df_k\bigl(D_v\Psi\vert_{(v_0,s_0)}(v_0,0),u_0\bigr)
  +
  df_k\bigl(\nabla_{D_v\Psi\vert_{(v_0,s_0)}(v_0,0)}u_0\bigr)
\nonumber
\\
  =
& {\nabla}df_k(u_0,u_0),
\label{equation:021}
\\
  \frac{\partial}{\partial \alpha_p}\bigg\vert_{\alpha_p=0}
& df_k\bigl(\Psi(v_0+\alpha_p(v_p,0),s_0)\bigr)
\nonumber
\\
  =
& {\nabla}df_k\bigl(D_v\Psi\vert_{(v_0,s_0)}(v_p,0),u_0\bigr)
  +
  df_k\bigl(\nabla_{D_v\Psi\vert_{(v_0,s_0)}(v_p,0)}u_0\bigr)
\nonumber
\\
  =
& {\nabla}df_k\bigl(\nabla_{u_0}Y_p(s_0),u_0\bigr)
  +
  df_k\bigl(\nabla_{\nabla_{u_0}Y_p(s_0)}u_0\bigr), 
\label{equation:022}
\\
  \frac{\partial}{\partial \beta_q}\bigg\vert_{\beta_q=0}
& df_k\bigl(\Psi(v_0+\beta_q(0,\zeta_q),s_0)\bigr)
\nonumber
\\
  =
& {\nabla}df_k\bigl(D_v\Psi\vert_{(v_0,s_0)}(0,\zeta_q,),u_0\bigr)
  +
  df_k\bigl(\nabla_{D_v\Psi\vert_{(v_0,s_0)}(0,\zeta_q)}u_0\bigr)
\nonumber
\\
  =
& {\nabla}df_k\bigl(\nabla_{u_0}Z_q(s_0),u_0\bigr)
  +
  df_k\bigl(\nabla_{\nabla_{u_0}Z_q(s_0)}u_0\bigr).  
\label{equation:023}
\end{align}
\par
We compute the third row. 
We remark again that the derivatives of vector fields at $x_0$ 
become the corresponding covariant derivatives. 
Note that $Y_0(0)=v_0$ and $Y_0(s_0)=u_0$. 
For the $\alpha_p$ variable ($p=0,1,\dotsc,n-2$), we deduce that 
\begin{align}
  \frac{\partial}{\partial \alpha_p}\bigg\vert_{\alpha_p=0}
& \Bigl\{
  \operatorname{grad}_{\pi\bigl(\Psi(v_0+\alpha_p(v_p,0),s_0)\bigr)}f_k
\nonumber
\\
& -
  P(0,s_0;\gamma_{v_0+\alpha_p(v_p,0)})
  \operatorname{grad}_{\pi\bigl(v_0+\alpha_p(v_p,0)\bigr)}f_j
  \Bigr\}
\nonumber
\\
  =
& \frac{\partial}{\partial \alpha_p}\bigg\vert_{\alpha_p=0}
  \bigl\{
  \operatorname{grad}_{\pi\bigl(\Psi(v_0+\alpha_p(v_p,0),s_0)\bigr)}f_k
\nonumber
\\
& \qquad\qquad\qquad
  -
  P(0,s_0;\gamma_{v_0+\alpha_p(v_p,0)})
  \operatorname{grad}_{\pi(v_0)}f_j
  \bigr\}
\nonumber
\\
  -
& P(0,s_0;\gamma_{v_0})
  \frac{\partial}{\partial \alpha_p}\bigg\vert_{\alpha_p=0}
  \operatorname{grad}_{\pi\bigl(v_0+\alpha_p(v_p,0)\bigr)}f_j
\nonumber
\\
  =
& Y_p(s_0)\bigl(\operatorname{grad}f_k-\operatorname{grad}f_k\bigr)
  -
  P(0,s_0;\gamma_{v_0})Y_p(0)\bigl(\operatorname{grad}f_j\bigr)
\nonumber
\\
  =
& - 
  P(0,s_0;\gamma_{v_0})\nabla_{Y_p(0)}\bigl(\operatorname{grad}f_j\bigr)
\nonumber
\\
  =
& - 
  P(0,s_0;\gamma_{v_0})
  \bigl\{A_jY_p(0)+\langle\nabla_{Y_p(0)}\operatorname{grad}f_j,\operatorname{grad}f_j\rangle\operatorname{grad}f_j\bigr\}
\nonumber
\\
  =
& - 
  P(0,s_0;\gamma_{v_0})A_jY_p(0)
\nonumber
\\
  =
& - 
  P(0,s_0;\gamma_{v_0})A_jv_p. 
\label{equation:0312}
\end{align}
Here we used Lemma~\ref{theorem:gradient}. 
In the same way we have 
\begin{align}
  \frac{\partial}{\partial \beta_q}\bigg\vert_{\beta_q=0}
& \Bigl\{
  \operatorname{grad}_{\pi\bigl(\Psi(v_0+\beta_q(0,\zeta_q),s_0)\bigr)}f_k
\nonumber
\\
& -
  P(0,s_0;\gamma_{v_0+\beta_q(0,\zeta_q)})
  \operatorname{grad}_{\pi(v_0+\beta_q(0,\zeta_q))}f_j
  \Bigr\}
\nonumber
\\
  =
& - 
  P(0,s_0;\gamma_{v_0})A_jZ_q(0)
  =
  0
\label{equation:033}
\end{align} 
since $Z_q(0)=0$. 
\par
Combining 
\eqref{equation:000}, 
\eqref{equation:011}, 
\eqref{equation:012}, 
\eqref{equation:013}, 
\eqref{equation:021}, 
\eqref{equation:022}, 
\eqref{equation:023}, 
\eqref{equation:0312} 
and 
\eqref{equation:033}, 
we obtain 
\begin{align*}
  DJ\vert_{(s_0,v_0)}
& =
  \begin{bmatrix}
  0 & 0
  \\
  J_1 & R_1
  \end{bmatrix},
\\
  J_1
& =
  \begin{bmatrix}
  \nabla{df_k}(u_0,u_0) 
  & 
  \bigl[\nabla{df_k}(u_0,u_0),M_1\bigr]
  \\
  A_ku_0 
  &
  \bigl[-P(0,s_0;\gamma_{v_0})A_jv_p\bigr]_{p=0,\dotsc,n-2}
  \end{bmatrix},
\\
  M_1
& =
  [m_1,\dotsc,m_{n-2}],
\\
  m_p
& :=
  {\nabla}df_k\bigl(\nabla_{u_0}Y_p(s_0),u_0\bigr)
  +
  df_k\bigl(\nabla_{\nabla_{u_0}Y_p(s_0)}u_0\bigr),
\\
  R_1
& =
  \begin{bmatrix}
  r_1 & \dotsb & r_{n-2}
  \\
  0 & \dotsb & 0 
  \end{bmatrix},
\\
  r_q
& :=
  {\nabla}df_k\bigl(\nabla_{u_0}Z_q(s_0),u_0\bigr)
  +
  df_k\bigl(\nabla_{\nabla_{u_0}Z_q(s_0)}u_0\bigr).
\end{align*}
We shall prove that $\operatorname{rank}([J_1,R_1])=n$. 
More precisely 
we show that $\operatorname{rank}([J_1,R_1])=n$ 
if $R_1\ne{O_{n\times(n-2)}}$, 
and we show that $\operatorname{rank}(J_1)=n$ if $R_1=O_{n\times(n-2)}$, 
where $O_{p{\times}q}$ denotes the $p{\times}q$ zero matrix. 
We now recall that \eqref{equation:hessian} shows ${\nabla}df_k(u_0,u_0)>0$. 
We set $w_p=-P(0,s_0;\gamma_{v_0})A_jv_p$, $p=0,1,\dotsc,n-2$ for short. 
It follows that $\{w_0,w_1,\dotsc,w_{n-2}\}$ 
forms a basis of $T_{y_0}(\partial{D_k})$ 
since 
$\{v_0,v_1,\dotsc,v_{n-2}\}$ 
forms a basis of $T_{x_0}(\partial{D_j})$, 
$\operatorname{rank}(A_j)=n-1$, 
$P(0,s_0;\gamma_{v_0})$ is the parallel transport of tangent vectors along $\gamma_{v_0}$, 
and 
$\dot{\gamma}_{v_0}(s_0) \in S(\partial{D_k})$. We have 
$$
J_1
=
\begin{bmatrix}
\nabla{df_k}(u_0,u_0) 
& 
\nabla{df_k}(u_0,u_0)
&
m_1
& 
\dotsb 
& 
m_{n-2}
\\
A_ku_0 
&
w_0
&
w_1
&
\dotsb
&
w_{n-2}
\end{bmatrix}.
$$
The $(2,1)$-element $A_ku_0$ belongs to $T_{y_0}(\partial{D_k})$, 
and is given by a linear combination of $w_0,w_1,\dotsc,w_{n-1}$. 
In particular $\operatorname{rank}(J_1) \geqq n-1$ since 
$w_0,w_1,\dotsc,w_{n-2}$ are linearly independent. 
\vspace{6pt}
\\
\underline{Case of $R_1{\ne}O_{n\times(n-2)}$}.\ 
Suppose that $R_1{\ne}O_{n\times(n-2)}$. 
Then there exists a $(2n-2)\times(2n-2)$ regular matrix $Q_1$ such that 
$[J_1,R_1]$ can be reduced to
\begin{align*}
  [J_1,R_1]Q_1
&  =
   [J_2,O_{2\times(n-3)}],
\\
  J_2
& =
  \begin{bmatrix}
  \nabla{df_k}(u_0,u_0) 
  & 
  \nabla{df_k}(u_0,u_0)
  &
  m_1
  & 
  \dotsb 
  & 
  m_{n-2}
  & 
  1 
  \\
  A_ku_0  
  &
  w_0
  &
  w_1
  &
  \dotsb
  &
  w_{n-2}
  & 
  0
  \end{bmatrix}.
\end{align*}
It suffices to show that $\operatorname{rank}(J_2)=n$. 
We can eliminate the first row using the $(1,n+1)$-element. 
There exists an $(n+1)\times(n+1)$ regular matrix $Q_2$ such that 
$$
J_2Q_2
=
\begin{bmatrix}
0
& 
0
&
0
& 
\dotsb 
& 
0
& 
1 
\\
A_ku_0  
&
w_0
&
w_1
&
\dotsb
&
w_{n-2}
& 
0
\end{bmatrix}.
$$
Since $A_ku_0$ can be expressed by a linear combination of 
$w_0,w_1,\dotsc,w_{n-2}$, there exists 
an $(n+1)\times(n+1)$ regular matrix $Q_3$ such that
$$
J_2Q_2Q_3
=
\begin{bmatrix}
0
& 
0
&
0
& 
\dotsb 
& 
0
& 
1 
\\
0
&
w_0
&
w_1
&
\dotsb
&
w_{n-2}
& 
0
\end{bmatrix}, 
$$
which shows that $\operatorname{rank}(J_2)=n$. 
\vspace{6pt}
\\
\underline{Case of $R_1=O_{n\times(n-2)}$}.\ 
Suppose that $R_1=O_{n\times(n-2)}$. 
We show $\operatorname{rank}(J_1)=n$, that is, 
all the columns of $J_1$ are linearly independent. 
To show $\operatorname{rank}(J_1)=n$, 
we suppose $\operatorname{rank}(J_1)=n-1$ and obtain a contradiction. 
We now assume that $\operatorname{rank}(J_1)=n-1$. 
Then we have 
$\operatorname{rank}(DJ\vert_{(s_0,v_0)})=n-1$, and 
$\operatorname{dim}\bigl(\operatorname{Ran}(DJ\vert_{(s_0,v_0)})\bigr)=n-1$. 
This means that the property $v_0 \in \mathcal{M}_{jk}^{(+)}$ 
is stable under the infinitesimal displacement of $x_0=\pi(v_0)$ 
in the $v_0$ direction since $\operatorname{dim}(\partial{D_j})=n-1$ 
and $v_0$ is uniquely determined by $x_0$. 
This contradicts the strictly convexity of 
$\partial{D_j}$ and $\partial{D_k}$. 
Indeed if we consider a curve on $\partial{D_j}$ of the form 
$x(\alpha)=x_0+\alpha{v_0}+\mathcal{O}(\alpha^2)$ near $\alpha=0$, 
we have $S_{x(\alpha)}(\partial{D_j})\cap\mathcal{M}_{jk}^{(\pm)}=\emptyset$ for $\alpha\ne0$ 
since the strictly convexity of $\partial{D_j}$ and $\partial{D_k}$. 
This completes the proof of Lemma~\ref{theorem:mjk}.
\end{proof}
Since the property $v \in \mathcal{M}_{jk}^{(\pm)}$ 
is uniquely determined by $x \in \partial{D_j}$, 
the restriction of the projection 
$\pi: S(M^\text{int}) \rightarrow M^\text{int}$ 
on $\mathcal{M}_{jk}^{(\pm)}$ are injective respectively. 
Then we can deduce the following immediately. 
\begin{corollary}
\label{theorem:ljk} 
Suppose {\rm ({\bf A0})}, {\rm ({\bf A1})} and {\rm ({\bf A2})}. 
If we set 
$\mathcal{B}_{jk}^{(\pm)}:=\{\pi(v) : v \in \mathcal{M}_{jk}^{(\pm)}\}$, 
then $\mathcal{B}_{jk}^{(\pm)}$ are 
$(n-2)$-dimensional connected submanifolds of $\partial{D_j}$. 
Moreover, if we set 
$$
\mathcal{L}_{jk}^{(\pm)}
=
\bigl\{
\gamma_v(s) : 
v \in \mathcal{M}_{jk}^{(\pm)}, 
s \in \bigl(\tau_-(v),\tau_+(v)\bigr)
\bigr\},
$$
then $\mathcal{L}_{jk}^{(\pm)}$ are hypersurfaces which are tangent to 
$\partial{D_j}$ on $\mathcal{B}_{jk}^{(\pm)}$ 
and $\partial{D_k}$ on $\mathcal{B}_{kj}^{(\pm)}$ respectively. 
See Figure~3. 
\end{corollary} 
We now state some elementary facts which follow immediately, 
and introduce some notation used later. 
\begin{itemize}
\item  
It follows that $\mathcal{L}_{jk}^{(\pm)}=\mathcal{L}_{kj}^{(\pm)}$. 
We set 
$$
\mathcal{L}_{jk}:=\mathcal{L}_{jk}^{(+)}\cup\mathcal{L}_{jk}^{(-)}, 
\quad
\mathcal{L}
:=
\bigcup_{1 \leqq j < k \leqq J}
\mathcal{L}_{jk}.
$$
\item 
We have 
$$
\Sigma_{jk}^{(1)}
=
\Sigma_{kj}^{(1)}
=
\bigl\{
F(v) : 
v \in \mathcal{M}_{jk}^{(+)}\cup\mathcal{M}_{jk}^{(-)}
\bigr\}
=
\bigl\{
F(\tilde{v}) : 
\tilde{v} \in \mathcal{M}_{kj}^{(+)}\cup\mathcal{M}_{kj}^{(-)}
\bigr\}. 
$$
If we set 
\begin{align*}
  \Delta(\mathcal{M}_{jk})
& :=
  \Delta(\mathcal{M}_{jk}^{(+)})\cup\Delta(\mathcal{M}_{jk}^{(-)}),
\\
  \Delta(\mathcal{M}_{jk}^{(\pm)})
& :=
  \bigl\{
  (v,\tilde{v})\in \mathcal{M}_{jk}^{(\pm)}\times\mathcal{M}_{kj}^{(\pm)}:
  \exists s\in\bigl(\tau_-(v),\tau_+(v)\bigr) 
  \ \text{s.t.}\ 
\\
& \qquad
  \tilde{v}=\dot{\gamma}_v(s), 
  P(0,s;\gamma_v)\eta_k\bigl(\pi(\tilde{v})\bigr)=\pm\eta_j\bigl(\pi(v)\bigr)
\bigr\}. 
\end{align*}
then we have 
\begin{align}
  N^\ast(\Sigma_{jk}^{(1)})
& =
  \bigl\{
  \xi+\tilde{\xi}:
  \xi,\tilde{\xi} \in T^\ast_{F(v)}\bigl(\partial_-S(M)\bigr), 
\nonumber
\\
& \qquad 
  \exists (v,\tilde{v}) \in \Delta(\mathcal{M}_{jk}), 
  \exists \eta \in N^\ast_{\pi(v)}(\partial{D_j}), 
  \exists \tilde{\eta} \in N^\ast_{\pi(\tilde{v})}(\partial{D_k})
  \ \text{s.t.}\ 
\nonumber
\\
& \qquad
  DF\vert^T_v\xi=D\pi\vert^T_v\eta, 
  DF\vert^T_{\tilde{v}}\tilde{\xi}=D\pi\vert^T_{\tilde{v}}\tilde{\eta} 
  \bigr\}.
\label{equation:nstarsigmajk1} 
\end{align}
\item 
Let $\Omega_{jk}$ be a connected subdomain of $\partial{D_j}$ 
enclosed by $\mathcal{B}_{jk}^{(+)}\cup\mathcal{B}_{jk}^{(-)}$. 
Fix arbitrary $x \in \Omega_{jk}$. 
See Figure~5 below. 
It follows that $H_j(x) \cap D_k \ne \emptyset$. 
Furthermore there exist 
$v_i \in S_x(\partial{D_j})$ and 
$s_i\in\bigl(\tau_-(v_i),\tau_+(v_i)\bigr)$ 
with $i=1,2$ such that $v_1 \ne v_2$ and 
$\dot{\gamma}_{v_i}(s_i) \in S(\partial{D_k})$. 
Clearly $H_j(x)$ is not tangent to $\partial{D_k}$ 
at $\gamma_{v_i}(s_i)$, and 
$P(0,s_i;\gamma_{v_i})\nu_j(x) \ne \pm\nu_k\bigl(\gamma_{v_i}(s_i)\bigr)$ 
for $i=1,2$. Hence we have 
\begin{align*}
  \Sigma_{jk}^{(2)}
& =
  \Sigma_{kj}^{(2)}
\\
& =
  \bigl\{
  F(v) : 
  \exists v \in S(\Omega_{jk}), 
  \exists s\in\bigl(\tau_-(v),\tau_+(v)\bigr)
  \ \text{s.t.}\ 
  \dot{\gamma}_{v}(s) \in S(\partial{D_k})
  \bigr\}
\\
& =
  \bigl\{
  F(\tilde{v}) : 
  \exists \tilde{v} \in S(\Omega_{kj}), 
  \exists s\in\bigl(\tau_-(\tilde{v}),\tau_+(\tilde{v})\bigr)
  \ \text{s.t.}\ 
  \dot{\gamma}_{\tilde{v}}(s) \in S(\partial{D_j})
  \bigr\}.
\end{align*}
If we set 
\begin{align*}
  \Delta\bigl(S(\Omega_{jk})\bigr)
& :=
  \bigl\{
  (v,\tilde{v})\in S(\Omega_{jk}){\times}S(\Omega_{kj}):  
  \exists s\in\bigl(\tau(v)_-,\tau_+(v)\bigr) 
  \ \text{s.t.}\ 
\\
& \qquad
  \tilde{v}=\dot{\gamma}_v(s), 
  P(0,s;\gamma_v)\nu_j\bigl(\pi(v)\bigr)\ne\pm\nu_k\bigl(\pi(\tilde{v})\bigr)
\bigr\},
\end{align*}
then we have 
\begin{align}
  N^\ast(\Sigma_{jk}^{(2)})
& =
  \bigl\{
  \xi+\tilde{\xi}:
  \xi,\tilde{\xi} \in T^\ast_{F(v)}\bigl(\partial_-S(M)\bigr), 
\nonumber
\\
& \qquad 
  \exists (v,\tilde{v}) \in \Delta\bigl(S(\Omega_{jk})\bigr), 
  \exists \eta \in N^\ast_{\pi(v)}(\partial{D_j}), 
  \exists \tilde{\eta} \in N^\ast_{\pi(\tilde{v})}(\partial{D_k})
  \ \text{s.t.}\ 
\nonumber
\\
& \qquad
  DF\vert^T_v\xi=D\pi\vert^T_v\eta, 
  DF\vert^T_{\tilde{v}}\tilde{\xi}=D\pi\vert^T_{\tilde{v}}\tilde{\eta} 
  \bigr\}.
\label{equation:nstarsigmajk2} 
\end{align}
\end{itemize} 
\begin{center}
\includegraphics[width=70mm]{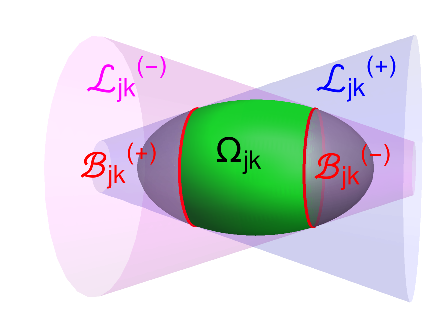}
\vspace{11pt}
\\
Figure~5. 
The situation of the neighborhood of $\partial{D_j}$ 
regarding the relationship with $\partial{D_k}$. 
\end{center}
%
%
%
%
% 
%%%%%%
%%%%%% section 5
%%%%%%
\section{Main Theorem}
\label{section:MAIN}
Let $E\geqq0$ be a parameter describing the energy level of the X-ray beam, and let $E_0$ be a positive constant describing the central energy level. We study the very simple model of the beam hardening effect originated by \cite{ParkChoiSeo}. We assume the following. 
\begin{description}
\item[(A3)] 
The half density $f_E(x)$ of the distribution of the attenuation coefficient is of the form
$$
f_E(x)
=
f_{E_0}(x)
+
\alpha(E-E_0)1_D(x)\lvert{dv_g(x)}\rvert^{1/2}, 
\quad
x \in M^\text{int},
$$
where the half density $f_{E_0}$ describes the distribution of the attenuation coefficient of normal tissue, 
and $\alpha>0$ is a constant. The spectral function $\rho(E)$, which is a probability density function on $[0,\infty)$, is of the form 
$$
\rho(E)
=
\frac{1}{2\varepsilon}
1_{[E_0-\varepsilon,E_0+\varepsilon]}(E),
\quad
E\in[0,\infty),
$$
where $\varepsilon \in(0,E_0)$ is a constant. 
\end{description}
In this case the measurements $P$ become 
\begin{align*}
  P(w)
& :=
  -
  \log\left\{
       \int_0^\infty
       \rho(E)
       \exp\left(
           -
           \frac{\mathcal{X}[f_E]}{\lvert{d\mu}\rvert^{1/2}}(w)
           \right)
       dE
       \right\}
  \lvert{d\mu}\rvert^{1/2}
\\
& =
  \mathcal{X}[f_{E_0}](w)
  +
  \sum_{l=1}^\infty
  B_l
  \left(\alpha\varepsilon\frac{\mathcal{X}[1_D]}{\lvert{d\mu}\rvert^{1/2}}(w)\right)^{2l}
  \lvert{d\mu}\rvert^{1/2},
  \quad
  w \in \partial_-S(M) 
\end{align*}
with some sequence $\{B_l\}_{l=1}^\infty$ of real numbers. 
Let $Q$ be a local parametrix for $\mathcal{X}^T\circ\mathcal{X}$. 
We set 
\begin{align*}
  P_\text{MA}
& :=
  P-\mathcal{X}[f_{E_0}]
  =
  \sum_{l=1}^\infty
  B_l
  \left(\alpha\varepsilon\frac{\mathcal{X}[1_D]}{\lvert{d\mu}\rvert^{1/2}}\right)^{2l}
  \lvert{d\mu}\rvert^{1/2},
\\
  f_\text{CT}
& :=
  Q\circ\mathcal{X}^T[P]
  \equiv
  f_{E_0}+f_\text{MA}
  \quad
  \text{mod}
  \quad
  C^\infty(\Omega_{M^\text{int}}^{1/2}), 
\\
  f_\text{MA}
& :=
  Q\circ\mathcal{X}^T[P_\text{MA}]
  =
  \sum_{l=1}^\infty
  B_l(\alpha\varepsilon)^{2l}
  Q\circ\mathcal{X}^T
  \left[
  \left(\frac{\mathcal{X}[1_D]}{\lvert{d\mu}\rvert^{1/2}}\right)^{2l}
  \lvert{d\mu}\rvert^{1/2}
  \right]
\\
& =
  -
  \frac{(\alpha\varepsilon)^2}{6}
  Q\circ\mathcal{X}^T
  \left[
  \left(\frac{\mathcal{X}[1_D]}{\lvert{d\mu}\rvert^{1/2}}\right)^2
  \lvert{d\mu}\rvert^{1/2}
  \right]
  +
  \frac{(\alpha\varepsilon)^4}{180}
  Q\circ\mathcal{X}^T
  \left[
  \left(\frac{\mathcal{X}[1_D]}{\lvert{d\mu}\rvert^{1/2}}\right)^4
  \lvert{d\mu}\rvert^{1/2}
  \right]
  +
  \dotsb.
\end{align*}
\par
We now state our main theorem of the present paper. 
\begin{theorem}
\label{theorem:51}
Suppose {\rm ({\bf A0})}, {\rm ({\bf A1})}, 
{\rm ({\bf A2})} and {\rm ({\bf A3})}. 
Then we have
\begin{equation}
f_\text{MA} 
\in 
I^{-1/2-3n/4}\bigl(N^\ast(\mathcal{L})\setminus0\bigr)
\label{equation:MAIN} 
\end{equation}
away from $N^\ast(\partial{D})$. 
Moreover the principal symbol $(\alpha\varepsilon)^2\sigma_\text{prin}$ 
of the top term of $f_\text{MA}$ never vanish, 
and the principal symbol of $f_\text{MA}$ is 
$(\alpha\varepsilon)^2\{\sigma_\text{prin}+\mathcal{O}((\alpha\varepsilon)^2)\}$.
\end{theorem}
The proof of Theorem~\ref{theorem:51} 
is the investigation of the interaction of singularities in 
\begin{equation}
 Q\circ\mathcal{X}^T
  \left[
  \left(\frac{\mathcal{X}[1_D]}{\lvert{d\mu}\rvert^{1/2}}\right)^{2l}
  \lvert{d\mu}\rvert^{1/2}
  \right], 
\quad
l=1,2,3,\dotsc. 
\label{equation:singularity}
\end{equation}
First we study the quadratic term ($l=1$) below, 
and secondly we study the higher order terms ($l\geqq2$) in the next section. 
For the quadratic term we study the intersection calculus of 
$C_{\mathcal{X}^T}{\circ}N^\ast(\Sigma_j)$ and 
$C_{\mathcal{X}^T}{\circ}N^\ast(\Sigma_{jk})$. 
In other words, we need to know what kind of conormal singularities 
in $N^\ast(\Sigma_j)$ or $N^\ast(\Sigma_{jk})$ 
can be composable with the canonical relation $C_{\mathcal{X}^T}$, 
which is determined by the image of the adjoints 
of the differentials $DF$ and $D\pi$.  
Firstly we need to know the basic facts on the relationship between 
these two adjoints, which was proved by Holman and Uhlmann in 
\cite{HolmanUhlmann}. 
\begin{lemma}[{\bf Holman and Uhlmann \cite[Lemma~5]{HolmanUhlmann}}]
\label{theorem:holmanuhlmann}
We assume {\rm ({\bf A0})} without the simplicity condition, 
that is, conjugate points are allowed for $(M,g)$.  
Suppose that 
$$
v, \tilde{v} \in S(M^\text{int}),
\quad 
F(v)=F(\tilde{v}),
\quad 
\xi \in T^\ast_{F(v)}(\partial_-S(M)),
$$
$$
\eta \in T^\ast_{\pi(v)}(M^\text{int}),
\quad 
\tilde{\eta} \in T^\ast_{\pi(\tilde{v})}(M^\text{int}),
$$ 
$$
DF\vert^T_v\xi=D\pi\vert^T_v\eta,
\quad
DF\vert^T_{\tilde{v}}\xi=D\pi\vert^T_{\tilde{v}}\tilde{\eta}. 
$$
If $v \ne \tilde{v}$, then $v$ and $\tilde{v}$ are conjugate. 
\end{lemma}
In the same notation of Lemma~\ref{theorem:holmanuhlmann}, 
\eqref{equation:etav} shows that 
$\eta(v)=0$ and $\tilde{\eta}(\tilde{v})=0$ hold. 
If we assume the simplicity condition on $(M,g)$, 
then Lemma~\ref{theorem:holmanuhlmann} shows that 
for any $\xi \in T^\ast\bigl(\partial_-S(M)\bigr)\setminus0$ 
with $w:=\pi_{\partial_-S(M)}(\xi)$, 
one of the following occurs:
\begin{itemize}
\item 
There exist  
$s \in \bigl(0,\tau_+(w)\bigr)$ and 
$\eta \in T^\ast_{\gamma_w(s)}(M^\text{int})$ uniquely 
such that 
$$
\eta\bigl(\dot{\gamma}_w(s)\bigr)=0, 
\quad
DF\vert^T_{\dot{\gamma}_w(s)}\xi=D\pi\vert^T_{\dot{\gamma}_w(s)}\eta.
$$
\item 
For any $s \in \bigl(0,\tau_+(w)\bigr)$ and 
$\eta \in T^\ast_{\gamma_w(s)}(M^\text{int})$ 
satisfying $\eta\bigl(\dot{\gamma}_w(s)\bigr)=0$, 
we have  
$$
DF\vert^T_{\dot{\gamma}_w(s)}\xi{\ne}D\pi\vert^T_{\dot{\gamma}_w(s)}\eta. 
$$
\end{itemize}
Roughly speaking, 
if we assume the simplicity condition for $(M,g)$, then 
for any $\xi \in T^\ast\bigl(\partial_-S(M)\bigr)\setminus0$, 
there exists only one covector $\eta \in T^\ast(M^\text{int})\setminus0$, if any, that  
satisfies the identity for $\xi$ and $\eta$. In this paper we need to study the propagation of singularities caused by $\eta, \tilde{\eta} \in N^\ast(\partial{D})\setminus0$ with $\pi(\eta)\ne\pi(\tilde{\eta})$. In other words we investigate the existence or the nonexistence 
of pairs of curves  $\xi(s) \in N^\ast(\Sigma_{jk})$ and 
$\eta(s) \in T^\ast(M^\text{int})$ satisfying 
$$
DF\vert^T_{\dot{\gamma}_w(s)}\xi(s)=D\pi\vert^T_{\dot{\gamma}_w(s)}\eta(s). 
$$ 
\par
We now set the notation. 
Suppose ({\bf A0}), ({\bf A1}) and ({\bf A2}). 
Fix arbitrary $w \in \Sigma_{jk}$ with $j{\ne}k$.  
Let $s,t_0,\tilde{t}_0 \in \bigl(0,\tau_+(w)\bigr)$ such that
$$
t_0\ne\tilde{t}_0,
\quad
\dot{\gamma}_w(t_0) \in S(\partial{D_j}),
\quad
\dot{\gamma}_w(\tilde{t}_0) \in S(\partial{D_k}). 
$$
Let $a(t;s)$ and $b(t;s)$ be solutions to 
\eqref{equation:aK} and \eqref{equation:bK} respectively, 
or \eqref{equation:akappa} and \eqref{equation:bkappa} respectively. 
Then Lemma~\ref{theorem:noconjugatepoint} ensures that 
$$
\Delta(t_0,\tilde{t}_0;s)
:=
a(t_0;s)b(\tilde{t}_0;s)-a(\tilde{t}_0;s)b(t_0;s)
\ne0. 
$$
Set 
\begin{align*}
  \eta 
& :=
  \bigl\langle\nu_j\bigl(\gamma_w(t_0)\bigr),\cdot\bigr\rangle 
  \in N^\ast_{\gamma_w(t_0)}(\partial{D_j}),
\\
  \tilde{\eta}
& :=
  \bigl\langle\nu_k\bigl(\gamma_w(\tilde{t}_0)\bigr),\cdot\bigr\rangle 
  \in N^\ast_{\gamma_w(\tilde{t}_0)}(\partial{D_k}),
\\
  \eta_1
& :=
  P(t_0,\tilde{t}_0;\gamma_w)^T\tilde{\eta} 
  \in T^\ast_{\gamma_w(t_0)}(M^\text{int}),
\\
  \eta(s)
& :=
  P(s,t_0;\gamma_w)^T\eta
  \in T^\ast_{\gamma_w(s)}(M^\text{int}),
\\
  \eta_1(s)
& :=
  P(s,t_0;\gamma_w)^T\eta_1
  =
  P(s,\tilde{t}_0;\gamma_w)^T\tilde{\eta}
  \in T^\ast_{\gamma_w(s)}(M^\text{int}). 
\end{align*}
Let $\xi$ and $\tilde{\xi}$ satisfy 
$$
\xi, \tilde{\xi} \in T^\ast_w\bigl(\partial_-S(M)\bigr), 
\quad
DF\vert^T_{\dot{\gamma}_w(t_0)}\xi
=
D\pi\vert^T_{\dot{\gamma}_w(t_0)}\eta, 
\quad
DF\vert^T_{\dot{\gamma}_w(\tilde{t}_0)}\tilde{\xi}
=
D\pi\vert^T_{\dot{\gamma}_w(\tilde{t}_0)}\tilde{\eta}.  
$$
The two covectors $\xi$ and $\tilde{\xi}$ are uniquely determined by 
$\eta$ and $\tilde{\eta}$ respectively. 
Then we have the following lemma regarding the existence or the nonexistence of 
pair of curves satisfying the desirable conditions. 
\begin{lemma}
\label{theorem:propagation}
Suppose {\rm ({\bf A0})}, {\rm ({\bf A1})} and {\rm ({\bf A2})}. 
\vspace{6pt}
\\
{\bf I}.\ 
Suppose that $\eta_1=\pm\eta$. 
If we set 
\begin{equation}
\xi(s)
:=
\frac{b(\tilde{t}_0;s)}{\Delta(t_0,\tilde{t}_0;s)}\xi
\mp
\frac{b(t_0;s)}{\Delta(t_0,\tilde{t}_0;s)}\tilde{\xi},
\quad
s\in \bigl(0,\tau_+(w)\bigr),
\label{equation:xiofs} 
\end{equation}
then  
\begin{equation}
DF\vert^T_{\dot{\gamma}_w(s)}\xi(s)
=
D\pi\vert^T_{\dot{\gamma}_w(s)}\eta(s),
\quad
s\in \bigl(0,\tau_+(w)\bigr). 
\label{equation:identity1}
\end{equation}
Moreover we have for any $s \in \bigl(0,\tau_+(w)\bigr)$
\begin{equation}
DF\vert^T_{\dot{\gamma}_w(s)}\bigl(N^\ast_w(\Sigma_{jk})\bigr)
\cap
D\pi\vert^T_{\dot{\gamma}_w(s)}\bigl(\mathcal{T}^\ast(s)\bigr)
=
\{\alpha D\pi\vert^T_{\dot{\gamma}_w(s)}\eta(s) : \alpha\in\mathbb{R}\},
\label{equation:identity2}
\end{equation}
where
$$
\mathcal{T}^\ast(s)
=
\bigl\{
\zeta \in T^\ast_{\gamma_w(s)}(M^\text{int}) : 
\zeta\bigl(\dot{\gamma}_w(s)\bigr)=0
\bigr\}. 
$$
{\bf II}.\ 
Suppose that $\eta_1\ne\pm\eta$. 
Then for any fixed 
$s \in \bigl(0,\tau_+(w)\bigr)\setminus\{t_0,\tilde{t}_0\}$, 
$$
DF\vert^T_{\dot{\gamma}_w(s)}\bigl(N^\ast_w(\Sigma_{jk})\bigr)
\cap
D\pi\vert^T_{\dot{\gamma}_w(s)}\bigl(\mathcal{T}^\ast(s)\bigr)
=
\{0\}.
$$
\end{lemma}
We remark that the condition $\eta_1=\pm\eta$ is equivalent 
to $w \in \Sigma_{jk}^{(1)}$ and 
to $\dot{\gamma}_w(t_0) \in \mathcal{M}_{jk}^{(\pm)}$ respectively, 
and that the condition $\eta_1\ne\pm\eta$ 
is equivalent to $w \in \Sigma_{jk}^{(2)}$ and to 
$\dot{\gamma}_w(t_0) \in S(\Omega_{jk})$ respectively. 
We show the figures of the curves $\xi(s)$ in $N^\ast_w(\Sigma_{jk})$ and 
$\eta(s)$ in $\bigcup_{s\in\bigl(0,\tau_+(w)\bigr)} T_{\gamma_w(s)}(M^\text{int})$, 
when $\eta_1=\eta$ and $(M,g)$ is a space of constant curvature $K=1,0,-1$ below. 
\begin{center}
\includegraphics[width=40mm]{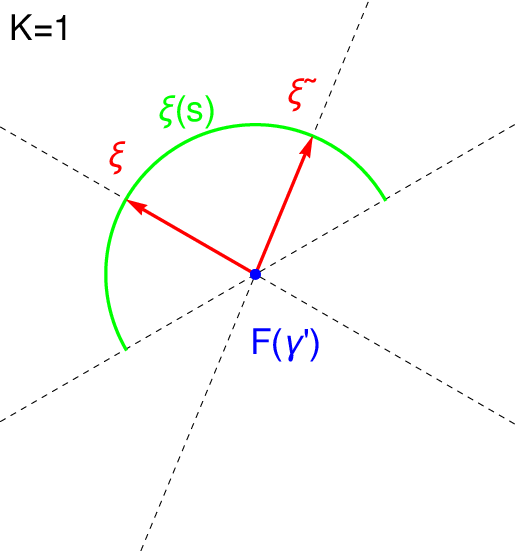}
\hspace{22pt}
\includegraphics[width=40mm]{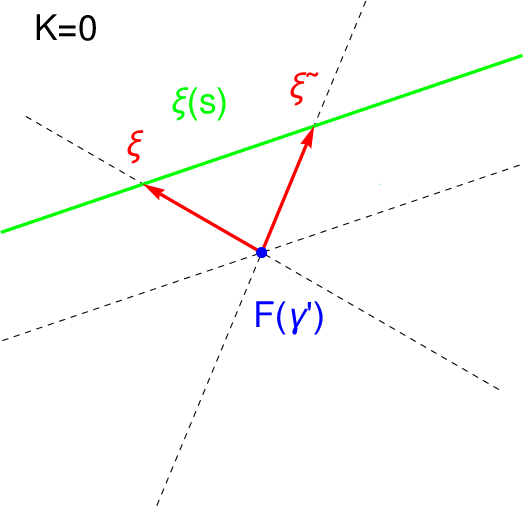} 
\vspace{11pt}\\
\includegraphics[width=40mm]{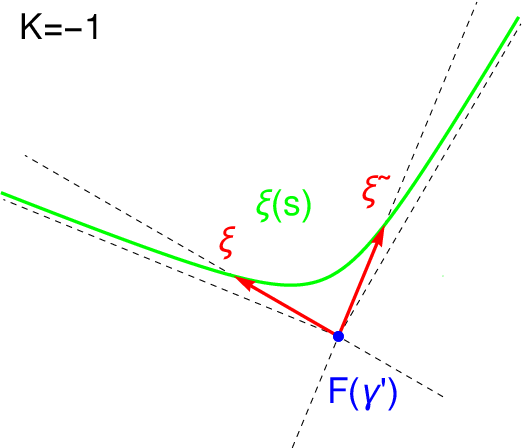}
\hspace{22pt}
\includegraphics[width=70mm]{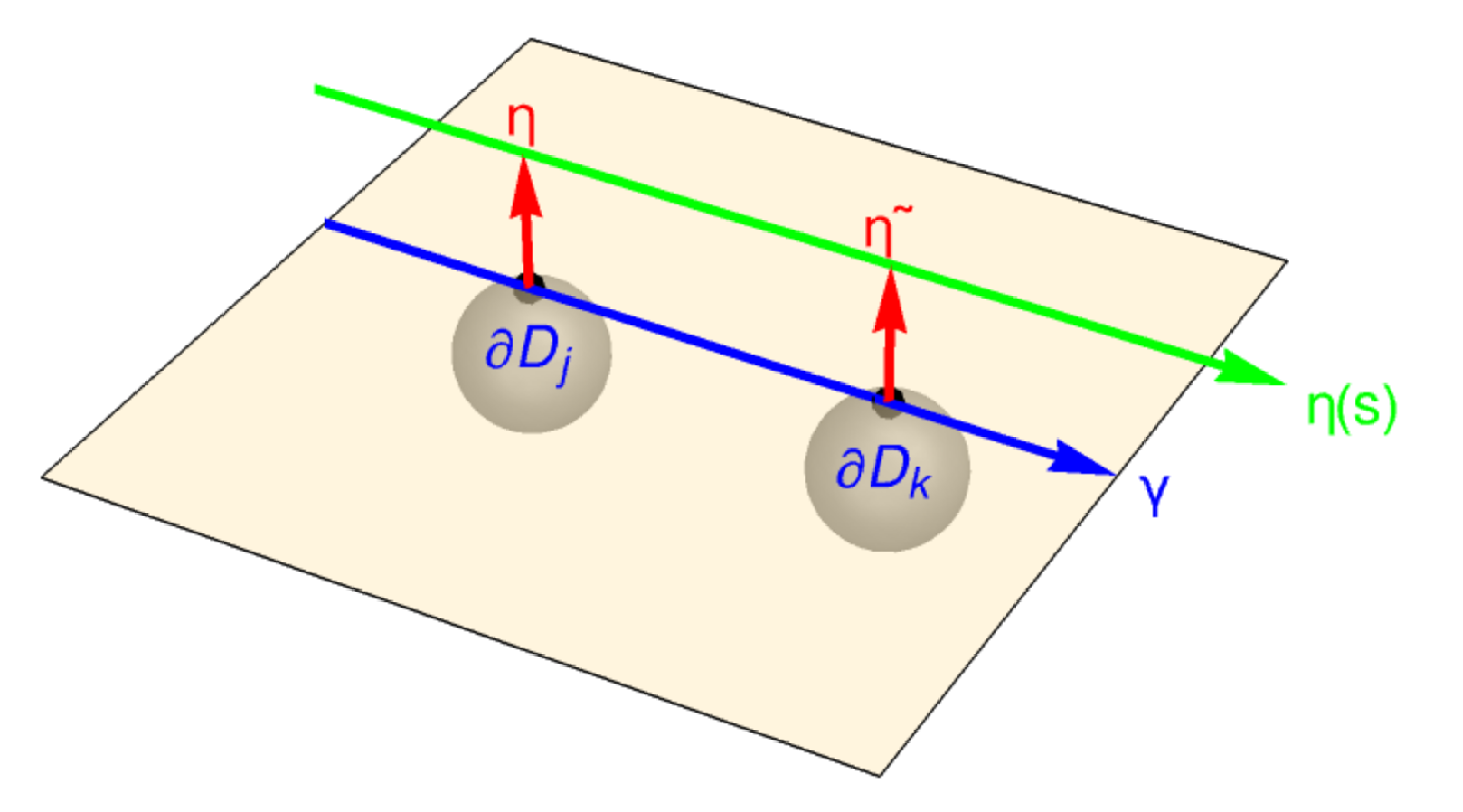}
\vspace{11pt}
\\
Figure~6.\ 
The curves $\xi(s)$ and $\eta(s)$ when $(M,g)$ is a space of constant curvature $K=1,0,-1$.
\end{center}
\begin{proof}[{\bf Proof of Lemma~\ref{theorem:propagation}}]
{\bf I}.\ 
Suppose that $\eta_1=\pm\eta$. Then $\eta_1(s)=\pm\eta(s)$. 
We first show \eqref{equation:identity1}. 
The definition of $\Delta(t_0,\tilde{t}_0;s)$ shows 
\begin{align}
\frac{b(\tilde{t}_0;s)}{\Delta(t_0,\tilde{t}_0;s)}
a(t_0;s)
-
\frac{b(t_0;s)}{\Delta(t_0,\tilde{t}_0;s)}
a(\tilde{t}_0;s)
& = 1,
\label{equation:a1}
\\
\frac{b(\tilde{t}_0;s)}{\Delta(t_0,\tilde{t}_0;s)}
b(t_0;s)
-
\frac{b(t_0;s)}{\Delta(t_0,\tilde{t}_0;s)}
b(\tilde{t}_0;s)
& = 0. 
\label{equation:b0}
\end{align}
Using \eqref{equation:tildexitildeeta} we deduce that 
\begin{align}
  DF\vert^T_{\dot{\gamma}_w(s)}\xi
& =
  D_v\Psi\vert^T_{\bigl(\dot{\gamma}_w(s),t_0-s\bigr)}
  \circ
  D\pi\vert^T_{\dot{\gamma}_w(t_0)}\eta,
\label{equation:dfdpi}
\\
  DF\vert^T_{\dot{\gamma}_w(s)}\tilde{\xi}
& =
  D_v\Psi\vert^T_{\bigl(\dot{\gamma}_w(s),\tilde{t}_0-s\bigr)}
  \circ
  D\pi\vert^T_{\dot{\gamma}_w(\tilde{t}_0)}\tilde{\eta}.
\label{equation:dfdpitilde} 
\end{align}
Let $\Xi=(X,\zeta) \in T_{\dot{\gamma}_w(s)}(S(M^\text{int}))$, 
and let $Z(t)$ be the Jacobi field along $\gamma_w$ with 
$\bigl(Z(s),{\nabla}Z(s)\bigr)=\Xi$. 
Lemma~\ref{theorem:jacobi2} shows that 
\begin{align}
  Z(t)
& =
  a(t;s)P(s,t;\gamma_w)X
  +
  b(t,s)P(s,t;\gamma_w)\zeta
\nonumber
\\
& =
  D\pi\vert_{\dot{\gamma}_w(t)}
  \circ
  D_v\Psi\vert_{\bigl(\dot{\gamma}_w(s),t-s\bigr)}
  \Xi. 
\label{equation:jacobiZ}
\end{align}
Recall \eqref{equation:xiofs}, which is the definition of $\xi(s)$. 
Using 
\eqref{equation:dfdpi}, 
\eqref{equation:dfdpitilde}, 
\eqref{equation:jacobiZ}, 
$\eta_1(s)=\pm\eta(s)$, 
\eqref{equation:a1} 
and 
\eqref{equation:b0} 
sequentially,  
we deduce that
\begin{align*}
  DF\vert^T_{\dot{\gamma}_w(s)}\xi(s)(\Xi)
& =
  \frac{b(\tilde{t}_0;s)}{\Delta(t_0,\tilde{t}_0;s)}
  DF\vert^T_{\dot{\gamma}_w(s)}\xi(\Xi)
  \mp
  \frac{b(t_0;s)}{\Delta(t_0,\tilde{t}_0;s)}
  DF\vert^T_{\dot{\gamma}_w(s)}\tilde{\xi}(\Xi)
\\
& =
  \frac{b(\tilde{t}_0;s)}{\Delta(t_0,\tilde{t}_0;s)}
  D_v\Psi\vert^T_{\bigl(\dot{\gamma}_w(s),t_0-s\bigr)}
  \circ
  D\pi\vert^T_{\dot{\gamma}_w(t_0)}\eta(\Xi)
\\
& \quad \mp
  \frac{b(t_0;s)}{\Delta(t_0,\tilde{t}_0;s)}
  D_v\Psi\vert^T_{\bigl(\dot{\gamma}_w(s),\tilde{t}_0-s\bigr)}
  \circ
  D\pi\vert^T_{\dot{\gamma}_w(\tilde{t}_0)}\tilde{\eta}(\Xi)
\\
& =
  \frac{b(\tilde{t}_0;s)}{\Delta(t_0,\tilde{t}_0;s)}
  \eta\left(
      D\pi\vert_{\dot{\gamma}_w(t_0)}
      \circ
      D_v\Psi\vert_{\bigl(\dot{\gamma}_w(s),t_0-s\bigr)}
      \Xi
      \right)
\\
& \quad \mp
  \frac{b(t_0;s)}{\Delta(t_0,\tilde{t}_0;s)}
  \tilde{\eta}\left(
              D\pi\vert_{\dot{\gamma}_w(\tilde{t}_0)}
              \circ
              D_v\Psi\vert_{\bigl(\dot{\gamma}_w(s),\tilde{t}_0-s\bigr)}
              \Xi
              \right)
\\
& =
  \frac{b(\tilde{t}_0;s)}{\Delta(t_0,\tilde{t}_0;s)}
  \eta\bigl(Z(t_0)\bigr)
  \mp
  \frac{b(t_0;s)}{\Delta(t_0,\tilde{t}_0;s)}
  \tilde{\eta}\bigl(Z(\tilde{t}_0)\bigr)
\\
& =
  \frac{b(\tilde{t}_0;s)}{\Delta(t_0,\tilde{t}_0;s)}
  \eta\bigl(a(t_0;s)P(s,t_0;\gamma_w)X
            +
            b(t_0,s)P(s,t_0;\gamma_w)\zeta
      \bigr)
\\
& \quad \mp
  \frac{b(t_0;s)}{\Delta(t_0,\tilde{t}_0;s)}
  \tilde{\eta}\bigl(
              a(\tilde{t}_0;s)P(s,\tilde{t}_0;\gamma_w)X
              +
              b(\tilde{t}_0,s)P(s,\tilde{t}_0;\gamma_w)\zeta
              \bigr)
\\
& =
  \frac{b(\tilde{t}_0;s)}{\Delta(t_0,\tilde{t}_0;s)}
  P(s,t_0;\gamma_w)^T\eta
  \bigl(a(t_0;s)X+b(t_0,s)\zeta\bigr)
\\
& \quad \mp
  \frac{b(t_0;s)}{\Delta(t_0,\tilde{t}_0;s)}
  P(s,\tilde{t}_0;\gamma_w)^T\tilde{\eta}
  \bigl(a(\tilde{t}_0;s)X+b(\tilde{t}_0,s)\zeta\bigr)
\\
& =
  \frac{b(\tilde{t}_0;s)}{\Delta(t_0,\tilde{t}_0;s)}
  P(s,t_0;\gamma_w)^T\eta
  \bigl(a(t_0;s)X+b(t_0,s)\zeta\bigr)
\\
& \quad \mp
  \frac{b(t_0;s)}{\Delta(t_0,\tilde{t}_0;s)}
  P(s,t_0;\gamma_w)^T\eta_1
  \bigl(a(\tilde{t}_0;s)X+b(\tilde{t}_0,s)\zeta\bigr)
\\
& =
  \frac{b(\tilde{t}_0;s)}{\Delta(t_0,\tilde{t}_0;s)}
  P(s,t_0;\gamma_w)^T\eta
  \bigl(a(t_0;s)X+b(t_0,s)\zeta\bigr)
\\
& \quad -
  \frac{b(t_0;s)}{\Delta(t_0,\tilde{t}_0;s)}
  P(s,t_0;\gamma_w)^T\eta
  \bigl(a(\tilde{t}_0;s)X+b(\tilde{t}_0,s)\zeta\bigr)
\\
& =
  \frac{b(\tilde{t}_0;s)}{\Delta(t_0,\tilde{t}_0;s)}
  \eta(s)
  \bigl(a(t_0;s)X+b(t_0,s)\zeta\bigr)
\\
& \quad -
  \frac{b(t_0;s)}{\Delta(t_0,\tilde{t}_0;s)}
  \eta(s)
  \bigl(a(\tilde{t}_0;s)X+b(\tilde{t}_0,s)\zeta\bigr)
\\
& =
  \left\{
  \frac{b(\tilde{t}_0;s)}{\Delta(t_0,\tilde{t}_0;s)}a(t_0;s)
  -
  \frac{b(t_0;s)}{\Delta(t_0,\tilde{t}_0;s)}a(\tilde{t}_0;s)
  \right\}
  \eta(s)(X)
\\
& \quad +
  \left\{
  \frac{b(\tilde{t}_0;s)}{\Delta(t_0,\tilde{t}_0;s)}b(t_0;s)
  -
  \frac{b(t_0;s)}{\Delta(t_0,\tilde{t}_0;s)}b(\tilde{t}_0;s)
  \right\}
  \eta(s)(\zeta)
\\
& =
  \eta(s)(X)
\\
& =
  \eta(s)\Bigl(D\pi\vert_{\dot{\gamma}_w(s)}\Xi\Bigr)
  \\
& =
  D\pi\vert^T_{\dot{\gamma}_w(s)}\eta(s)(\Xi),  
\end{align*} 
which proves \eqref{equation:identity1} and 
$$
DF\vert^T_{\dot{\gamma}_w(s)}\bigl(N^\ast_w(\Sigma_{jk})\bigr)
\cap
D\pi\vert^T_{\dot{\gamma}_w(s)}\bigl(\mathcal{T}^\ast(s)\bigr)
\supset
\{\alpha D\pi\vert^T_{\dot{\gamma}_w(s)}\eta(s) : \alpha\in\mathbb{R}\}. 
$$
It suffices to show the converse of the above inclusion relation. 
The conormal space $N^\ast_w(\Sigma_{jk})$ 
is spanned by $\xi$ and $\tilde{\xi}$. 
We have only to consider a linear combination 
$\alpha\xi+\beta\tilde{\xi}$ with $(\alpha,\beta)\ne(0,0)$. 
We show that if 
$DF\vert^T_{\dot{\gamma}_w(s)}(\alpha\xi+\beta\tilde{\xi}) \in D\pi\vert^T_{\dot{\gamma}_w(s)}\bigl(\mathcal{T}^\ast(s)\bigr)$, then 
\begin{equation}
\alpha b(t_0;s) \pm \beta b(\tilde{t}_0;s)=0, 
\label{equation:zerozero1} 
\end{equation}
\begin{equation}
DF\vert^T_{\dot{\gamma}_w(s)}(\alpha\xi+\beta\tilde{\xi})
=
D\pi\vert^T_{\dot{\gamma}_w(s)}
\Bigl(
\bigl(\alpha a(t_0;s)\pm \beta a(\tilde{t}_0;s)\bigr)
\eta(s)
\Bigr). 
\label{equation:zerozero2} 
\end{equation}
Let $\Xi=(X,\zeta) \in T_{\dot{\gamma}_w(s)}\bigl(M^\text{int})\bigr)$. 
In the same way as the above derivation of \eqref{equation:identity1}, 
we obtain 
\begin{align}
  DF\vert^T_{\dot{\gamma}_w(s)}(\alpha\xi+\beta\tilde{\xi})(\Xi)
& =
  \alpha\eta(s)\bigl(a(t_0;s)X+b(t_0;s)\zeta\bigr)
\nonumber
\\
& \quad +
  \beta\eta_1(s)\bigl(a(\tilde{t}_0;s)X+b(\tilde{t}_0;s)\zeta\bigr) 
\label{equation:5121}
\\
& =
  \alpha\eta(s)\bigl(a(t_0;s)X+b(t_0;s)\zeta\bigr)
\nonumber
\\
& \quad \pm
  \beta\eta(s)\bigl(a(\tilde{t}_0;s)X+b(\tilde{t}_0;s)\zeta\bigr) 
\nonumber
\\
& =
  \eta(s)
  \Bigl(\bigl(\alpha a(t_0;s)\pm\beta a(\tilde{t}_0;s)\bigr)X\Bigr)
\nonumber
\\
& \quad +
  \eta(s)
  \Bigl(\bigl(\alpha b(t_0;s)\pm\beta b(\tilde{t}_0;s)\bigr)\zeta\Bigr). 
\label{equation:zerozero3}
\end{align}
If 
$DF\vert^T_{\dot{\gamma}_w(s)}(\alpha\xi+\beta\tilde{\xi}) \in D\pi\vert^T_{\dot{\gamma}_w(s)}\bigl(\mathcal{T}^\ast(s)\bigr)$, 
then $DF\vert^T_{\dot{\gamma}_w(s)}(\alpha\xi+\beta\tilde{\xi})(\Xi)=0$ 
for $\Xi \in \operatorname{Ker}\bigl(D\pi\vert_{\dot{\gamma}_w(s)}\bigr)$,  
and \eqref{equation:zerozero3} shows \eqref{equation:zerozero1}. 
Hence \eqref{equation:zerozero3} becomes 
\begin{align*}
  DF\vert^T_{\dot{\gamma}_w(s)}(\alpha\xi+\beta\tilde{\xi})(\Xi) 
& =
  \eta(s)
  \Bigl(\bigl(\alpha a(t_0;s)\pm \beta a(\tilde{t}_0;s)\bigr)X\Bigr)
\\
& =
  \eta(s)
  \Bigl(
  \bigl(\alpha a(t_0;s)\pm \beta a(\tilde{t}_0;s)\bigr)
  D\pi\vert_{\dot{\gamma}_w(s)}\Xi
  \Bigr)
\\
& =
  D\pi\vert^T_{\dot{\gamma}_w(s)}
  \Bigl(
  \bigl(\alpha a(t_0;s)\pm \beta a(\tilde{t}_0;s)\bigr)\eta(s)
  \Bigr)
  (\Xi),
\end{align*}
which proves \eqref{equation:zerozero2}. 
\vspace{6pt}
\\
{\bf II}.\ 
Suppose that $s \ne t_0$, $s \ne \tilde{t}_0$ and $\eta_1\ne\pm\eta$. 
Then it follows that $b(t_0;s)\ne0$ and $b(\tilde{t}_0;s)\ne0$, 
and that $\eta$ and $\eta_1$ are linearly independent  
and so are $\eta(s)$ and $\eta_1(s)$.  
It suffices to show that 
$DF\vert^T_{\dot{\gamma}_w(s)}(\alpha\xi+\beta\tilde{\xi}) \not\in D\pi\vert^T_{\dot{\gamma}_w(s)}\bigl(\mathcal{T}^\ast(s)\bigr)$ 
for $(\alpha,\beta)\ne(0,0)$. 
In this setting we have 
$$
\theta
:=
\alpha b(t_0;s) \eta(s) 
+ 
\beta b(\tilde{t}_0;s) \eta_1(s)
\ne
0
\quad
\text{in}
\quad
T^\ast_{\dot{\gamma}_w(s)}(M^\text{int}). 
$$
For any $\Xi=(0,\zeta) \in \operatorname{Ker}\bigl(D\pi\vert_{\dot{\gamma}_w(s)}\bigr)$, \eqref{equation:5121} becomes
\begin{align*}
  DF\vert^T_{\dot{\gamma}_w(s)}(\alpha\xi+\beta\tilde{\xi})(\Xi)
& =
  \alpha\eta(s)\bigl(b(t_0;s)\zeta\bigr)
  +
  \beta\eta_1(s)\bigl(b(\tilde{t}_0;s)\zeta\bigr) 
\\
& = 
  \bigl(
  \alpha b(t_0;s) \eta(s) 
  + 
  \beta b(\tilde{t}_0;s) \eta_1(s)
  \bigr)
  (\zeta)
  =
  \theta(\zeta). 
\end{align*}
This implies that there exists 
$\Xi_0=(0,\zeta_0) \in \operatorname{Ker}\bigl(D\pi\vert_{\dot{\gamma}_w(s)}\bigr)$ such that 
$DF\vert^T_{\dot{\gamma}_w(s)}(\alpha\xi+\beta\tilde{\xi})(\Xi_0)=1$, 
and 
$DF\vert^T_{\dot{\gamma}_w(s)}(\alpha\xi+\beta\tilde{\xi}) \not\in D\pi\vert^T_{\dot{\gamma}_w(s)}\bigl(\mathcal{T}^\ast(s)\bigr)$. 
This completes the proof.
\end{proof}
\par
We can do the intersection calculus of 
$C_{\mathcal{X}^T}{\circ}N^\ast(\Sigma_j)$, 
$C_{\mathcal{X}^T}{\circ}N^\ast(\Sigma_{jk}^{(1)})$ 
and 
$C_{\mathcal{X}^T}{\circ}N^\ast(\Sigma_{jk}^{(2)})$ 
using Lemma~\ref{theorem:propagation}.   
\begin{lemma}
\label{theorem:compositionT}
Suppose 
{\rm ({\bf A0})}, 
{\rm ({\bf A1})} 
and  
{\rm ({\bf A2})}. 
Then the following compositions are all clean with some excess $e${\rm :}
\begin{alignat*}{2}
  C_{\mathcal{X}^T}{\circ}N^\ast(\Sigma_j)
& =
  N^\ast(\partial{D_j})\setminus0,
& \quad
  e
& =
  n-2,
\\ 
  C_{\mathcal{X}^T}{\circ}N^\ast(\Sigma_{jk}^{(1)})
& =
  N^\ast(\mathcal{L}_{jk})\setminus0,
& \quad
  e
& =
  0,
\\ 
  C_{\mathcal{X}^T}{\circ}N^\ast(\Sigma_{jk}^{(2)})
& =
  N^\ast(\Omega_{jk}\cup\Omega_{kj})\setminus0,
& \quad
  e
& =
  0.
\end{alignat*}
The last one is valid only for $n\geqq3$. 
\end{lemma}
We set 
\begin{align*}
  \mathcal{Z}
& =
  T^\ast(M^\text{int})
  \times
  T^\ast\bigl(\partial_-S(M)\bigr)
  \times
  T^\ast\bigl(\partial_-S(M)\bigr),
\\
  \mathcal{D}
& =
  T^\ast(M^\text{int})
  \times
  \Delta\Bigl(T^\ast\bigl(\partial_-S(M)\bigr)\Bigr)
\end{align*}
for short until the end of this section.  
Then $\operatorname{dim}(\mathcal{Z})=10n-8$ 
and $\operatorname{codim}(\mathcal{D})=4n-4$. 
Since 
$C_{\mathcal{X}^T}$, 
$N^\ast(\Sigma_j)$ 
and 
$N^\ast(\Sigma_{jk}^{(i)})$ 
are conic Lagrangian submanifolds, we have for $i=1,2$
$$
\operatorname{dim}
\bigl(C_{\mathcal{X}^T}{\times}N^\ast(\Sigma_j)\bigr)
=
\operatorname{dim}
\bigl(C_{\mathcal{X}^T}{\times}N^\ast(\Sigma_{jk}^{(i)})\bigr)
=
\operatorname{dim}(M^\text{int})
+
2
\operatorname{dim}\bigl(\partial_-S(M)\bigr)
=
5n-4.
$$
We split the proof of Lemma~\ref{theorem:compositionT} 
into three parts, and prove the lemma for each composition. 
\begin{proof}[{\bf Proof of Lemma~\ref{theorem:compositionT}, Part~1: $C_{\mathcal{X}^T}{\circ}N^\ast(\Sigma_j)$}] 
Theorem~\ref{theorem:adointofgeodesicxraytransform}, 
Lemma~\ref{theorem:4-1} 
and 
\eqref{equation:nstarsigmaj} 
imply that  
\begin{align*}
  C_{\mathcal{X}^T}{\times}N^\ast(\Sigma_j)
& =
  \bigl\{
  (\eta,\xi,\zeta): 
  \exists v \in S(M^\text{int}), 
  \exists u \in S(\partial{D_j}), 
  \exists \theta \in N^\ast_{\pi(u)}(\partial{D_j})
  \ \text{s.t.}
\\
& \qquad
  \xi \in T^\ast_{F(v)}\bigl(\partial_-S(M)\bigr)\setminus\{0\}, 
  \eta \in T^\ast_{\pi(v)}(M^\text{int})\setminus\{0\}, 
  DF\vert^T_v\xi=D\pi\vert^T_v\eta,
\\
& \qquad
  \zeta \in T^\ast_{F(u)}\bigl(\partial_-S(M)\bigr), 
  DF\vert^T_u\zeta=D\pi\vert^T_u\theta 
  \bigr\}. 
\end{align*}
Then we have 
\begin{align*}
  C_0
& :=
  \bigl(C_{\mathcal{X}^T}{\times}N^\ast(\Sigma_j)\bigr)\cap\mathcal{D}
\\
& =
  \bigl\{
  (\eta,\xi,\xi): 
  \exists v \in S(\partial{D_j})  
  \ \text{s.t.}\ 
  \eta \in N^\ast_{\pi(v)}(\partial{D_j})\setminus\{0\},
\\
& \qquad
  \xi \in T^\ast_{F(v)}\bigl(\partial_-S(M)\bigr)\setminus\{0\}, 
  DF\vert^T_v\xi=D\pi\vert^T_v\eta
  \bigr\},   
\end{align*}
$$
\operatorname{dim}(C_0)
=
\operatorname{dim}\bigl(S(\partial{D_j})\bigr)
+
\operatorname{dim}\bigl(N^\ast_{\pi(v)}(\partial{D_j})\bigr)
=
2n-2
$$
since $\xi$ is uniquely determined by $\eta$ 
if $DF\vert^T_v\xi=D\pi\vert^T_v\eta$. 
Hence we have 
$$
C_{\mathcal{X}^T}{\circ}N^\ast(\Sigma_j)
=
\{\eta: \exists v \in S(\partial{D_j})  
  \ \text{s.t.}\ 
  \eta \in N^\ast_{\pi(v)}(\partial{D_j})\setminus\{0\}\}
=
N^\ast(\partial{D_j})\setminus0,    
$$
\begin{align*}
  e
& =
  \operatorname{codim}(\mathcal{D})
  +
  \operatorname{codim}
  \bigl(C_{\mathcal{X}^T}{\times}N^\ast(\Sigma_j)\bigr)
  -
  \operatorname{codim}(C_0)
\\
& =
  \operatorname{codim}(\mathcal{D})
  -
  \operatorname{dim}
  \bigl(C_{\mathcal{X}^T}{\times}N^\ast(\Sigma_j)\bigr)
  +
  \operatorname{dim}(C_0) 
\\
& =
  (4n-4)-(5n-4)+(2n-2)
  =
  n-2.
\end{align*}
\par
It remains to show that $C_{\mathcal{X}^T}{\circ}N^\ast(\Sigma_j)$ is clean. 
Consider the projection map $\pi_0:C_0\ni(\eta,\xi,\xi)\mapsto\eta\in{T^\ast(M^\text{int})}$. 
Note that for any $\eta \in N^\ast(\partial{D_j})$
$$
\pi_0^{-1}(\eta)
=
\bigl\{
\bigr(\eta,(DF\vert^T_v)^{-1}{\circ}D\pi\vert^T_v\eta,(DF\vert^T_v)^{-1}{\circ}D\pi\vert^T_v\eta\bigr): 
v \in S_{\pi(\eta)}(\partial{D_j}) 
\bigr\}, 
$$
where $(DF\vert^T_v)^{-1}$ is the inverse of 
the linear map $DF\vert^T_v$ of $T^\ast_v\bigl(\partial_-S(M)\bigr)$ onto the range. 
Hence $\pi_0$ is proper and $\pi_0^{-1}(\eta)$ is connected for any $\eta \in N^\ast(\partial{D_j})$. 
\par
Finally we show that $C_0$ is a clean intersection. 
Fix arbitrary $c_0 \in C_0$ and $X=(X_1,X_2,X_3) \in T_{c_0}\bigl(C_{\mathcal{X}^T}{\times}N^\ast(\Sigma_j)\bigr)$, consider a curve in $C_{\mathcal{X}^T}{\times}N^\ast(\Sigma_j)$ of the form 
$$
\mathbb{R}\ni\alpha \mapsto 
\bigl(\eta(\alpha),\xi(\alpha),\zeta(\alpha)\bigr)
=
c_0+\alpha{X}+\mathcal{O}(\alpha^2)
=
c_0+\alpha(X_1,X_2,X_3)+\mathcal{O}(\alpha^2).
$$
Then we have 
$$
X
=
(X_1,X_2,X_3)
=
\frac{d}{d\alpha}\bigg\vert_{\alpha=0}
\bigl(\eta(\alpha),\xi(\alpha),\zeta(\alpha)\bigr). 
$$
If $X$ also belongs to $T_{c_0}(\mathcal{D})$, then $X_2=X_3$ and 
$$
X=\frac{d}{d\alpha}\bigg\vert_{\alpha=0}
\bigl(\eta(\alpha),\zeta(\alpha),\zeta(\alpha)\bigr)
\in T_{c_0}(C_0). 
$$
This completes the proof. 
\end{proof}
\begin{proof}[{\bf Proof of Lemma~\ref{theorem:compositionT}, Part~2: $C_{\mathcal{X}^T}{\circ}N^\ast(\Sigma_{jk}^{(1)})$}]
Theorem~\ref{theorem:adointofgeodesicxraytransform} and \eqref{equation:nstarsigmajk1} imply that  
\begin{align*}
  C_{\mathcal{X}^T}{\times}N^\ast(\Sigma_{jk}^{(1)})
& =
  \bigl\{
  (\eta,\xi,\zeta+\tilde{\zeta}): 
  \exists v \in S(M^\text{int}), 
  \exists (u,\tilde{u}) \in \Delta(\mathcal{M}_{jk}), 
\\
& \qquad
  \exists \theta \in N^\ast_{\pi(u)}(\partial{D_j}), 
  \exists \tilde{\theta} \in N^\ast_{\pi(\tilde{u})}(\partial{D_k}), 
  \ \text{s.t.}\ F(u)=F(\tilde{u})
\\
& \qquad
  \xi \in T^\ast_{F(v)}\bigl(\partial_-S(M)\bigr)\setminus\{0\}, 
  \eta \in T^\ast_{\pi(v)}(M^\text{int})\setminus\{0\}, 
  DF\vert^T_v\xi=D\pi\vert^T_v\eta,
\\
& \qquad
  \zeta \in T^\ast_{F(u)}\bigl(\partial_-S(M)\bigr), 
  DF\vert^T_u\zeta=D\pi\vert^T_u\theta,
\\
& \qquad
  \tilde{\zeta} \in T^\ast_{F(u)}\bigl(\partial_-S(M)\bigr), 
  DF\vert^T_{\tilde{u}}\tilde{\zeta}=D\pi\vert^T_{\tilde{u}}\tilde{\theta},  
  \bigr\}. 
\end{align*}
Set $C_1:=\bigl(C_{\mathcal{X}^T}{\times}N^\ast(\Sigma_{jk}^{(1)})\bigr)\cap\mathcal{D}$. 
If $(\eta,\xi,\zeta+\tilde{\zeta}) \in C_{\mathcal{X}^T}{\times}N^\ast(\Sigma_{jk}^{(1)}){\cap}\mathcal{D}$, 
then Part~I of Lemma~\ref{theorem:propagation} implies that $\xi=\zeta+\tilde{\zeta}$ must be of the form 
$$
\xi
=
(DF\vert^T_{\dot{\gamma}_v(s)})^{-1}
{\circ}
D\pi\vert^T_{\dot{\gamma}_v(s)}
{\circ}
P(s,0;\gamma_v)^T\eta
$$
for some $v \in \mathcal{M}_{jk}$, $\eta \in N^\ast_{\pi(v)}(\partial{D_j})$ and  $s\in\bigl(\tau_-(v),\tau_+(v)\bigr)$. 
Hence we have 
\begin{align*}
  C_1
& =
  \bigl\{
  (\eta^\prime,\xi,\xi):
  \exists v \in \mathcal{M}_{jk}, 
  \exists \eta \in N^\ast_{\pi(v)}(\partial{D_j})\setminus\{0\}, 
  \exists s \in \bigl(\tau_-(v),\tau_+(v)\bigr)
  \ \text{s.t.}\
\\ 
& \qquad
  \eta^\prime=P(s,0:\gamma_v)^T\eta, 
  DF\vert^T_{\dot{\gamma}_v(s)}\xi=D\pi\vert^T_{\dot{\gamma}_v(s)}\eta^\prime
  \bigr\}, 
\end{align*}
$$
\operatorname{dim}(C_1)
=
\operatorname{dim}(\mathcal{M}_{jk})
+
\operatorname{dim}\bigl(N^\ast_{\pi(v)}(\partial{D_j})\bigr)
+
\operatorname{dim}\Bigl(\bigl(\tau_-(v),\tau_+(v)\bigr)\Bigr)
=
n. 
$$
Therefore we obtain
\begin{align*}
  C_{\mathcal{X}^T}{\circ}N^\ast(\Sigma_{jk}^{(1)})
& =
  \bigl\{
  P(s,0;\gamma_v)^T\eta:
  v \in \mathcal{M}_{jk}, 
  \eta \in N^\ast_{\pi(v)}(\partial{D_j})\setminus\{0\}, 
  s \in \bigl(\tau_-(v),\tau_+(v)\bigr)
  \bigr\}
\\
& =
  N^\ast(\mathcal{L}_{jk})\setminus0, 
\end{align*}
\begin{align*}
  e
& =
  \operatorname{codim}(\mathcal{D})
  +
  \operatorname{codim}
  \bigl(C_{\mathcal{X}^T}{\times}N^\ast(\Sigma_{jk}^{(1)})\bigr)
  -
  \operatorname{codim}(C_1)
\\
& =
  \operatorname{codim}(\mathcal{D})
  -
  \operatorname{dim}
  \bigl(C_{\mathcal{X}^T}{\times}N^\ast(\Sigma_{jk}^{(1)})\bigr)
  +
  \operatorname{dim}(C_1) 
\\
& =
  (4n-4)-(5n-4)+n
  =
  0.
\end{align*}
In the same way as $\pi_0$, we can show that 
$\pi_1:C_1\ni(\eta,\xi,\xi)\mapsto\eta\in{T^\ast(M^\text{int})}$ is proper, 
$\pi_1^{-1}(\eta)$ is connected for any $\eta \in N^\ast(\mathcal{L}_{jk})$, 
and $C_1$ is a clean intersection. We omit the detail. This completes the proof. 
\end{proof}
\begin{proof}[{\bf Proof of Lemma~\ref{theorem:compositionT}, Part~3: $C_{\mathcal{X}^T}{\circ}N^\ast(\Sigma_{jk}^{(2)})$}]
Theorem~\ref{theorem:adointofgeodesicxraytransform} and \eqref{equation:nstarsigmajk2} imply that  
\begin{align*}
  C_{\mathcal{X}^T}{\times}N^\ast(\Sigma_{jk}^{(2)})
& =  
  \bigl\{
  (\eta,\xi,\zeta+\tilde{\zeta}): 
  \exists v \in S(M^\text{int}), 
  \exists (u,\tilde{u}) \in \Delta\bigl(S(\Omega_{jk})\bigr), 
\\
& \qquad
  \exists \theta \in N^\ast_{\pi(u)}(\partial{D_j}), 
  \exists \tilde{\theta} \in N^\ast_{\pi(\tilde{u})}(\partial{D_k}), 
  \ \text{s.t.}\ 
\\
& \qquad
  \xi \in T^\ast_{F(v)}\bigl(\partial_-S(M)\bigr)\setminus\{0\}, 
  \eta \in T^\ast_{\pi(v)}(M^\text{int})\setminus\{0\}, 
\\
& \qquad
  DF\vert^T_v\xi=D\pi\vert^T_v\eta,
\\
& \qquad
  \zeta \in T^\ast_{F(u)}\bigl(\partial_-S(M)\bigr), 
  DF\vert^T_u\zeta=D\pi\vert^T_u\theta,
\\
& \qquad
  \tilde{\zeta} \in T^\ast_{F(u)}\bigl(\partial_-S(M)\bigr), 
  DF\vert^T_{\tilde{u}}\tilde{\zeta}=D\pi\vert^T_{\tilde{u}}\tilde{\theta},  
  \bigr\}. 
\end{align*}
Set $C_2:=\bigl(C_{\mathcal{X}^T}{\times}N^\ast(\Sigma_{jk}^{(2)})\bigr)\cap\mathcal{D}$. 
If $(\eta,\xi,\zeta+\tilde{\zeta}) \in C_{\mathcal{X}^T}{\times}N^\ast(\Sigma_{jk}^{(2)}){\cap}\mathcal{D}$, 
then Part~II of Lemma~\ref{theorem:propagation} implies that $\xi=\zeta+\tilde{\zeta}$ must be one of the following: 
\begin{alignat*}{2}
  \xi
& =
  (DF\vert^T_u)^{-1}{\circ}D\pi\vert^T_u\eta, 
& \quad
  \eta
& \in
  N^\ast_{\pi(u)}(\Omega_{jk})\setminus\{0\},
\\
  \xi
& =
  (DF\vert^T_{\tilde{u}})^{-1}{\circ}D\pi\vert^T_{\tilde{u}}\tilde{\eta}, 
& \quad
  \tilde{\eta}
& \in
  N^\ast_{\pi(\tilde{u})}(\Omega_{kj})\setminus\{0\},
\end{alignat*}  
for $(u,\tilde{u}) \in \Delta\bigl(S(\Omega_{jk})\bigr)$. 
Hence we have 
\begin{align*}
  C_2
& =
  \bigl\{
  (\eta,\xi,\xi):
  \exists (u,\tilde{u}) \in \Delta\bigl(S(\Omega_{jk})\bigr)
  \ \text{s.t.}\  
  \eta \in N^\ast_{\pi(u)}(\Omega_{jk})\setminus\{0\},  
  DF\vert^T_u\xi=D\pi\vert^T_u\eta
  \bigr\}
\\
& \cup
  \bigl\{
  (\tilde{\eta},\tilde{\xi},\tilde{\xi}):
  \exists (\tilde{u},u) \in \Delta\bigl(S(\Omega_{kj})\bigr)
  \ \text{s.t.}\ 
  \tilde{\eta} \in N^\ast_{\pi(\tilde{u})}(\Omega_{kj})\setminus\{0\},  
  DF\vert^T_{\tilde{u}}\tilde{\xi}=D\pi\vert^T_{\tilde{u}}\tilde{\eta}
  \bigr\},
\end{align*}
$$
\operatorname{dim}(C_2)
=
\operatorname{dim}\bigl(N^\ast(\Omega_{jk})\bigr)
=
n. 
$$
Therefore we obtain
$C_{\mathcal{X}^T}{\circ}N^\ast(\Sigma_{jk}^{(2)})=N^\ast(\Omega_{jk}\cup\Omega_{kj})\setminus0$ and $e=0$. 
In the same way as $\pi_0$, we can show that 
$\pi_2:C_2\ni(\eta,\xi,\xi)\mapsto\eta\in{T^\ast(M^\text{int})}$ is proper, 
$\pi_2^{-1}(\eta)$ is connected for any $\eta \in N^\ast(\Omega_{jk}\cup\Omega_{kj})$, 
and $C_2$ is a clean intersection. We omit the detail. This completes the proof. 
\end{proof}
%
%
%%%%%%
%%%%%% section 6
%%%%%%
\section{Proof of the Main Theorem}
\label{section:proof}
Finally we prove Theorem~\ref{theorem:51} in several steps. 
Throughout this section we assume {\rm ({\bf A0})}, {\rm ({\bf A1})}, {\rm ({\bf A2})} and {\rm ({\bf A3})}.   
We study the quadratic interactions firstly, and  higher order interactions later.  
We begin by confirming the singularities of $1_D$ and $\mathcal{X}[1_D]$. 
We denote by $S^m(\mathbb{R}^N\times\mathbb{R}^n)$ 
the set of all smooth functions $a(x,\xi)$ of $(x,\xi)\in\mathbb{R}^N\times\mathbb{R}^k$ satisfying 
$$
\lvert\partial_x^\beta\partial_\xi^\alpha a(x,\xi)\rvert
\leqq
C_{\alpha,\beta}
(1+\lvert\xi\rvert)^{m-\lvert\alpha\rvert}. 
$$
\begin{lemma}
\label{theorem:6-1}
We have 
\begin{align}
  1_{D_j}
& \in 
  I^{-1/2-n/4}\bigl(N^\ast(\partial{D_j})\setminus0; \Omega^{1/2}_{M^\text{int}}\bigr),
\label{equation:6-1} 
\\
  \mathcal{X}[1_{D_j}]
& \in
  I^{-(n+1)/2}(N^\ast(\Sigma_j)\setminus0; \Omega_{\partial_-S(M)}^{1/2}) 
\label{equation:6-2}
\end{align}
for $j=1,\dotsc,J$. 
Moreover the principal symbols of these conormal distributions do not vanish on the associated conormal bundles.
\end{lemma}
\begin{proof}
We remark that the proof of \eqref{equation:6-1} is due to local argument 
in a small neighborhood at an arbitrary point $x \in \partial{D}_j$. 
Then \eqref{equation:6-1} can be proved by the exactly same symbolic calculus 
in \cite[Lemma~5.1]{Chihara} on the Euclidean space. 
Indeed $1_{D_j}/\lvert{dv_g}\rvert^{1/2}$ can be locally expressed as the characteristic function of a half space 
$\{(0,x_1,\dotsc,x_n) : x_2,\dotsc,x_n \in \mathbb{R}\}$ in $\mathbb{R}^n$ near $\partial{D_j}$, 
and is given by an oscillatory integral of the form 
$$
\frac{1_{D_j}}{\lvert{dv_g}\rvert}(x_1,\dotsc,x_n)
=
\int_{\mathbb{R}}
e^{ix_1\xi_1}
a(x_2,\dotsc,x_n,\xi_1)
d\xi_1,
$$
$$
a(x_2,\dotsc,x_n,\xi_1)
\in 
S^{-1}(\mathbb{R}^{n-1}\times\mathbb{R}),
\quad
a(x_2,\dotsc,x_n,\xi_1)
\simeq
\frac{1}{\xi_1}.  
$$
Then \eqref{equation:6-1} follows from \cite[Theorem~18.2.8]{Hoermander3}. 
The local symbol $a(x_2,\dotsc,x_n,\xi_1)$ of the conormal distribution 
$1_{D_j}$ is elliptic. 
Since $C_\mathcal{X}{\circ}N^\ast(\partial{D_j})=N^\ast(\Sigma_j)$ is transversal, and 
$$
\mathcal{X} \in I^{-n/4}\bigl(\partial_-S(M){\times}M^\text{int},C_\mathcal{X}^\prime; \Omega_{\partial_-S(M{\times}M^\text{int})}\bigr), 
$$
then \cite[Theorem~25.2.3]{Hoermander4} shows \eqref{equation:6-2}. 
The ellipticity of the conormal distribution $1_{D_j}$ implies that of $\mathcal{X}[1_{D_j}]$.  
\end{proof}
To see the interaction of singularities, we need to split the square into the several parts:
\begin{align}
  \left(
  \frac{\mathcal{X}[1_D]}{\lvert{d\mu}\rvert^{1/2}}
  \right)^2
  \lvert{d\mu}\rvert^{1/2}
& =
  \sum_{j=1}^J
  \left(
  \frac{\mathcal{X}[1_{D_j}]}{\lvert{d\mu}\rvert^{1/2}}
  \right)^2
  \lvert{d\mu}\rvert^{1/2}
\nonumber
\\
& +
  2
  \sum_{1 \leqq j < k \leqq J}
  \left(
  \frac{\mathcal{X}[1_{D_j}]}{\lvert{d\mu}\rvert^{1/2}}
  \cdot
  \frac{\mathcal{X}[1_{D_k}]}{\lvert{d\mu}\rvert^{1/2}}
  \right)
  \lvert{d\mu}\rvert^{1/2}
\label{equation:split1D}
\end{align}
For the first term of the right hand side of \eqref{equation:split1D} we have the following. 
\begin{lemma}
\label{theorem:6-2}
Let $Q$ be a parametrix of the elliptic pseudodifferential operator $\chi^T\circ\chi$ of order $-1$. 
If $u,v \in I^{-(n+1)/2}\bigl(N^\ast(\Sigma_j); \Omega_{\partial_-S(M)}^{1/2}\bigr)$, then 
\begin{align}
  \left(
  \frac{u}{\lvert{d\mu}\rvert^{1/2}}
  \cdot
  \frac{v}{\lvert{d\mu}\rvert^{1/2}}
  \right)
  \lvert{d\mu}\rvert^{1/2}
& \in 
  I^{-(n+1)/2}\bigl(N^\ast(\Sigma_j)\setminus0; \Omega_{\partial_-S(M)}^{1/2}\bigr),
\label{equation:6-3} 
\\
  Q\circ\mathcal{X}^T
  \left[
  \left(
  \frac{u}{\lvert{d\mu}\rvert^{1/2}}
  \cdot
  \frac{v}{\lvert{d\mu}\rvert^{1/2}}
  \right)
  \lvert{d\mu}\rvert^{1/2}
  \right]
& \in 
  I^{-1/2-n/4}
  \bigl(N^\ast(\partial{D_j})\setminus0; \Omega_{M^\text{int}}^{1/2}\bigr).
\label{equation:6-4} 
\end{align}
\end{lemma}
\begin{proof}
On one hand \eqref{equation:6-3} can be proved by the exactly same symbolic calculus 
for \cite[Lemma~5.2]{Chihara} on the Euclidean space. 
Set $N:=\operatorname{dim}\bigl(\partial_-S(M)\bigr)=2n-2$. 
Then there exist smooth symbols $a(x,\xi_1), b(x,\xi_1) \in S^{-3/2}(\mathbb{R}^N\times\mathbb{R})$ such that 
$$
\frac{u}{\lvert{d\mu}\rvert^{1/2}}(x)
=
\int_\mathbb{R}
e^{ix_1\xi_1}
a(x,\xi_1)
d\xi_1,
\quad
\frac{v}{\lvert{d\mu}\rvert^{1/2}}(x)
=
\int_\mathbb{R}
e^{ix_1\xi_1}
b(x,\xi_1)
d\xi_1
$$
locally. 
Hence we have 
\begin{align*}
  \left(
  \frac{u}{\lvert{d\mu}\rvert^{1/2}}
  \cdot
  \frac{v}{\lvert{d\mu}\rvert^{1/2}}
  \right)
  (x)
& =
  \int_{\mathbb{R}}
  e^{ix_1\xi_1}
  c(x,\xi_1)
  d\xi_1,
\\
  c(x,\xi_1)
& =
  \int_{\mathbb{R}}
  a(x,\xi_1-\eta_1)b(x,\eta_1)
  d\eta_1.
\end{align*}
We can show that $c(x,\xi_1) \in S^{-3/2}(\mathbb{R}^N\times\mathbb{R})$. 
We omit the detail. 
On the other hand \eqref{equation:6-4} immediately follows from \cite[Theorem~25.2.3]{Hoermander4} 
with Lemma~\ref{theorem:compositionT}, Lemma~\ref{theorem:adointofgeodesicxraytransform} and 
$$
Q\circ\mathcal{X}^T
\in
I^{1-n/4}
\bigl(
M^\text{int}{\times}\partial_-S(M),C_{\mathcal{X^T}}^\prime; \Omega_{M^\text{int}{\times}\partial_-S(M)}^{1/2}
\bigr). 
$$
We omit the detail. 
\end{proof}
Lemma~\ref{theorem:6-2} says that each 
$I^{-(n+1)/2}\bigl(N^\ast(\Sigma_j)\setminus0; \Omega_{\partial_-S(M)}^{1/2}\bigr)$, 
$j=1,\dotsc,J$ forms an algebra. Moreover 
$$
1_{D_j},
\quad
Q\circ\mathcal{X}^T\circ\mathcal{X}[1_{D_j}], 
\quad
Q\circ\mathcal{X}^T
\left[
\left(
\frac{\mathcal{X}[1_{D_j}]}{\lvert{d\mu}\rvert^{1/2}}
\right)^{2l}
\lvert{d\mu}\rvert^{1/2}
\right]
$$
belong to 
$$
I^{-1/2-n/4}\bigl(N^\ast(\partial{D_j})\setminus0; \Omega_{M^\text{int}}^{1/2}\bigr),
\quad
l=1,2,3,\dotsc, \quad j=1,\dotsc,J.
$$
To deal with the second term of the right hand side of \eqref{equation:split1D}, 
we now introduce classes of paired Lagrangian distributions 
(See \cite{MelroseUhlmann,GuilleminUhlmann,GreenleafUhlmann}), 
and quote the results on the product of conormal distributions, 
which are due to Greenleaf and Uhlmann in \cite{GreenleafUhlmann}.  
\begin{definition}
\label{theorem:6-3}
Let $X$ be a smooth manifold, and let $\mu, \nu \in \mathbb{R}$. 
Suppose that $\Lambda_0$ and $\Lambda_1$ 
are conic Lagrangian submanifolds of $T^\ast{X}\setminus0$, 
and that they intersect cleanly. 
We say that $u \in I^{\mu,\nu}(\Lambda_0,\Lambda_1;\Omega^{1/2}_X)$ 
if $u \in \mathcal{D}^\prime(\Omega^{1/2}_X)$, 
$\operatorname{WF}(u) \subset \Lambda_0\cup\Lambda_1$, and microlocally away from $\Lambda_0\cap\Lambda_1$ 
$$
I^{\mu,\nu}(\Lambda_0,\Lambda_1;\Omega^{1/2}_X)
\subset
I^{\mu+\nu}(\Lambda_0\setminus\Lambda_1;\Omega^{1/2}_X),
\quad
I^{\mu,\nu}(\Lambda_0,\Lambda_1;\Omega^{1/2}_X)
\subset
I^{\mu}(\Lambda_1;\Omega^{1/2}_X). 
$$
Such a distribution $u$ is said to be a paired Lagrangian distribution associated 
to the pair $(\Lambda_0,\Lambda)$ of order $(\mu,\nu)$. 
\end{definition}
We now define a product of half-densities. 
We use the same notation of Definition~\ref{theorem:6-3}, and 
denote by $\lvert{d\mu_X}\rvert$ a volume element of $X$. 
We define the set of formal products of 
$I^\mu(\Lambda_0;\Omega^{1/2}_X)$ and $I^\nu(\Lambda_1;\Omega^{1/2}_X)$ by 
\begin{align*}
& I^\mu(\Lambda_0;\Omega^{1/2}_X){\cdot}I^\nu(\Lambda_1;\Omega^{1/2}_X)
\\
:=
& \left\{
\frac{u}{\lvert{d\mu_X}\rvert^{1/2}}
\cdot
\frac{v}{\lvert{d\mu_X}\rvert^{1/2}}
\cdot
\lvert{d\mu_X}\rvert^{1/2}, 
u \in I^\mu(\Lambda_0;\Omega^{1/2}_X), 
v \in I^\nu(\Lambda_1;\Omega^{1/2}_X) 
\right\}. 
\end{align*}
\par
We use the following lemma to deal with the product of conormal distributions:
\begin{lemma}[{\bf Greenleaf and Uhlmann \cite[Lemma~1.1]{GreenleafUhlmann}}]
\label{theorem:6-4} 
Let $X$ be an $N$-dimensional smooth manifold. 
Suppose that $Y$ and $Z$ are closed submanifolds of $X$ such that 
$\operatorname{codim}(Y)=d_1$, $\operatorname{codim}(Y)=d_2$, and 
$Y$ and $Z$ intersect transversely, that is, $\operatorname{codim}(Y{\cap}Z)=d_1+d_2$. 
Let $m_1$ and $m_2$ be real numbers. 
Then we have 
\begin{align}
& I^{m_1+d_1/2-N/4}\bigl(N^\ast(Y)\setminus0; \Omega_X^{1/2}\bigr)
  \cdot
  I^{m_2+d_2/2-N/4}\bigl(N^\ast(Z)\setminus0; \Omega_X^{1/2}\bigr)
\nonumber
\\
  \subset
& I^{m_1+d_1/2-N/4, m_2+d_2/2}\bigl(N^\ast(Y{\cap}Z)\setminus0, N^\ast(Y)\setminus0; \Omega_X^{1/2}\bigr)
\nonumber
\\
  +
& I^{m_2+d_2/2-N/4, m_1+d_1/2}\bigl(N^\ast(Y{\cap}Z)\setminus0, N^\ast(Z)\setminus0; \Omega_X^{1/2}\bigr).
\label{equation:productconormal}
\end{align}
\end{lemma}
Combining \cite[Theorem~25.2.3]{Hoermander4} and Lemma~\ref{theorem:6-4}  
we have the following. 
\begin{lemma}
\label{theorem:6-5}
Let $X$ and $Y$ be smooth manifolds, and let $m$, $\mu$ and $\nu$ be real numbers. 
Suppose that the distribution kernel of a Fourier integral operator $A$ belongs to 
$I^m(X{\times}Y,C^\prime; \Omega_{X{\times}Y}^{1/2})$ with a canonical relation $C$. 
Suppose that $\Lambda_0$ and $\Lambda_1$ are conic Lagrangian submanifolds of $T^\ast(Y)$ such that 
$\Lambda_0$ and $\Lambda_1$ intersect cleanly, and $C\circ\Lambda_i$ 
is clean with excess $e_i$ for $i=0,1$. 
If $u \in I^{\mu, \nu}(\Lambda_0, \Lambda_1; \Omega_Y^{1/2})$, then 
$$
Au 
\in 
I^{m+\mu+e_1/2, \nu+(e_0-e_1)/2}(C\circ\Lambda_0, C\circ\Lambda_1; \Omega_X^{1/2}). 
$$
\end{lemma}
For the second term of the right hand side of \eqref{equation:split1D}, we have the following.
\begin{lemma}
\label{theorem:6-6}
Assume that $j{\ne}k$. Then we have 
\begin{equation}
\left(
\frac{\mathcal{X}[1_{D_j}]}{\lvert{d\mu}\rvert^{1/2}}
\cdot 
\frac{\mathcal{X}[1_{D_k}]}{\lvert{d\mu}\rvert^{1/2}}
\right)
\lvert{d\mu}\rvert^{1/2}
\in 
\mathcal{A}_{jk},
\label{equation:6-5}
\end{equation}
\begin{align*}
  \mathcal{A}_{jk}
& :=
  I^{-(n+1)/2, -1}\bigl(N^\ast(\Sigma_{jk})\setminus0,N^\ast(\Sigma_j)\setminus0; \Omega_{\partial_-S(M)}^{1/2}\bigr)
\\
& +
  I^{-(n+1)/2, -1}\bigl(N^\ast(\Sigma_{jk})\setminus0,N^\ast(\Sigma_k)\setminus0; \Omega_{\partial_-S(M)}^{1/2}\bigr),  
\end{align*}
and the principal symbol of this product does not vanish on $N^\ast(\Sigma_{jk})\setminus0$. 
Furthermore we have 
\begin{equation}
Q{\circ}\mathcal{X}^T
\left[
\left(
\frac{\mathcal{X}[1_{D_j}]}{\lvert{d\mu}\rvert^{1/2}}
\cdot 
\frac{\mathcal{X}[1_{D_k}]}{\lvert{d\mu}\rvert^{1/2}}
\right)
\lvert{d\mu}\rvert^{1/2}
\right]
\in 
\mathcal{W}_{jk},
\label{equation:6-6}
\end{equation}
\begin{align*}
  \mathcal{W}_{jk}
& :=
  I^{-1/2-n/4, -n/2}\bigl(N^\ast(\mathcal{L}_{jk})\setminus0,N^\ast(\partial{D_j})\setminus0; \Omega_{M^\text{int}}^{1/2}\bigr)
\\
& +
  I^{-1/2-n/4, -n/2}\bigl(N^\ast(\mathcal{L}_{jk})\setminus0,N^\ast(\partial{D_k})\setminus0; \Omega_{M^\text{int}}^{1/2}\bigr)
\\
& =
  I^{-1/2-n/4, -n/2}
  \bigl(
  N^\ast(\mathcal{L}_{jk})\setminus0,
  N^\ast(\partial{D_j}\cup\partial{D_k})\setminus0; 
  \Omega_{M^\text{int}}^{1/2}
  \bigr).  
\end{align*}
and the principal symbol of this product does not vanish on $N^\ast(\mathcal{L}_{jk})\setminus0$. 
\end{lemma}
\begin{proof}
\eqref{equation:6-5} immediately follows from 
Lemma~\ref{theorem:6-1} and Lemma~\ref{theorem:6-4} 
with $N=2n-2$, $m_1=m_2=-3/2$ and $d_1=d_2=1$. 
Moreover Lemma~\ref{theorem:6-1} shows that $\mathcal{X}[1_{D_j}]$ and $\mathcal{X}[1_{D_k}]$ are 
elliptic conormal distributions with 
$\operatorname{WF}(\mathcal{X}[1_{D_j}]) \subset N^\ast(\Sigma_j)\setminus0$ 
and 
$\operatorname{WF}(\mathcal{X}[1_{D_k}]) \subset N^\ast(\Sigma_k)\setminus0$. 
Then it follows that the product \eqref{equation:6-5} is also an elliptic conormal distribution 
restricted on $N^\ast(\Sigma_{jk})\setminus0$. 
\par
Next we prove \eqref{equation:6-6}. 
We apply 
Theorem~\ref{theorem:adointofgeodesicxraytransform},  
Lemma~\ref{theorem:compositionT} 
and  Lemma~\ref{theorem:6-5} to \eqref{equation:6-5} with 
$m=-n/4+1$, $\mu=-(n+1)/2$, $\nu=-1$, $e_1=n-2$ and $e_0=0$, that is, 
$$
m+\mu+\frac{e_1}{2}=-\frac{1}{2}-\frac{n}{4},
\quad
\nu+\frac{e_0-e_1}{2}=-\frac{n}{2}.           
$$
We deduce that
\begin{align*}
& Q{\circ}\mathcal{X}^T
  \left[
  \left(
  \frac{\mathcal{X}[1_{D_j}]}{\lvert{d\mu}\rvert^{1/2}}
  \cdot 
  \frac{\mathcal{X}[1_{D_k}]}{\lvert{d\mu}\rvert^{1/2}}
  \right)
  \lvert{d\mu}\rvert^{1/2}
  \right]
\\
& \in
  I^{-1/2-n/4, -n/2}
  \bigl(
  C_{\mathcal{X}^T}{\circ}N^\ast(\Sigma_{jk})\setminus0, 
  C_{\mathcal{X}^T}{\circ}N^\ast(\Sigma_j)\setminus0; 
  \Omega_{M^\text{int}}^{1/2}
  \bigr)
\\
& \quad +
  I^{-1/2-n/4, -n/2}
  \bigl(
  C_{\mathcal{X}^T}{\circ}N^\ast(\Sigma_{jk})\setminus0, 
  C_{\mathcal{X}^T}{\circ}N^\ast(\Sigma_k)\setminus0; 
  \Omega_{M^\text{int}}^{1/2}
  \bigr)
\\
& =
  I^{-1/2-n/4, -n/2}
  \bigl(
  C_{\mathcal{X}^T}{\circ}N^\ast(\Sigma_{jk}^{(1)})\setminus0, 
  C_{\mathcal{X}^T}{\circ}N^\ast(\Sigma_j)\setminus0; 
  \Omega_{M^\text{int}}^{1/2}
  \bigr)
\\
& \quad +
  I^{-1/2-n/4, -n/2}
  \bigl(
  C_{\mathcal{X}^T}{\circ}N^\ast(\Sigma_{jk}^{(2)})\setminus0, 
  C_{\mathcal{X}^T}{\circ}N^\ast(\Sigma_j)\setminus0; 
  \Omega_{M^\text{int}}^{1/2}
  \bigr)
\\
& \quad \quad +
  I^{-1/2-n/4, -n/2}
  \bigl(
  C_{\mathcal{X}^T}{\circ}N^\ast(\Sigma_{jk}^{(1)})\setminus0, 
  C_{\mathcal{X}^T}{\circ}N^\ast(\Sigma_k)\setminus0; 
  \Omega_{M^\text{int}}^{1/2}
  \bigr)
\\
& \quad \quad \quad +
  I^{-1/2-n/4, -n/2}
  \bigl(
  C_{\mathcal{X}^T}{\circ}N^\ast(\Sigma_{jk}^{(2)})\setminus0, 
  C_{\mathcal{X}^T}{\circ}N^\ast(\Sigma_k)\setminus0; 
  \Omega_{M^\text{int}}^{1/2}
  \bigr) 
\\
& =
  I^{-1/2-n/4, -n/2}
  \bigl(
  N^\ast(\mathcal{L}_{jk})\setminus0, 
  N^\ast(\partial{D_j})\setminus0; 
  \Omega_{M^\text{int}}^{1/2}
  \bigr)
\\
& \quad +
  I^{-1/2-n/4, -n/2}
  \bigl(
  N^\ast(\Omega_{jk}\cup\Omega_{kj})\setminus0, 
  N^\ast(\partial{D_j})\setminus0; 
  \Omega_{M^\text{int}}^{1/2}
  \bigr)
\\
& \quad \quad +
  I^{-1/2-n/4, -n/2}
  \bigl(
  N^\ast(\mathcal{L}_{jk})\setminus0, 
  N^\ast(\partial{D_k})\setminus0; 
  \Omega_{M^\text{int}}^{1/2}
  \bigr)
\\
& \quad \quad \quad +
  I^{-1/2-n/4, -n/2}
  \bigl(
  N^\ast(\Omega_{jk}\cup\Omega_{kj})\setminus0, 
  N^\ast(\partial{D_k})\setminus0; 
  \Omega_{M^\text{int}}^{1/2}
  \bigr)
\\
& =
  \mathcal{W}_{jk}
\\
& \quad +
  I^{-1/2-n/4, -n/2}
  \bigl(
  N^\ast(\Omega_{jk}\cup\Omega_{kj})\setminus0, 
  N^\ast(\partial{D_j})\setminus0; 
  \Omega_{M^\text{int}}^{1/2}
  \bigr)
\\
& \quad \quad +
  I^{-1/2-n/4, -n/2}
  \bigl(
  N^\ast(\Omega_{jk}\cup\Omega_{kj})\setminus0, 
  N^\ast(\partial{D_k})\setminus0; 
  \Omega_{M^\text{int}}^{1/2}
  \bigr)
\\
& \subset
  \mathcal{W}_{jk}. 
\end{align*}
Recall that the product \eqref{equation:6-5} is an elliptic conormal distribution restricted on 
$N^\ast(\Sigma_{jk})$ and Lemma~\ref{theorem:compositionT}. Then we deduce that the product 
\eqref{equation:6-6} is also an elliptic conormal distribution restricted on 
$N^\ast(\mathcal{L}_{jk})\setminus0$ 
since $Q\circ\mathcal{X}^T$ is an elliptic Fourier integral operator.   
This completes the proof. 
\end{proof}
In what follows we study higher order terms in $P_\text{MA}$. 
We now introduce notation:
$$
\mathcal{A}
:=
\sum_{1 \leqq j < k \leqq J}
\mathcal{A}_{jk},
\quad
\mathcal{W}
:=
\sum_{1 \leqq j < k \leqq J}
\mathcal{W}_{jk}. 
$$
We set for an integer $r\geqq2$ 
$$
h_k
:=
\left(
\frac{\chi[1_{D_k}]}{\lvert{d\mu}\rvert^{1/2}}
\right)^{p_k}
\lvert{d\mu}\rvert^{1/2},
\quad
k=1,\dotsc,r,
\quad
p_k=1,2,3,\dotsc. 
$$
Also we set 
$$
\Sigma_{k_1 \dotsc k_m}
:=
\bigcap_{j=1}^m\Sigma_{k_j}, 
\quad
1 \leqq k_1 < k_2 < \dotsb <k_m \leqq r. 
$$
Lemma~\ref{equation:split1D} implies that 
$h_k \in I^{-(n+1)/2}\bigl(N^\ast(\Sigma_k)\setminus0;\Omega_{\partial_-S(M)}^{1/2}\bigr)$ 
for $k=1,\dotsc,r$. 
Using Lemma~\ref{theorem:6-4}, we have 
\begin{align*}
  \operatorname{WF}\left(
  \left\{\prod_{k=1}^2\frac{h_k}{\lvert{du}\rvert^{1/2}}\right\}
  \lvert{du}\rvert^{1/2}
  \right)
& \subset 
  N^\ast(\Sigma_1)\setminus0 
  \cup 
  N^\ast(\Sigma_2)\setminus0 
  \cup 
  N^\ast(\Sigma_{12})\setminus0,
\\
  \operatorname{WF}\left(
  \left\{\prod_{k=1}^3\frac{h_k}{\lvert{du}\rvert^{1/2}}\right\}
  \lvert{du}\rvert^{1/2}
  \right)
& \subset 
  \bigl(
  N^\ast(\Sigma_1)\setminus0 
  \cup 
  N^\ast(\Sigma_2)\setminus0 
  \cup 
  N^\ast(\Sigma_{12})\setminus0
  \bigr)
\\
& \cup
  \bigl(
  N^\ast(\Sigma_{13})\setminus0 
  \cup 
  N^\ast(\Sigma_{23})\setminus0 
  \cup 
  N^\ast(\Sigma_{123})\setminus0
  \bigr)
\\
& =
  \bigl(
  N^\ast(\Sigma_1)\setminus0 
  \cup 
  N^\ast(\Sigma_2)\setminus0 
  \cup 
  N^\ast(\Sigma_3)\setminus0
  \bigr)
\\
& \cup
  \bigl(
  N^\ast(\Sigma_{12})\setminus0 
  \cup 
  N^\ast(\Sigma_{13})\setminus0 
  \cup 
  N^\ast(\Sigma_{23})\setminus0
  \bigr)
\\
& \cup
  N^\ast(\Sigma_{123})\setminus0.
\end{align*}
Then we deduce that 
$$
\operatorname{WF}\left(
\left\{\prod_{k=1}^r\frac{h_k}{\lvert{du}\rvert^{1/2}}\right\}
\lvert{du}\rvert^{1/2}
\right)
\subset
\bigcup_{m=1}^r
\left(
\bigcup_{1 \leqq k_1 < k_2 < \dotsb <k_m \leqq r}
N^\ast(\Sigma_{k_1 \dotsc k_m})\setminus0
\right). 
$$
We study the more detail about this. 
So we need the extension of Lemma~\ref{theorem:propagation} for the higher order terms. 
We need some notation as follows. 
For any $w \in \Sigma_{12{\dotsc}r} \subset \partial_-S(M)$, 
there exist $t_k \in (0,\tau_+(w))$ and $u_k \in S(\partial{D_k})$ for $k=1,\dotsc, r$ such that 
$t_j{\ne}t_k$ for $j{\ne}k$, 
$F(u_k)=w$ and $u_k=\Psi(w,t_k)=\dot{\gamma}_w(t_k)$ for $k=1,\dotsc, r$. 
We set 
\begin{align*}
  \tilde{\eta}_k
& :=
  \bigl\langle\nu_k\bigl(\pi(u_k)\bigr),\cdot\bigr\rangle
  \in
  N^\ast_{\pi(u_k)}(\partial{D_k}), 
\\
  \eta_k
& :=
  P(t_1,t_k;\gamma_w)^T\tilde{\eta_k}
  \in
  T^\ast_{\pi(u_1)}(M^\text{int}),
\\
  \xi_k
& :=
  (DF\vert^T_{u_k})^{-1}
  \circ
  D\pi\vert^T_{u_k}
  \tilde{\eta}_k
  \in
  T^\ast_w\bigl(\partial_-S(M)\bigr). 
\end{align*}
Note that $\eta_1=\tilde{\eta}_1$. We have the following. 
\begin{lemma}
\label{theorem:propagation2}
Suppose that $\eta_1, \dotsc, \eta_r$ are linearly independent. 
\begin{itemize}
\item 
If $s \in \bigl(0,\tau_+(w)\bigr)\setminus\{t_1,\dotsc,t_r\}$ and 
$(\alpha_1,\dotsc,\alpha_r)\ne(0,\dotsc,0)$, then there exists 
$\Xi_0=(0,\zeta_0) \in \operatorname{Ker}(D\pi\vert_{\dot{\gamma}_w(s)})\setminus\{0\}$ 
such that 
$DF\vert^T_{\dot{\gamma}_w(s)}(\alpha_1\xi_1+\dotsb+\alpha_r\xi_r)(\Xi_0)=1$, that is, 
$DF\vert^T_{\dot{\gamma}_w(s)}(\alpha_1\xi_1+\dotsb+\alpha_r\xi_r) \not\in D\pi\vert_{\dot{\gamma}_w(s)}^T\bigl(T^\ast_{\gamma_w(s)}(M^\text{int})\bigr)$.
\item 
If at least two of $\alpha_1,\dotsc,\alpha_r$ are not zero, then for any $s \in \bigl(0,\tau_+(w)\bigr)$, 
then there exists 
$\Xi_1=(0,\zeta_1) \in \operatorname{Ker}(D\pi\vert_{\dot{\gamma}_w(s)})\setminus\{0\}$ 
such that 
$DF\vert^T_{\dot{\gamma}_w(s)}(\alpha_1\xi_1+\dotsb+\alpha_r\xi_r)(\Xi_1)=1$, that is, 
$DF\vert^T_{\dot{\gamma}_w(s)}(\alpha_1\xi_1+\dotsb+\alpha_r\xi_r) \not\in D\pi\vert_{\dot{\gamma}_w(s)}^T\bigl(T^\ast_{\gamma_w(s)}(M^\text{int})\bigr)$.
\end{itemize}
\end{lemma}
\begin{proof}
Fix arbitrary $s \in \bigl(0,\tau_+(w)\bigr)$. 
In the same way as \eqref{equation:dfdpi}-\eqref{equation:dfdpitilde}, we have 
\begin{align*}
  DF\vert^T_{\dot{\gamma}_w(s)}\xi_k
& =
  D_v\Psi\vert^T_{\bigl(\dot{\gamma}_w(s),t_k-s\bigr)}
  \circ
  DF\vert^T_{u_k}\xi_k
\\
& =
  D_v\Psi\vert^T_{\bigl(\dot{\gamma}_w(s),t_k-s\bigr)}
  \circ
  D\pi\vert^T_{u_k}\tilde{\eta}_k,
  \quad
  k=1,\dotsc,r. 
\end{align*}
For any $\Xi=(X,\zeta) \in T_{\dot{\gamma}_w(s)}\bigl(S(M^\text{int})\bigr)$, we deduce that 
$$
D\pi\vert_{\Psi(w,t)}
\circ
D_v\Psi\vert_{\bigl(\dot{\gamma}_w(s),t-s\bigr)}\Xi
=
a(t;s)P(s,t;\gamma_w)X+b(t;s)P(s,t;\gamma_w)\zeta. 
$$
Then for any $(0,\zeta) \in \operatorname{Ker}(D\pi\vert_{\dot{\gamma}_w(s)})$, we deduce that 
\begin{align*}
  DF\vert^T_{\dot{\gamma}_w(s)}\left(\sum_{k=1}^r\alpha_k\xi_k\right)(0,\zeta)
& =
  \sum_{k=1}^r
  \alpha_k
  DF\vert^T_{\dot{\gamma}_w(s)}\xi_k(0,\zeta)
\\
& =
  \sum_{k=1}^r
  \alpha_k 
  D_v\Psi\vert^T_{\bigl(\dot{\gamma}_w(s),t_k-s\bigr)}
  \circ
  D\pi\vert^T_{u_k}\tilde{\eta}_k(0,\zeta)
\\
& =
  \sum_{k=1}^r
  \alpha_k
  \tilde{\eta}_k
  \left(
  D\pi\vert_{u_k}
  \circ
  D_v\Psi\vert_{\bigl(\dot{\gamma}_w(s),t_k-s\bigr)}
  (0,\zeta)
  \right) 
\\
& =
  \sum_{k=1}^r
  \alpha_k
  \tilde{\eta}_k
  \bigl(
  b(t_k;s)P(s,t_k;\gamma_w)\zeta
  \bigr)
\\
& =
  \sum_{k=1}^r
  \alpha_k
  b(t_k;s)
  P(s,t_k;\gamma_w)^T\tilde{\eta_k}(\zeta)
\\
& =
  \left(
  \sum_{k=1}^r
  \alpha_k
  b(t_k;s)
  \eta_k(s)
  \right)(\zeta), 
\end{align*}
where $\eta_k(s):=P(s,t_k;\gamma_w)^T\tilde{\eta_k}$. 
Note that $\eta_1(s),\dotsc,\eta_r(s)$ are linearly independent 
for any $s\in\bigl(0,\tau_+(w)\bigr)$ since 
$\eta_k(t_1)=\eta_k$ and $\eta_1,\dotsc,\eta_r$ are linearly independent.    
\par
Suppose that $\sum_{k=1}^r\alpha_k\xi_k=0$ in $T^\ast_w\bigl(\partial_-S(M)\bigr)$. 
Then we have 
$$
\left(
\sum_{k=1}^r
\alpha_k
b(t_k;s)
\eta_k(s)
\right)(\zeta)
=
0
$$
for any $\zeta \in T_{\gamma_w(s)}(M^\text{int})$ and for any $s\in\bigl(0,\tau_+(w)\bigr)$, 
and therefore we deduce that 
$$
\sum_{k=1}^r
\alpha_k
b(t_k;s)
\eta_k(s)
=
0
\quad
\text{in}
\quad
T^\ast_{\gamma_w(s)}(M^\text{int})
$$
for any $s \in \bigl(0,\tau_+(w)\bigr)$. 
Hence we obtain $\alpha_kb(t_k;s)=0$, $k=1,\dotsc,r$ 
for any $s\in\bigl(0,\tau_+(w)\bigr)$ 
since $\eta_1(s),\dotsc,\eta_r(s)$ are linearly independent 
for any $s\in\bigl(0,\tau_+(w)\bigr)$.  
\par
If $s \ne t_1,\dotsc,t_r$, then $b(t_k;s)\ne0$ for all $k=1,\dotsc,r$ and we obtain 
$(\alpha_1,\dotsc,\alpha_r)=(0,\dotsc,0)$. 
Hence if $(\alpha_1,\dotsc,\alpha_r)\ne(0,\dotsc,0)$, then 
$$
\sum_{k=1}^r
\alpha_k
b(t_k;s)
\eta_k(s)
\ne
0
\quad
\text{in}
\quad
T^\ast_{\gamma_w(s)}(M^\text{int}), 
$$
and there exists $\zeta_0 \in T_{\gamma_w(s)}(M^\text{int})\setminus\{0\}$ such that 
$$
1
=
\left(
\sum_{k=1}^r
\alpha_k
b(t_k;s)
\eta_k(s)
\right)(\zeta_0)
=
DF\vert^T_{\dot{\gamma}_w(s)}\left(\sum_{k=1}^r\alpha_k\xi_k\right)(0,\zeta_0),
$$
which proves the first half of Lemma~\ref{theorem:propagation2}. 
\par
Finally we suppose that two of $\alpha_1,\dotsc,\alpha_r$ are not zero, say, 
$\alpha_1\ne0$ and $\alpha_2\ne0$. 
We may assume that $s{\ne}t_1$ 
since at least one of $s{\ne}t_1$ and $s{\ne}t_2$ holds. 
Then $\alpha_1b(t_1;s)\ne0$, and 
$$
\sum_{k=1}^r
\alpha_k
b(t_k;s)
\eta_k(s)
\ne
0
\quad
\text{in}
\quad
T^\ast_{\gamma_w(s)}(M^\text{int}).
$$
There exists $\zeta_1 \in T_{\gamma_w(s)}(M^\text{int})\setminus\{0\}$ such that 
$$
1
=
\left(
\sum_{k=1}^r
\alpha_k
b(t_k;s)
\eta_k(s)
\right)(\zeta_1)
=
DF\vert^T_{\dot{\gamma}_w(s)}\left(\sum_{k=1}^r\alpha_k\xi_k\right)(0,\zeta_1),
$$
which proves the second half of Lemma~\ref{theorem:propagation2}. 
\end{proof}
Next we compute the singularities of higher order terms. 
\begin{lemma}
\label{theorem:higherorderterms} 
Let $l\geqq2$ and $\nu:=\operatorname{codim}(\Sigma_{12\dotsc{l}})$. 
Note that $\nu\geqq2$. Suppose that 
$h \in I^{-(n+1)/2}\bigl(N^\ast(\Sigma_{l+1})\setminus0;\Omega_{\partial_-S(M)}^{1/2}\bigr)$, 
and that 
$H \in I^{-(n+1)/2-(\nu-1)}\bigl(N^\ast(\Sigma_{12\dotsc{l}})\setminus0;\Omega_{\partial_-S(M)}^{1/2}\bigr)$ 
except for $\cup_{k=1}^l\Sigma_k\setminus\Sigma_{12\dotsc{l}}$. 
\begin{itemize}
\item
If  $\Sigma_{12\dotsc(l+1)}:=\Sigma_{12\dotsc{l}}\cap\Sigma_{l+1}$ is a transversal intersection, 
then $\operatorname{codim}(\Sigma_{12\dotsc(l+1)})=\nu+1$, and 
$h{\cdot}H \in  I^{-(n+1)/2-\nu}\bigl(N^\ast(\Sigma_{12\dotsc{(l+1)}})\setminus0;\Omega_{\partial_-S(M)}^{1/2}\bigr)$ 
except for $\cup_{k=1}^l\Sigma_k\setminus\Sigma_{12\dotsc{(l+1)}}$. 
\item
If  $\Sigma_{12\dotsc(l+1)}$ is not a transversal intersection, that is, 
$\Sigma_{12\dotsc(l+1)}=\Sigma_{12\dotsc{l}}$. 
Then there exists $\tilde{\Sigma}$ such that 
$\operatorname{codim}(\tilde{\Sigma})=\nu-1$, 
$\Sigma_{12\dotsc{l}}=\tilde{\Sigma}\cap\Sigma_{l+1}$, and 
$h{\cdot}H \in  I^{-(n+1)/2-(\nu-1)}\bigl(N^\ast(\Sigma_{12\dotsc{l}})\setminus0;\Omega_{\partial_-S(M)}^{1/2}\bigr)$ 
except for $\tilde{\Sigma}$. 
\end{itemize}
\end{lemma}
\begin{proof}
We can see the regularities of $h$ and $H$ as follows:
\begin{itemize}
\item 
$h \in I^{-3/2+1/2-(2n-2)/4}\bigl(N^\ast(\Sigma_{l+1})\setminus0;\Omega_{\partial_-S(M)}^{1/2}\bigr)$.
\item
$H \in I^{-3\nu/2+\nu/2-(2n-2)/4}\bigl(N^\ast(\Sigma_{12\dotsc{l}})\setminus0;\Omega_{\partial_-S(M)}^{1/2}\bigr)$ 
except for $\cup_{k=1}^l\Sigma_k\setminus\Sigma_{12\dotsc{l}}$. 
\end{itemize}
\par
Suppose that $\Sigma_{12\dotsc(l+1)}:=\Sigma_{12\dotsc{l}}\cap\Sigma_{l+1}$ is a transversal intersection. 
Then $\operatorname{codim}(\Sigma_{12\dotsc(l+1)})=\nu+1$ holds. 
Lemma~\ref{theorem:6-4} implies that  
$h{\cdot}H \in  I^{-(n+1)/2-\nu}\bigl(N^\ast(\Sigma_{12\dotsc{(l+1)}})\setminus0;\Omega_{\partial_-S(M)}^{1/2}\bigr)$ 
except for  $\cup_{k=1}^l\Sigma_k\setminus\Sigma_{12\dotsc{(l+1)}}$. 
\par
Suppose that $\Sigma_{12\dotsc(l+1)}:=\Sigma_{12\dotsc{l}}\cap\Sigma_{l+1}$ is not a transversal intersection. 
Then there exists $\tilde{\Sigma}$ such that 
$\operatorname{codim}(\tilde{\Sigma})=\nu-1$, 
$\Sigma_{12\dotsc{l}}=\tilde{\Sigma}\cap\Sigma_{l+1}$ locally. 
In what follows we discuss locally.  
Fix arbitrary $p \in \Sigma_{12\dotsc{l}}$, and consider $h{\cdot}H$ near $p$. 
We choose local coordinates 
$(x,y,z) \in \mathbb{R}\times\mathbb{R}^{\nu-1}\times\mathbb{R}^{2n-2-\nu}$ 
near $p$ so that 
$\bigl(x(p),y(p),z(p)\bigr)=(0,0,0)$, 
$\Sigma_{l+1}=\{(0,y,z)\}$, $\tilde{\Sigma}=\{(x,0,z)\}$ and $\Sigma_{12\dotsc{l}}=\{(0,0,z)\}$. 
Then we have 
\begin{align*}
   N^\ast(\Sigma_{l+1})
& =
  \{(0,y,z;\xi,0,0) \in T^\ast(\mathbb{R}\times\mathbb{R}^{\nu-1}\times\mathbb{R}^{2n-2-\nu})\},
\\
   N^\ast(\Sigma_{12\dotsc{l}})
& =
  \{(0,0,z;\xi,\eta,0) \in T^\ast(\mathbb{R}\times\mathbb{R}^{\nu-1}\times\mathbb{R}^{2n-2-\nu})\}
\end{align*}
locally. Hence there exist symbols 
\begin{align*}
  a(y,z,\xi)
& \in 
  S^{-3/2}\bigl((\mathbb{R}^{\nu-1}\times\mathbb{R}^{2n-2-\nu})\times\mathbb{R}\bigr),
\\
  b(z,\xi,\eta)
& \in 
  S^{-3\nu/2}\bigl(\mathbb{R}^{2n-2-\nu}\times(\mathbb{R}\times\mathbb{R}^{\nu-1})\bigr),
\end{align*}
such that $h(x,y,z)$ and $H(x,y,z)$ are given by 
\begin{align*}
   h(x,y,z)
& = 
  \int_{\mathbb{R}}
  e^{ix\xi}
  a(y,z,\xi)
  d\xi,
\\
  H(x,y,z)
& =
  \iint_{\mathbb{R}\times\mathbb{R}^{\nu-1}}
  e^{i(x\xi+y\cdot\eta)}
  b(z,\xi,\eta)
  d\xi
  d\eta
\end{align*}
near $(x,y,z)=(0,0,0)$. 
See \cite[Section~18.2]{Hoermander3} for the characterization of conormal distributions. 
Then we have 
$$
h{\cdot}H(x,y,z)
=
\iint_{\mathbb{R}\times\mathbb{R}^{\nu-1}}
e^{i(x\xi+y\cdot\eta)}
c(y,z,\xi,\eta)
d\xi
d\eta, 
$$
$$
c(y,z,\xi,\eta)
=
\int_{\mathbb{R}}
a(y,z,\xi-\zeta)b(z,\zeta,\eta)
d\zeta. 
$$
\par
Let $\alpha=(\alpha_1,\alpha_2) \in \mathbb{N}_0\times\mathbb{N}_0^{\nu-1}$ and 
$\beta=(\beta_2,\beta_3) \in \mathbb{N}_0^{\nu-1}\times\mathbb{N}_0^{2n-2-\nu}$, 
where $\mathbb{N}_0=\{0,1,2,\dotsc\}$. 
By the direct computation and the integration by parts, we deduce that 
\begin{align}
&  \partial_y^{\beta_2}\partial_z^{\beta_3}\partial_\xi^{\alpha_1}\partial_\eta^{\alpha_2}
  c(y,z,\xi,\eta)
\nonumber
\\
& =
  \sum_{\gamma_3\leqq\beta_3}
  \frac{\beta_3!}{\gamma_3!(\beta_3-\gamma_3)!}
  \int_\mathbb{R}
  \partial_y^{\beta_2}\partial_z^{\gamma_3}\partial_\xi^{\alpha_1}a(y,z,\xi-\zeta)
  \partial_z^{\beta_3-\gamma_3}\partial_\eta^{\alpha_2}b(z,\zeta,\eta)
  d\zeta
\label{equation:wanton1}
\\
& =
  \sum_{\gamma_3\leqq\beta_3}
  \frac{\beta_3!}{\gamma_3!(\beta_3-\gamma_3)!}
  \int_\mathbb{R}
  \partial_y^{\beta_2}\partial_z^{\gamma_3}a(y,z,\xi-\zeta)
  \partial_z^{\beta_3-\gamma_3}\partial_\zeta^{\alpha_1}\partial_\eta^{\alpha_2}b(z,\zeta,\eta)
  d\zeta.
\label{equation:wanton2}
\end{align}
We split our argument into three parts: 
1. $\lvert\xi\rvert\leqq2\lvert\eta\rvert$, 
2. $\lvert\xi\rvert>2\lvert\eta\rvert$ and $\lvert\xi\rvert>2\lvert\zeta\rvert$, 
and 3. $\lvert\xi\rvert>2\lvert\eta\rvert$ and $\lvert\xi\rvert\leqq2\lvert\zeta\rvert$.  
We should introduce a smooth cut-off function for this purpose, 
but we omit the cut-off to simplify the notation.    
\par
Suppose that $\lvert\xi\rvert\leqq2\lvert\eta\rvert$. 
Using \eqref{equation:wanton2}, we deduce that 
\begin{align*}
  \partial_y^{\beta_2}\partial_z^{\beta_3}\partial_\xi^{\alpha_1}\partial_\eta^{\alpha_2}
  c(y,z,\xi,\eta)
& =
  \int_{\mathbb{R}}
  \mathcal{O}
  \bigl(
  (1+\lvert\xi-\zeta\rvert)^{-3/2}
  (1+\lvert\zeta\rvert+\lvert\eta\rvert)^{-3\nu/2-\lvert\alpha\rvert}
  \bigr)
  d\zeta
\\
& =
  \int_{\mathbb{R}}
  \mathcal{O}
  \bigl(
  (1+\lvert\xi-\zeta\rvert)^{-3/2}
  (1+\lvert\eta\rvert)^{-3\nu/2-\lvert\alpha\rvert}
  \bigr)
  d\zeta
\\
& =
  \mathcal{O}
  \bigl(
  (1+\lvert\eta\rvert)^{-3\nu/2-\lvert\alpha\rvert}
  \bigr). 
\end{align*}
Note that $(1+\lvert\eta\rvert)\geqq(1+\lvert\xi\rvert+\lvert\eta\rvert)/3$ 
since $\lvert\xi\rvert\leqq2\lvert\eta\rvert$. Then we obtain 
\begin{equation}
c(y,z,\xi,\eta)
\in 
S^{-3\nu/2}
\bigl(
(\mathbb{R}^{\nu-1}\times\mathbb{R}^{2n-2-\nu})
\times
(\mathbb{R}\times\mathbb{R}^{\nu-1})
\bigr) 
\quad
\text{for}
\quad
\lvert\xi\rvert\leqq2\lvert\eta\rvert
\label{equation:wanton3}
\end{equation}
\par
Suppose that $\lvert\xi\rvert>2\lvert\eta\rvert$ 
and $\lvert\xi\rvert>2\lvert\zeta\rvert$. 
Using \eqref{equation:wanton1}, we deduce that 
\begin{align}
&  \partial_y^{\beta_2}\partial_z^{\beta_3}\partial_\xi^{\alpha_1}\partial_\eta^{\alpha_2}
  c(y,z,\xi,\eta)
\nonumber
\\
& =
  \int_{\mathbb{R}}
  \mathcal{O}
  \bigl(
  (1+\lvert\xi-\zeta\rvert)^{-3/2-\alpha_1}
  (1+\lvert\zeta\rvert+\lvert\eta\rvert)^{-3\nu/2-\lvert\alpha_2\rvert}
  \bigr)
  d\zeta
\nonumber
\\
& =
  (1+\lvert\eta\rvert)^{-3(\nu-1)/2-\lvert\alpha_2\rvert}
  \int_{\mathbb{R}}
  \mathcal{O}
  \bigl(
  (1+\lvert\xi-\zeta\rvert)^{-3/2-\alpha_1}
  (1+\lvert\zeta\rvert)^{-3/2}
  \bigr)
  d\zeta
\nonumber
\\
& \leqq
  C
  (1+\lvert\eta\rvert)^{-3(\nu-1)/2-\lvert\alpha_2\rvert}
  \int_{\mathbb{R}}
  (1+\lvert\xi\rvert)^{-3/2-\alpha_1}
  (1+\lvert\zeta\rvert)^{-3/2}
  d\zeta
\nonumber
\\
& =
  \mathcal{O}
  \bigl(
  (1+\lvert\xi\rvert)^{-3/2-\alpha_1}  
  (1+\lvert\eta\rvert)^{-3(\nu-1)/2-\lvert\alpha_2\rvert}
  \bigr)
\nonumber
\\
& =
  \mathcal{O}
  \bigl(
  (1+\lvert\xi\rvert+\lvert\eta\rvert)^{-3/2-\alpha_1}  
  (1+\lvert\eta\rvert)^{-3(\nu-1)/2-\lvert\alpha_2\rvert}
  \bigr)
  \quad\text{for}\quad
  \lvert\xi\rvert>2\lvert\eta\rvert, 
  \lvert\xi\rvert>2\lvert\zeta\rvert. 
\label{equation:wanton4}  
\end{align}
Suppose that $\lvert\xi\rvert>2\lvert\eta\rvert$ 
and $\lvert\xi\rvert\leqq2\lvert\zeta\rvert$. 
Using \eqref{equation:wanton2}, we deduce that 
\begin{align}
&  \partial_y^{\beta_2}\partial_z^{\beta_3}\partial_\xi^{\alpha_1}\partial_\eta^{\alpha_2}
  c(y,z,\xi,\eta)
\nonumber
\\
& =
  \int_{\mathbb{R}}
  \mathcal{O}
  \bigl(
  (1+\lvert\xi-\zeta\rvert)^{-3/2}
  (1+\lvert\zeta\rvert+\lvert\eta\rvert)^{-3\nu/2-\lvert\alpha\rvert}
  \bigr)
  d\zeta
\nonumber
\\
& \leqq
  C
  (1+\lvert\eta\rvert)^{-3(\nu-1)/2-\lvert\alpha_2\rvert}
  \int_{\mathbb{R}}
  (1+\lvert\xi-\zeta\rvert)^{-3/2}
  (1+\lvert\zeta\rvert)^{-3/2-\alpha_1}
  d\zeta
\nonumber
\\
& \leqq
  C
  (1+\lvert\eta\rvert)^{-3(\nu-1)/2-\lvert\alpha_2\rvert}
  \int_{\mathbb{R}}
  \mathcal{O}
  \bigl(
  (1+\lvert\xi-\zeta\rvert)^{-3/2}
  (1+\lvert\xi\rvert)^{-3/2-\alpha_1}
  \bigr)
  d\zeta
\nonumber
\\
& =
  \mathcal{O}
  \bigl(
  (1+\lvert\xi\rvert)^{-3/2-\alpha_1}  
  (1+\lvert\eta\rvert)^{-3(\nu-1)/2-\lvert\alpha_2\rvert}
  \bigr)
\nonumber
\\
& =
  \mathcal{O}
  \bigl(
  (1+\lvert\xi\rvert+\lvert\eta\rvert)^{-3/2-\alpha_1}  
  (1+\lvert\eta\rvert)^{-3(\nu-1)/2-\lvert\alpha_2\rvert}
  \bigr)
  \quad\text{for}\quad
  \lvert\xi\rvert>2\lvert\eta\rvert, 
  \lvert\xi\rvert\leqq2\lvert\zeta\rvert. 
\label{equation:wanton5}  
\end{align}
Combining \eqref{equation:wanton4} and \eqref{equation:wanton5}, we obtain 
\begin{equation}
c(y,z,\xi,\eta)
\in 
S^{-3/2,-3(\nu-1)/2}
\bigl(
(\mathbb{R}^{\nu-1}\times\mathbb{R}^{2n-2-\nu})
\times
(\mathbb{R}\times\mathbb{R}^{\nu-1})
\bigr) 
\quad
\text{for}
\quad
\lvert\xi\rvert>2\lvert\eta\rvert, 
\label{equation:wanton6}
\end{equation}
that is, for $\lvert\xi\rvert>2\lvert\eta\rvert$
$$
\partial_y^{\beta_2}\partial_z^{\beta_3}\partial_\xi^{\alpha_1}\partial_\eta^{\alpha_2}
c(y,z,\xi,\eta)
=
\mathcal{O}
\bigl(
(1+\lvert\xi\rvert+\lvert\eta\rvert)^{-3/2-\alpha_1}  
(1+\lvert\eta\rvert)^{-3(\nu-1)/2-\lvert\alpha_2\rvert}
\bigr). 
$$
This shows that $c(y,z,\xi,\eta)$ is an amplitude of a paired Lagrangian distribution. 
See \cite{MelroseUhlmann,GuilleminUhlmann,GreenleafUhlmann}. 
\par
Combining \eqref{equation:wanton3} and \eqref{equation:wanton6}, 
we deduce that except for  $\cup_{k=1}^l\Sigma_k\setminus\Sigma_{12\dotsc{l}}$ 
\begin{align*}
   h{\cdot}H 
& \in 
    I^{3\nu/2+\nu/2-(2n-2)/4}\bigl(N^\ast(\Sigma_{12\dotsc{l}})\setminus0\bigr)
\\
& +
   I^{-3/2+1/2-(2n-2)/2,-3(\nu-1)/2+(\nu-1)/2}\bigl(N^\ast(\Sigma_{12\dotsc{l}})\setminus0,N^\ast(\tilde{\Sigma})\setminus0\bigr)
\\
& =
    I^{-(n+1)/2-(\nu-1)}\bigl(N^\ast(\Sigma_{12\dotsc{l}})\setminus0\bigr)
\\
& +
   I^{-(n+1)/2-2,-\nu+1}\bigl(N^\ast(\Sigma_{12\dotsc{l}})\setminus0,N^\ast(\tilde{\Sigma})\setminus0\bigr)
\\
& \subset 
   I^{-(n+1)/2-(\nu-1)}\bigl(N^\ast(\Sigma_{12\dotsc{l}})\setminus0\bigr)
\\
& +
   I^{-(n+1)/2-(\nu-1)}\Bigl(\bigl(N^\ast(\Sigma_{12\dotsc{l}})\setminus0\bigr) \setminus \bigl(N^\ast(\tilde{\Sigma})\setminus0\bigr)\Bigr) 
\\
& +
   I^{-(n+1)/2-2}\bigl(N^\ast(\tilde{\Sigma})\setminus0\bigr)
\\
& \subset
    I^{-(n+1)/2-(\nu-1)}\bigl(N^\ast(\Sigma_{12\dotsc{l}})\setminus0\bigr).
\end{align*}
This completes the proof.  
\end{proof}
In view of Lemmas~\ref{theorem:compositionT} and \ref{theorem:propagation2}, 
we notice that the singularities of the higher order interactions which are composable with $\mathcal{C}_{\chi^T}$, 
are limited to the elements of $N^\ast(\Sigma_{jk}^{(1)})$. 
More precisely, Lemma~\ref{theorem:propagation2} shows that no linear combination of two or more linearly independent covectors in the space of normal geodesics cannot come back to the real world $(M,g)$ under the canonical relation $\mathcal{C}_{\mathcal{X}^T}$. 
Then, in the same way as Lemmas~\ref{theorem:compositionT}, we can deduce that only the parallel transport of common singularities in the space of normal geodesics can be composable with $\mathcal{C}_{\chi^T}$. 
We need to know the order of singularities arising in the higher order interactions, 
and prepare the following.  
\begin{lemma}
\label{theorem:conormalx}
Let $X$ be a smooth manifold, 
let $Y$ be a smooth submanifold of $X$ with 
$\operatorname{codim}(Y)=\rho\geqq1$, 
and $Z$ be a smooth submanifold of $Y$ with 
$\operatorname{codim}(Z)=\rho+\mu>\rho$ in $X$. 
If $u \in I^m(X,N^\ast(Z)\setminus0;\Omega^{1/2}_X)$ 
and $\operatorname{WF}(u)\subset N^\ast(Y)\setminus0$, then 
$u \in I^{m-\mu/2}(X,N^\ast(Y)\setminus0;\Omega^{1/2}_X)$.  
\end{lemma}
\begin{proof}
We prove Lemma~\ref{theorem:conormalx} using the symbolic calculus of conormal distributions. 
Set $N:=\dim(X)$ for short. 
Pick up an arbitrary point $p \in Z$, and choose local coordinates 
$(x,y,z) \in \mathbb{R}^\rho\times\mathbb{R}^\mu\times\mathbb{R}^{N-\rho-\mu}$ near $p$ so that 
$\bigl(x(p),y(p),z(p)\bigr)=(0,0,0)$, and $Y$ and $Z$ are locally expressed as 
$Y=\{(0,y,z)\}$ and $Z=\{(0,0,z)\}$ respectively. 
Since $u \in I^m(X,N^\ast(Z)\setminus0;\Omega^{1/2}_X)$, there exists a symbol 
$a(z,\xi,\eta) \in S^{m_0}\bigl(\mathbb{R}^{N-\rho-\mu}\times(\mathbb{R}^\rho\times\mathbb{R}^\mu)\bigr)$ 
such that $m_0$ satisfies $m=m_0+(\rho+\mu)/2-N/4$ and 
$$
u(x,y,z)
=
\iint_{\mathbb{R}^\rho\times\mathbb{R}^\mu}
e^{ix\cdot\xi+iy\cdot\eta}
a(z,\xi,\eta)
d\xi
d\eta
=
\int_{\mathbb{R}^\rho}
e^{ix\cdot\xi}
b(y,z,\xi)
d\xi,
$$
$$
b(y,z,\xi)
=
\int_{\mathbb{R}^\mu}
e^{iy\cdot\eta}
a(z,\xi,\eta)
d\eta. 
$$
It suffices to show that 
$b(y,z,\xi) \in S^{m_0}\bigl((\mathbb{R}^\mu\times\mathbb{R}^{N-\rho-\mu})\times\mathbb{R}^\rho\bigr)$. 
Since $u(x,y,z)$ is smooth in $(y,z)$, 
we may assume that for any $L>0$ and for any 
$(\alpha,\gamma)\in\mathbb{N}_0^\rho\times\mathbb{N}_0^{N-\rho-\mu}$, 
\begin{align*}
  \partial_\xi^\alpha\partial_z^\gamma a(z,\xi,\eta)
& =
  \mathcal{O}
  \bigl(
  (1+\lvert\xi\rvert+\lvert\eta\rvert)^{m_0-\lvert\alpha\rvert}
  (1+\lvert\eta\rvert)^{-L}
  \bigr)
\\
& =
  \mathcal{O}
  \bigl(
  (1+\lvert\xi\rvert)^{m_0-\lvert\alpha\rvert}
  (1+\lvert\eta\rvert)^{-L}
  \bigr). 
\end{align*}
Then for any $L>0$ and for any 
$(\alpha,\beta,\gamma)\in\mathbb{N}_0^\rho\times\mathbb{N}_0^\mu\times\mathbb{N}_0^{N-\rho-\mu}$, 
we deduce that 
\begin{align*}
  \partial_\xi^\alpha\partial_y^\beta\partial_z^\gamma b(y,z,\xi)
& =
  \int_{\mathbb{R}^\mu}
  e^{iy\cdot\eta}
  (i\eta)^\beta
  \partial_\xi^\alpha \partial_z^\gamma a(z,\xi,\eta)
  d\eta
\\
& =
  \int_{\mathbb{R}^\mu}
  \mathcal{O}
  \bigl(
  (1+\lvert\xi\rvert)^{m_0-\lvert\alpha\rvert}
  (1+\lvert\eta\rvert)^{-L+\lvert\beta\rvert}
  \bigr)
  d\eta
\\
& =
  \mathcal{O}
  \bigl(
  (1+\lvert\xi\rvert)^{m_0-\lvert\alpha\rvert}
  \bigr)
\end{align*}
provided that $L$ satisfies $L-\lvert\beta\rvert>\mu$. 
This shows that 
$b(y,z,\xi) \in S^{m_0}\bigl((\mathbb{R}^\mu\times\mathbb{R}^{N-\rho-\mu})\times\mathbb{R}^\rho\bigr)$, 
and this completes the proof.    
\end{proof}
\begin{proof}[{\bf Proof of Theorem~\ref{theorem:51}}]
Finally we complete the proof of Theorem~\ref{theorem:51}. 
In the same way as the proof of the latter half of Lemma~\ref{theorem:6-6}, 
we obtain 
\begin{align*}
  f_\text{MA}
& :=
  Q\circ\mathcal{X}^T[P_\text{MA}]
  =
  \sum_{l=1}^\infty
  B_l(\alpha\varepsilon)^{2l}
  Q\circ\mathcal{X}^T
  \left[
  \left(\frac{\mathcal{X}[1_D]}{\lvert{d\mu}\rvert^{1/2}}\right)^{2l}
  \lvert{d\mu}\rvert^{1/2}
  \right]
\\
& =
  (\alpha\varepsilon)^2f_\text{MA,1}
  +
  (\alpha\varepsilon)^4f_\text{MA,2},
\\
  f_\text{MA,1}
& =
  -
  \frac{1}{6}
  Q\circ\mathcal{X}^T
  \left[
  \left(\frac{\mathcal{X}[1_D]}{\lvert{d\mu}\rvert^{1/2}}\right)^2
  \lvert{d\mu}\rvert^{1/2}
  \right],
\\
  f_\text{MA,2}
& =
  \sum_{l=2}^\infty
  B_l(\alpha\varepsilon)^{2(l-2)}
  Q\circ\mathcal{X}^T
  \left[
  \left(\frac{\mathcal{X}[1_D]}{\lvert{d\mu}\rvert^{1/2}}\right)^{2l}
  \lvert{d\mu}\rvert^{1/2}
  \right].
\end{align*}
Lemmas~\ref{theorem:6-2} and \ref{theorem:6-6} show that $f_\text{MA,1} \in \mathcal{W}$, 
and the principal symbol of $f_\text{MA,1}$ is elliptic. 
\par
Next we consider the higher order terms in $f_\text{MA,2}$. 
Note that for $l\geqq2$ we have 
\begin{align*}
  \left(
  \frac{\chi[1_D]}{\lvert{d\mu}\rvert^{1/2}}
  \right)^{2l}
  \lvert{d\mu}\rvert^{1/2}
& =
  \sum_{l_1+\dotsc+l_J=2l}
  \frac{(2l)!}{l_1!\dotsb{l_J!}}
  \prod_{j=1}^J
  \left(
  \frac{\chi[1_{D_j}]}{\lvert{d\mu}\rvert^{1/2}}
  \right)^{l_j}
  \cdot
  \lvert{d\mu}\rvert^{1/2}
\end{align*}
Then we have only to consider terms of the form 
$$
\left(
\prod_{k=1}^r
\frac{h_k}{\lvert{d\mu}\rvert^{1/2}}
\right)
\lvert{d\mu}\rvert^{1/2},
\quad
h_k \in I^{-(n+1)/2}\bigl(N^\ast(\Sigma_k)\setminus0\bigr),
\quad
r=1,\dotsc,J.
$$
For this purpose we now introduce notation. 
For $l=1,\dotsc,r-1$ and $1 \leqq j_1 <\dotsb<j_l \leqq r-1$, set
$$
\tilde{\Sigma}_{j_1\dotsc{j_l}}(r)
:=
\Sigma_{j_1\dotsc{j_l}}
\setminus
\bigcup_{\substack{k=1,\dotsc,r \\ k{\ne}j_1,\dotsc,j_l}}
\Sigma_{j_1{\dotsc}j_lk}, 
\quad
\tilde{\Sigma}_{1\dotsc{r}}(r):=\Sigma_{1\dotsc{r}}. 
$$
Using this notation, we have 
$$
\left(
\prod_{k=1}^2
\frac{h_k}{\lvert{d\mu}\rvert^{1/2}}
\right)
\lvert{d\mu}\rvert^{1/2}
\in 
\sum_{k=1,2}
I^{-(n+1)/2}\bigl(N^\ast(\tilde{\Sigma}_k(2))\setminus0\bigr)
+
I^{(n+1)/2-1}\bigl(N^\ast(\tilde{\Sigma}_{12}(2))\setminus0\bigr). 
$$
Multiply the above by $h_3$ and apply 
Lemma~\ref{theorem:higherorderterms} to this. 
If $\Sigma_{123}:=\Sigma_{12}\cap\Sigma_3$ is a transversal intersection, we deduce that 
$$
\left(
\prod_{k=1}^3
\frac{h_k}{\lvert{d\mu}\rvert^{1/2}}
\right)
\lvert{d\mu}\rvert^{1/2}
\in 
\sum_{l=1}^3
\sum_{1 \leqq j_1 < \dotsb <j_l \leqq 3}
I^{-(n+1)/2-(l-1)}
\bigl(N^\ast(\tilde{\Sigma}_{j_1{\dotsc}j_l}(3))\setminus0\bigr). 
$$
If $\Sigma_{123}:=\Sigma_{12}\cap\Sigma_3$ is not a transversal intersection, 
we deduce that $\Sigma_{123}=\Sigma_{12}$ and 
\begin{align*}
  \left(
  \prod_{k=1}^3
  \frac{h_k}{\lvert{d\mu}\rvert^{1/2}}
  \right)
  \lvert{d\mu}\rvert^{1/2}
& \in 
  \sum_{l=1}^2
  \sum_{1 \leqq j_1 < \dotsb <j_l \leqq 3}
  I^{-(n+1)/2-(l-1)}\bigl(N^\ast(\tilde{\Sigma}_{j_1{\dotsc}j_l}(3))\setminus0\bigr)
\\
& \subset
  \sum_{l=1}^3
  \sum_{1 \leqq j_1 < \dotsb <j_l \leqq 3}
  I^{-(n+1)/2-(l-1)}\bigl(N^\ast(\tilde{\Sigma}_{j_1{\dotsc}j_l}(3))\setminus0\bigr). 
\end{align*}
In both cases, the same regularity estimates hold. 
Repeating this argument up to $r$, we deduce that 
$$
\left(
\prod_{k=1}^r
\frac{h_k}{\lvert{d\mu}\rvert^{1/2}}
\right)
\lvert{d\mu}\rvert^{1/2}
\in 
\sum_{l=1}^r
\sum_{1 \leqq j_1 < \dotsb <j_l \leqq 3}
I^{-(n+1)/2-(l-1)}\bigl(N^\ast(\tilde{\Sigma}_{j_1{\dotsc}j_l}(r))\setminus0\bigr). 
$$
In view of Lemma~\ref{theorem:propagation2}, 
we have only to consider the microlocal singularities 
contained in $N^\ast(\Sigma_{jk}^{(1)})$ for $j{\ne}k$. 
The other microlocal singularities cannot be composed with the canonical relation $\mathcal{C}_\chi^T$.
Hence using Lemma~\ref{theorem:conormalx}, we deduce that  
$$
\left(
\prod_{k=1}^r
\frac{h_k}{\lvert{d\mu}\rvert^{1/2}}
\right)
\lvert{d\mu}\rvert^{1/2}
\in 
\mathcal{A} 
\quad
+
\quad 
\text{terms with singularities rejected by $\mathcal{C}_\chi^T$}.
$$
Thus we obtain 
$$
Q\circ\chi^T
\left[
\left(
\prod_{k=1}^r
\frac{h_k}{\lvert{d\mu}\rvert^{1/2}}
\right)
\lvert{d\mu}\rvert^{1/2}
\right]
\in 
\mathcal{W}. 
$$ 
Hence 
$f_\text{MA}=(\alpha\varepsilon)^2\{f_\text{MA,1}+(\alpha\varepsilon)^2f_\text{MA,2}\}\in\mathcal{W}$, 
and in particular, it follows that away from $\partial{D}$,  
$f_\text{MA} \in I^{-1/2-3n/4}\bigl(N^\ast(\mathcal{L})\setminus0;\Omega_{M}^{1/2}\bigr)$. 
This completes the proof.  
\end{proof} 
%
%
%
%
%
%
%%%%%%
%%%%%% bibliography
%%%%%%

%%%%%%
%%%%%% End
%%%%%%
\end{document}